\documentclass[]{amsart}
\usepackage[applemac]{inputenc}
\usepackage{amsmath,amsfonts,amssymb,
graphicx,enumerate,amsthm,pictex, 
upgreek,mathrsfs,stmaryrd, amsbsy}
\usepackage{etex}

\usepackage{hyperref}
\usepackage[usenames]{color}

\usepackage{ulem}
\newcommand{\e}{\mathrm{e}}
\newcommand{\msD}{\mathscr{D}}
\newcommand{\msP}{\mathscr{P}}
\newcommand{\cD}{\mathcal{D}}
\newcommand{\BD}{\mathbb{D}}
\newcommand{\cI}{\mathcal{I}}
\newcommand{\cH}{\mathcal{H}}
\newcommand{\cL}{\mathcal{L}}
\newcommand{\cK}{\mathcal{K}}
\newcommand{\cP}{\mathcal{P}}
\newcommand{\cR}{\mathcal{R}}
\renewcommand{\d}{\,\text{d}}
\newcommand{\BI}{\mathbb{I}}
\newcommand{\BN}{\mathbb{N}}
\newcommand{\BR}{\mathbb{R}}
\newcommand{\BL}{\mathbb{L}}
\newcommand{\BP}{\mathbb{P}}
\newcommand{\BT}{\mathbb{T}}
\newcommand{\BC}{\mathbb{C}}

\newcommand{\im}{{\rm im}}
\newcommand\norm[2]{{\Vert{#1}\Vert_{#2}}}
\newcommand\normto[3]{{\Vert{#1}\Vert_{#2}^{#3}}}
\newcommand{\set}[1]{\{{#1}\}}
\newcommand{\Lpr}[2]{L^{#1}(\rdp,\mua;\Uplambda^{#2})}
\newcommand{\Lprs}[2]{L^{#1}(\mua;\Uplambda^{#2})}
\newcommand{\Lprss}[2]{L^{#1}(\Uplambda^{#2})}
\newcommand{\card}{\#}
\newcommand{\supp}{\mathrm{supp}}

\newcommand{\rmi}{\textrm{(i)}}
\newcommand{\rmii}{\textrm{(ii)}}
\newcommand{\rmiii}{\textrm{(iii)}}

\newcommand{\Bb}{\mathbb}
\newcommand{\mc}{\mathcal}
\newcommand{\rd}{\mathbb{R}^d}
\newcommand{\rdp}{\mathbb{R}^{d}_+}
\newcommand{\rp}{\mathbb{R}_+}

\newcommand{\la}{\mathcal{L}_{\alpha}}

\newcommand{\mua}{\mu_{\alpha}}

\numberwithin{equation}{section}
\newtheorem{thm}{Theorem}[section]
\newtheorem{prop}[thm]{Proposition}
\newtheorem*{thm*}{Theorem}
\newtheorem{cor}[thm]{Corollary}
\newtheorem{lem}[thm]{Lemma}
\theoremstyle{definition}
\newtheorem{definition}[thm]{Definition} 
\newtheorem*{remark*}{Remark}
\newtheorem{remark}[thm]{Remark}
\newtheorem{notation}[thm]{Notation}

\begin{document}
\setcounter{section}{0}
\title[Hodge-Laguerre operator]
{ Riesz Transforms and Spectral Multipliers \\ of the Hodge-Laguerre Operator }

\subjclass[2010]{33C50, 42C10, 42B15, 42A50, 58A14} 

\keywords{Hodge decomposition, Laguerre polynomials, Riesz transforms, spectral multipliers}

\thanks{
Work partially supported by PRIN 2010 ``Real and complex manifolds: 
geometry, topology and harmonic analysis". The first author is a member of the Gruppo Nazionale per l'Analisi Matematica, la Probabilit\`a e le loro Applicazioni (GNAMPA) of the Istituto Nazionale di Alta Matematica (INdAM)
}

\author[G. Mauceri, M. Spinelli]
{Giancarlo Mauceri, Micol Spinelli}

\address{Giancarlo Mauceri: Dipartimento di Matematica\\ 
Universit\`a di Genova\\
via Dodecaneso 35\\ 16146 Genova\\ Italy 
\hfill\break
mauceri@dima.unige.it}

\address{Micol Spinelli: 
Dipartimento di Matematica\\ 
Universit\`a di Genova\\
via Dodecaneso 35\\ 16146 Genova\\ Italy 
\hfill\break
micol.spinelli@gmail.com}

\maketitle

\begin{abstract}
On$\BR^d_+$, endowed with the Laguerre probability measure $
\mu_\alpha$, we define a Hodge-Laguerre operator 
$\BL_\alpha=\delta\delta^*+\delta^*
\delta$ acting on 
differential forms. Here  
$\delta$ is the Laguerre exterior differentiation operator, 
defined as the classical exterior differential, except that the 
partial derivatives $\partial_{x_i}$ are replaced by the \lq\lq Laguerre 
derivatives\rq\rq $\sqrt{x_i}\partial_{x_i}$, and $\delta^*$ is the 
adjoint 
of $\delta$ with respect to inner product on forms defined by the 
Euclidean structure and the Laguerre measure $\mu_\alpha$. We 
prove dimension-free bounds on $L^p$, $1<p<\infty$,  for  
the  Riesz transforms $\delta \BL_\alpha^{-1/2}$ and $\delta^* \BL_
\alpha^{-1/2}$.
As applications we prove the strong Hodge-de Rahm-Kodaira 
decomposition for forms in $L^p$ and deduce existence and 
regularity results for the solutions of the Hodge and de Rham 
equations in $L^p$. We also prove that for suitable functions $m$ 
the operator $m(\BL^\alpha)$ is bounded on $L^p$, $1<p<\infty$.

\end{abstract}
\tableofcontents
\section{Introduction}
For every multiindex $\alpha=(\alpha_1,\ldots,\alpha_d)\in
 (-1,\infty)^d
$ denote by $\mu_
\alpha$  the \emph{Laguerre measure} on $\rdp$, i.e. the probability 
measure with density 
$$
\rho_\alpha(x)= \prod_{i=1}^d{\frac{x_i^{\alpha_i} e^{-x_i}}
{\Gamma(\alpha_i+1)}d x_i}
$$
with respect to the Lebesgue measure on $\rdp$.

The \emph{Laguerre operator} $\cL_\alpha$ is the differential 
operator on $\rdp$ defined by
$$
\cL_\alpha = - \sum_{i=1}^d\left[{x_i \frac{\partial^2}{\partial x_i ^2} + 
(\alpha_i +1 - x_i) \frac{\partial}{\partial x_i}}\right].
$$
The operator $\cL_\alpha$, considered as an operator on the Hilbert 
space $L^2(\rdp,\mua)$ with domain the space of smooth functions 
with compact support  $C^\infty_c(\rdp)$,  is symmetric and 
nonnegative   for all $\alpha\in(-1,\infty)^d$. If $\alpha_i\in (-1,1)$ 
for some $i\in\set{1,2,\ldots,d}$ then it has several self-adjoint 
extensions, depending on the boundary 
condition at $\set{x\in\rdp: x_i=0}$. For $\alpha\in[1,\infty)^d$ it is 
essentially self-adjoint and so it has a unique self-adjoint extension 
whose spectral resolution is given by the orthogonal system  of 
generalized Laguerre polynomials $\set{L^\alpha_k: k\in\BN}$ in $d$ 
variables. The Laguerre orthogonal system  
provides also the spectral 
resolution of an extension of  
$\cL_\alpha$ for $\alpha\notin[1,\infty)^d$ that 
satisfies an appropriate boundary condition at $x_i=0$ whenever $
\alpha_i<1$ \cite{Zettl:MSM}.  
\par
Henceforth, for all $\alpha\in(-1,\infty)^d$ we shall denote by $\cL_
\alpha$ the self-adjoint realization of the Laguerre operator that has 
the Laguerre polynomials as eigenfunctions. 
It is well known that for {$\alpha\in [-1,\infty)^d$}  the 
operator $\cL_\alpha$ generates a symmetric 
diffusion semigroup on 
the measure space $(\rdp,\mua)$, called the \emph{Laguerre  
semigroup}. In his pioneering work \cite{Stein:PUP}, 
E. M. Stein realized that 
many classical results concerning the heat  and the Poisson 
semigroups associated to the Laplacian on the Euclidean setting 
generalize to the more abstract setting of symmetric diffusion 
semigroups. This was the starting point of the so called harmonic 
analysis of semigroups, which centers around the study of the 
boundedness properties on the Lebesgue spaces $L^p$ of various 
operators naturally associated to the semigroup and its generator, 
such as maximal functions, Littlewood-Paley functions, spectral 
multipliers and Riesz transforms.  Some milestones in this abstract 
theory are the works of M. Cowling \cite{Cow:AM}, 
of R. R. Coifman, R. 
Rochberg and G. Weiss \cite{CRW} and, quite recently, 
the paper of A. 
Carbonaro and O. Dragi\v{c}evi\v{c} \cite{CDmult} 
on spectral multipliers for 
symmetric contraction semigroups.
\par
 Even though some of these 
general results are optimal, sharper results can be obtained for 
particular subclasses of symmetric 
contraction semigroups, and there 
is a vast body of  literature concerning specific 
semigroups such as, for 
instance, the semigroups generated by invariant Laplacians or 
sublaplacians on Lie groups, by the Laplace-Beltrami operator on 
Riemannian manifolds, the Ornstein-Uhlenbeck semigroup, both in 
the finite and in the infinite dimensional setting, and the semigroups 
associated to various orthogonal systems of polynomials. 
In the last class falls the Laguerre semigroup, whose 
harmonic analysis has been investigated recently by various people. 
Specifically we mention the work of B. Muckenhoupt in the one-dimensional setting \cite{M1:Tams,M2:TAMS} and, in higher dimension, those of  U. Dinger on the weak-type 
estimate for the maximal function \cite{Dinger:RMI}, of E. Sasso on 
spectral multipliers \cite{Sasso:BAMS, Sasso:MZ} and on the maximal 
operator for the holomorphic semigroup \cite{Sasso:MS} of E. Sasso 
\cite{Sasso:BAMSw} and of L. Forzani, E. Sasso and R. Scotto on the weak 
type inequality for the Riesz-Laguerre transforms \cite{FSS:JFA,FSS:JFAA}, of P. Graczyk \emph{et al.} on higher order Riesz transforms \cite{GLLNU:}.
Particularly relevant for the results of this paper are the papers of A. 
Nowak \cite{Nowak:JFA} and of A. Nowak and K. Stempak \cite{NowakStempak:ILLJM} on Riesz 
transforms and on $L^p$-contractivity of the Laguerre semigroup.
\par
In this paper we investigate the $L^p$-boundedness of the Riesz 
transforms and spectral multipliers for the Hodge-Laguerre operator 
$
\BL_\alpha$, a generalisation to differential forms of the Laguerre 
operator $\cL_\alpha$ on functions. We recall that, if $M$ is a 
Riemannian manifold, the Hodge-de Rham operator on differential 
forms is the operator $\square=dd^*+d^*d$, where $d$ denotes the 
exterior differentiation operator mapping $r$-forms to $(r+1)$-forms 
and $d^*$ is its adjoint with respect to the inner product on forms 
defined by the the Riemannian structure and the Riemannian measure. 
We define the Hodge-Laguerre operator acting on 
differential forms on $\rdp$ as $\BL_\alpha=\delta\delta^*+\delta^*
\delta$, where $\delta$ is the Laguerre exterior differentiation operator, 
defined much as the classical exterior differential, except that the 
partial derivatives $\partial_{x_i}$ are replaced by the \lq\lq Laguerre 
derivatives\rq\rq $\sqrt{x_i}\partial_{x_i}$, and $\delta^*$ is the adjoint 
of $\delta$ with respect to inner product on forms defined by the 
Euclidean structure and the Laguerre measure $\mu_\alpha$ (see 
Section \ref{s: 2.1} for more details).

On manifolds, the Riesz transforms on forms are the 
operators $\cR=d\square^{-1/2}$, mapping $r$-forms to $(r+1)$-forms, 
and its formal adjoint $\cR^*=\square^{-1/2}d^*$, mapping $r$-forms 
to $
(r-1)$-forms. 

In \cite{S:JFA} R. Strichartz proved the on a 
\emph{complete} Riemannian manifold the Hodge operator is 
essentially self-adjoint. He also proved that the Riesz transforms $d
\square^{-1/2}$ and $\square^{-1/2}d^*$ are bounded on $L^2$. It is 
well known that there is a connection between the boundedness of 
Riesz transforms on $L^2$-forms 
and the $L^2$-Hodge-de Rham decomposition of the space $L^2(M,
\Uplambda^r)$ of square-integrable $r$-forms as the direct 
orthogonal sum
$$
L^2(M,\Uplambda^r)=\ker_r(\square)\oplus \overline{d(C^\infty_c(M;
\Uplambda^{r-1}))}\oplus \overline{d^*(C^\infty_c(M;\Uplambda^{r
+1}))}
$$
where 
\begin{itemize}
\item[\rmi] $\ker_r(\square)$ is the kernel of $\square$ in  $L^2(M,
\Uplambda^r)$,
\item[\rmii] $\overline{d(C^\infty_c(M;
\Uplambda^{r-1}))}$ is  the closure in $L^2(M,\Uplambda^r)$ of the image  of $d$ on
the space $C^\infty_c(M;\Uplambda^{r-1})$ of smooth forms with compact support,
\item[\rmiii] $\overline{d^*(C^\infty_c(M;
\Uplambda^{r
+1}))}$ is the closure in $L^2(M,\Uplambda^r)$ of the image of 
$d^*$ on the space $C^\infty_c(M,\Uplambda^{r+1})$,
\end{itemize}
(see \cite{deRham, Gromov:JDG}). Indeed the operators $\cR
\cR^*$ and $\cR^*\cR$ are precisely the orthogonal projections onto 
the spaces $\overline{d(C^\infty_c(M;
\Uplambda^{r-1}))}$ and $\overline{d^*(C^\infty_c(M;\Uplambda^{r
+1}))}$, respectively. To prove the existence of an analogue of the 
Hodge decomposition for $L^p$ forms when $p\not=2$, various 
authors have investigated the boundedness of the Riesz transforms 
on $L^p$, under suitable geometric conditions on the manifold 
\cite{L:MN,LM:Asian, Li:PTRF, Li:RMI, Li:JFA}. 
\par
On $0$-forms the operator $\square$ reduces to $d^*d=\Delta$, the 
Laplace-Beltrami operator on functions, and $d^*=0$. Thus $
\cR^*=0$, and $\cR=d\Delta^{-1/2}$ is the Riesz transform mapping 
functions to $1$-forms, a singular integral whose boundedness on 
$L^p(M)$ has been extensively investigated by many authors \cite{Bakry:LN, ACDH:ASENS,CCH:Duke,CD:JFA}.
\par
In the Laguerre context, the boundedness on $L^p(\rdp,\mua)$, 
$1<p<\infty$, of the scalar Riesz transforms $\delta\cL_\alpha^{-1/2}$, $i=1,\ldots,d$ on 
functions has been proved by A. Nowak when $\alpha\in[-1/2,\infty)^d
$, 
using a Littlewood-Paley-Stein square function \cite{Nowak:JFA}. The 
estimates obtained by Nowak are independent of the dimension.  Quite recently B. Wr\'obel has described a general scheme for deducing dimension free $L^p$ estimates of $d$-dimensional Riesz transforms from the boundedness of one-dimensional Riesz transforms \cite{Wrobel}. By combining his result with the one-dimensional estimate of B.~Muckenhoupt \cite{M2:TAMS}, Wr\'obel obtains dimension independent estimates of the scalar Riesz transforms $\delta_i\cL_\alpha^{-1/2}$, $i=1,\ldots,d$, for all $\alpha\in(-1,\infty)^d$. To the best of our knowledge, no dimension-free estimates are known for the vector valued Riesz transform $\delta\cL_\alpha^{-1/2}$.
\par\medskip
Our first result is that for  {$\alpha\in [-1/2,\infty)^d$} the Riesz 
transforms associated to the Hodge-Laguerre operator, $\delta\BL_
\alpha^{-1/2}$ and $\BL_\alpha^{-1/2}\delta^*$, are bounded from $
\Lpr{p}{r}$ to $\Lpr{p}{r+1}$ and to $\Lpr{p}{r-1}$ respectively, for all 
$p\in(1,\infty)$ and all $r\in\{1,2,\ldots,d\}$. We emphasise the fact 
that our bounds are independent of the dimension $d$ and of the 
multi index $\alpha$.
 When $r=0$ the analogous 
result holds for the Riesz transform $\delta 
\BL_\alpha^{-1/2}=\delta\cL_\alpha^{-1/2}$  provided that the domain 
is restricted to  the 
forms in $\Lpr{p}{0}$ with integral zero. Actually, we shall prove 
the stronger 
result that for every $\rho\le r/2$ the shifted Riesz transforms $
\delta(\BL_\alpha-\rho I)^{-1/2}$ and $(\BL_\alpha-
\rho)^{-1/2}\delta^*
$ are bounded from $\Lpr{p}{r}$ for all $p\in(1,\infty)$, with 
dimension-free bounds. We shall apply this result to obtain the 
following strong form of the 
Hodge-De Rham decomposition in $L^p$ for all $1<p<\infty$ and 
$r=1,\ldots,d$
$$
\Lpr{p}{r}=\d W^{1,p}(\rdp,\mua;\Uplambda^{r-1})\oplus\d W^{1,p}
(\rdp,\mua;\Uplambda^{r+1})
$$
where $W^{1,p}(\rdp,\mua;\Uplambda^{r})$ denotes a $(1,p)$-
Sobolev space of $r$-forms, defined as the domain of the closure in 
$\Lpr{p}{r}$ of the operator that maps a \lq\lq polynomial\rq\rq\  form 
$\omega$ into the pair of forms $(\delta\omega,\delta^*\omega)$ 
(see Section \ref{s: Hodge Lp} for more details).
\par
A second application is to show existence theorems and $L^p$-
estimates for the Hodge-Laguer\-re system and de Rham-Laguerre 
operator. The Hodge-\-Laguer\-re system concerns the solvability in 
$\Lpr{p}{r}$ of the system of equations $\delta \omega=\alpha$ and 
$\delta^* \omega=\beta$ for all $\alpha\in W^{1,p}(\rdp,\mua;
\Uplambda^{r+1})$ and all $\beta\in W^{1,p}(\rdp,\mua;
\Uplambda^{r-1})$ such that $\delta\alpha=0$ and $\delta^*
\beta=0$. The de Rham-Laguerre equation concerns the solvability 
in $\Lpr{p}{r-1}$ of the equation $\delta \omega=\alpha$ for all $
\alpha \in W^{1,p}(\rdp,\mua;\Uplambda^{r})$ such that $\delta 
\alpha=0$. Our results here are the analogues in the Laguerre 
setting of results obtained by X.-D. Li over weighted complete 
Riemannian manifolds under suitable curvature and completeness 
assumptions \cite{Li:JFA}. Note that $\rdp$ is not a complete 
manifold.
\par
We now describe briefly the method used to prove the $L^p$-
boundedness of Riesz transforms. We adapt to our setting 
Carbonaro and Dragi\v{c}evi\v{c}'s proof of Bakry's result on the 
$L^p$-boundedness of the Riesz transform on functions on 
complete Riemannian manifolds whose Ricci curvature is bounded 
from below  \cite{CD:JFA}. They reduce the problem to a bilinear 
estimate involving the covariant derivatives of the Poisson 
semigroups acting on functions and on $1$-forms. To prove the 
bilinear estimate they adapt the technique of Bellman functions, 
introduced in harmonic analysis by F. Nazarov, S. Treil and A. 
Volberg in \cite{NTV:JAMS}. A crucial role in our analysis is played 
by the fact that the Hodge-Laguerre operator $\BL_\alpha$ on forms 
acts diagonally on the coefficients of the form. Namely, if $\omega=
\sum_I \omega_I\d x_I$ is a form in $C^\infty(\rdp,\Uplambda^r)$ 
then
\begin{equation}\label{diagact}
\BL_\alpha \omega=\sum_I \cL_{\alpha,I} \omega_I \d x_I,
\end{equation}
where the $\cL_{\alpha,I}$ are some generalisations of the Laguerre 
operator $\cL_\alpha$. On $1$-forms the operators $\cL_{\alpha,
\set{i}}$ coincide with the operators $M^\alpha_i$ introduced by 
Nowak in \cite{Nowak:JFA} and studied by Nowak and Stempak in 
\cite{NowakStempak:ILLJM} in connection with conjugate Poisson 
integrals.
\par
In the last section we prove a spectral multiplier theorem for the 
Hodge-Laguerre operator.
The Hodge-Laguerre operator on $r$-forms has a self-adjoint 
extension $\BL_\alpha$ on $\Lpr{2}{r}$ with spectral resolution
$$
\BL_\alpha=\sum_{n\ge r} n\,\cP^\alpha_n,
$$
where, for each integer $n\ge r$, $\cP_n$ is the orthogonal 
projection on a finite-dimensional space of \lq\lq polynomial\rq\rq\ 
forms. Thus, by the spectral theorem, if $m=(m_n)_{n\ge r}$ is a 
bounded sequence in $\BC$ the operator
$$
m(\BL_\alpha)=\sum_{n\ge r} m_n\,\cP^\alpha_n
$$
is bounded on $\Lpr{2}{r}$ and $\norm{m(\BL_\alpha)}{2-2}= \sup_{n
\ge r}|m_n|$. \par
We give a sufficient condition for the 
boundedness of $m(\BL_\alpha)$ on \break $L^p(\Uplambda^r(\rdp),
\mua)$ 
also for $p\not=2$. This is a particular instance of the 
\emph{spectral 
multiplier problem} for self-adjoint operators on $L^2(X,\mu)$ where 
$(X,\mu)$ is a $\sigma$-finite measure space, which has been 
investigated 
in a great variety of contexts in the last thirty years. Since the 
literature on the subject is huge, here we only cite a few landmarks 
which are more pertinent to our result. The pioneering work is the 
already cited monograph  \cite{Stein:PUP}, where E. M. Stein 
proved a spectral 
multiplier theorem for generators of symmetric diffusion semigroups. 
Stein's result was subsequently sharpened by M. Cowling in \cite{Cow:AM}, 
who proved that if the operator $-A$ generates a symmetric
 contraction 
semigroup on $(X,\mu)$ and the function $m$ is bounded and 
holomorphic in the sector 
$$
S_{\theta_p}=\set{z\in\BC: |\arg z|<\theta_p}
$$
where $\theta_p = \frac{\pi}{2} \left| \frac{2}{p}-1\right|$, for some $p
\in(1,2)$, then the operator $m(A)$ defined spectrally on $L^2(X,
\mu)
$ extends to a bounded operator on $L^q(X,\mu)$ for all $q$ with 
$p<q<p'$. 

It was known for some time that for some specific generators the 
angle $\theta_p$ is 
not optimal. In particular, if $\cL_{OU}=-\frac{1}{2}\Delta+x\nabla$ is 
the generator of the Ornstein-Uhlenbeck semigroup on  $\BR^n$ 
endowed with the Gauss measure $\gamma$, then
J. Garc\'\i a-Cuerva et al. \cite{GMMST:JFA} (see also 
\cite{ MMS:ASNS}), proved that it suffices to 
assume 
that the spectral multiplier $m$ is bounded and holomorphic in the 
smaller sector $S_{\phi^*_p}$ with $\phi^*_p=  \arcsin \left| 
\frac{2}{p}-1\right|$ (plus some additional differential condition on the 
boundary of the sector) to obtain $L^q$ boundedness of $m(\cL)$ 
on 
all $L^q$, with $p\le q\le p'$.  It is noteworthy to remark that  W. 
Hebisch, G. Mauceri and S. Meda proved that  holomorphy of $m$ 
in 
the sector $S_{\phi^*_p}$ becomes also necessary if we assume 
that 
the multiplier $m$ is \emph{uniform}, i.e. that the all the operators 
$m(t\cL)$, $t>0$, are uniformly bounded on $L^p$ \cite{HMM:JFA}. 
These results are particularly relevant here, 
because the Ornstein-Uhlenbeck operator is strictly related to the 
Laguerre operator on functions via a change of variables, as 
remarked in \cite{GIT}. Indeed, exploiting the relationship 
between $\cL_{OU}$ and $\cL_\alpha$ E. Sasso in \cite{Sasso:MZ} 
proved that 
the results in \cite{GMMST:JFA, HMM:JFA} hold also for the 
Laguerre operator $\cL_\alpha$ on functions. Finally, quite recently,  
A. Carbonaro and O. Dragi\v{c}evi\v{c} \cite{CDmult} proved that  $
\phi^*_p$ is indeed  the optimal angle in a universal multiplier
theorem for generators of symmetric contraction semigroups. 
\par
Since, by \eqref{diagact}
$$
m(\BL_\alpha)\omega=\sum_I m(\cL_{\alpha,I}) \omega_I \d x_I,
$$
via a randomisation argument based on Rademacher's function we 
may reduce the problem to studying spectral multipliers of the 
operators $\cL_{\alpha,I}$. When $\alpha\in[-1/2,\infty)^d$ the 
operators $\cL_{\alpha,I}$,  generate contractions semigroups on 
$L^p(\rdp,\mua)$, $1\le p<\infty$. Therefore Carbonaro and 
Dragi\v{c}evi\v{c}'s result applies to them. 
Exploiting these facts, we prove that if $m$ is a bounded 
holomorphic function in the translated sector $S_{\phi_p^*}+r/2$, 
satisfying suitable Hörmander type conditions on the boundary, then 
the operator $m(\BL_\alpha)$ is bounded on $\Lpr{p}{r}$. We 
emphasise the fact that our estimates depend on the 
dimension $\binom{d}{r}$ of the space of alternating tensors of rank 
$r$ on $\BR^d$.
 \par
Now we describe in some detail how the paper is organised. 
Section \ref{c: Lonf} contains the results on Laguerre 
operators on functions. In Section \ref{s: 2.1} we describe the setup 
and we introduce the operator $\cL_\alpha$, the Laguerre 
derivatives 
$\delta_i$, $i=1,\ldots,d$ and their adjoints $\delta_i^*$. We recall 
the 
spectral resolution of $\cL_\alpha$ and the definition of the Laguerre 
semigroup. In Section \ref{ss: LaI} we define the generalised 
Laguerre operators $\cL_{\alpha,I}$, we give their spectral 
resolutions and we describe the properties of the heat and Poisson 
semigroups generated by them.\par
In Section \ref{c: HLonf} we define the Hodge-Laguerre  operator $
\BL_\alpha$ on forms and prove its basic properties. After a few 
preliminaries on differential forms, in Section \ref{s: HLonf} we define 
first $\delta$, $\delta^*$ and $\BL_\alpha=\delta\delta^*+\delta^*
\delta$ on smooth forms. Then, in Section \ref{s: diag} we prove that 
$\BL_\alpha$ acts diagonally on the coefficients of the form, i.e. 
formula \eqref{diagact}. In Section \ref{s: saeHL} we tackle the 
problem of 
defining $\delta$, $\delta^*$ and $\BL_\alpha$ as closed densely 
defined operators on $\Lpr{2}{r}$. Since for $\alpha\notin [1,\infty)^d
$ the 
operator $\BL_\alpha$ has several self-adjoint extensions, we must 
specify the specific extension we work with. This is done by 
choosing 
an appropriate orthonormal basis $B_r$ of $\Lpr{2}{r}$ consisting of 
eigenfunctions of $\BL_\alpha$. Then the domains of the operators  
$
\delta$, $\delta^*$ and $\BL_\alpha$ are 
described in terms of size conditions on the coefficients of a form 
with respect to the orthonormal basis $B_r$. The map that sends a 
form $\omega\in\Lpr{2}{r}$ to the multisequence of its coefficients 
with respect to the basis $B_r$ can be seen as a map from $\Lpr{2}
{r}$ to a space of square summable multi-sequences of alternating 
covariant tensors of rank $r$ on $\BR^d$, that we call 
\emph{Fourier-Laguerre} transform.  Then we derive simple and 
elegant formulas 
relating the Fourier-Laguerre transform of a form $\omega$ and 
those of the forms $\delta \omega$, $\delta^*\omega$ and $\BL_
\alpha 
\omega$. These formulas are useful in proving the fundamental 
identities $\delta^2=0$, $(\delta^*)^2=0$ and $\BL_\alpha=\delta
\delta^*+\delta\delta^*$ on the appropriate domains of these 
operators. We end this rather long section by proving that for all $
\rho\le r/2$ the \lq\lq heat\rq \rq\  and Poisson semigroups generated 
by the operators $\BL_\alpha-\rho I$ and $(\BL_\alpha-\rho I)^{1/2}$ 
are bounded on $\Lpr{p}{r}$ for all $p\in[1,\infty]$. In Section \ref{s: 
HdRL2} we state the Hodge-de Rham decomposition for the space 
$\Lpr{2}{r}$, whose proof can be obtained using the results of the 
previous section along classical lines.  Then we show that the Riesz-
Laguerre transforms $\cR=\delta \BL_\alpha^{-1/2}$ and $\cR^*= 
\BL_\alpha^{-1/2}\delta^*$ are bounded on $\Lpr{2}{r}$ and that  the 
operators $\cR\cR^*$ and $\cR^*\cR$ are the Hodge-de Rham 
projections onto the spaces of exact and coexact forms respectively.
\par
In Section \ref{s: BETapp} we state the bilinear embedding theorem. 
Deferring its proof, we deduce from it the boundedness on $\Lpr{p}
{r}$ of the Riesz-Laguerre transforms. As applications, we deduce 
the strong Hodge-de Rham decomposition in $\Lpr{p}{r}$ and the 
existence and regularity result for the Hodge system and the de 
Rham equation in $\Lpr{p}{r}$. 
\par
In Section \ref{chbell}, to prepare the proof of the bilinear embedding 
theorem, we recall the definition and the basic properties of the 
particular Bellman function used by A. Carbonaro and O. Dragi\v{c}
evi\v{c} in \cite{CD:JFA} to prove the boundedness of Riesz 
transforms on Riemannian manifolds. Even though the results 
coincide with those in \cite{CD:JFA} we have included full proofs for 
completeness.
\par
In Section \ref{c: proofBET} we prove the bilinear embedding 
theorem, adapting to our situation the arguments in \cite{CD:JFA}.
\par
Finally in Section \ref{c: specmult} we state and prove the spectral 
multiplier theorem for $\BL_\alpha$ when $\alpha\in[-1/2,\infty)^d$.
\par\medskip\noindent
\section{Laguerre operators on functions}\label{c: Lonf}
\subsection{The operator $\cL_\alpha$}\label{s: 2.1}
Let $\rdp$ be the cone $ \left\{ x \in \rd : x_i > 0, \  \forall i=1, \dots, d 
\right\}$.
Given a multi-index $\alpha = (\alpha_1, \dots, \alpha_d), \ \alpha \in 
(-1, \infty)^d$, we define the \emph{Laguerre polynomial of type $
\alpha$ and multidegree $k$ on $\rdp$} as
\[L^{\alpha}_k (x)= L^{\alpha_1}_{k_1}(x_1) \cdots L^{\alpha_d}
_{k_d}(x_d),\]
where $k= (k_1, \dots, k_d)$ has integer components $k_i \geq 0$ 
for each $ i=1, \dots, d$, and 
\[ L^{\alpha_i}_{k_i}(x_i) = \frac{1}{k !} e^{x_i} x_i^{-\alpha_i} 
\frac{\d^k}{\d x_i^k}(e^{-x_i} x_i^{k_i + \alpha_i})\]
is the one-dimensional Laguerre polynomial of type $\alpha_i$ and 
degree $k_i$.

These polynomials are orthogonal with respect to the probability 
measure
\begin{equation}\label{mua}
\d \mua (x) = \rho_\alpha(x)\d x=\prod_{i=1}^d{\frac{x_i^{\alpha_i} 
e^{-x_i}}{\Gamma(\alpha_i+1)}\d x_i},
\end{equation}
which is called the \emph{Laguerre measure on $\rdp$}. 
We denote by 
$$
\ell^{\alpha}_k=L^{\alpha}_k/\norm{L^{\alpha}_k}{2}
$$
their normalizations in $L^2(\rdp,\mua)$. 
The system $\{\ell^{\alpha}_k~:k\in \BN^{d}\}$  is an orthonormal 
basis of $L^2(\rdp,\mu_\alpha)$ of eigenfunctions of the Laguerre 
operator
of type $\alpha$
\[
\la = - \sum_{i=1}^d\left[{x_i \frac{\partial^2}{\partial x_i ^2} + 
(\alpha_i +1 - x_i) \frac{\partial}{\partial x_i}}\right].
\]
Namely 
$$
\cL_\alpha \ell^\alpha_k= |k|\ \ell^\alpha_k, \qquad\forall k\in\BN^d,
$$
where $|k|=k_1+\cdots+k_d$ denotes as usual the length of the 
multiindex $k$. 
\par
The operator $\cL^\alpha$ is nonnegative and symmetric on the 
domain $\mathcal{C}^{\infty}_c(\rdp)$ with respect to the inner 
product in $L^2(\rdp, \mua)$. 
If $\alpha_i\ge 1$ for all $i=1,\ldots,d$ then it is also essentially self-
adjoint on $L^2(\rdp, \mua)$ and the spectral resolution of its 
closure, denoted by $\mathcal{L}_{\alpha}$ as well, is
\begin{equation}\label{f: sar}
\la = \sum_{n=0}^{\infty}{n \ \mathcal{P}_n^{\alpha}},
\end{equation}
where $\cP^{\alpha}_n$ denotes the orthogonal projection on the 
space spanned by Laguerre polynomials of degree $n$. If at least 
one of the indices $\alpha_i$ is such that $\alpha_i <1$ the operator 
$\cL^\alpha$, with domain $C^\infty_c(\rdp)$, has several selfadjoint 
realizations, depending on the boundary conditions at $x_i=0$ 
\cite{Zettl:MSM}.
 \par
  In the following we shall always work with the realization  provided by the spectral resolution \eqref{f: sar} for all $\alpha\in (-1,\infty)^d$. This selfadjoint realisation can also be characterised as the closure of the operator $\cL^\alpha$ on the domain $\cP(\rdp)$ of all finite linear combinations of Laguerre polynomials,  i.e. of  all polynomial functions on $\rdp$. 
\par
The Laguerre operator can be expressed in the more compact way
$$
\cL^\alpha=\sum_{i=1}^d \delta_i^*\delta_i,
$$ 
which makes the symmetry evident, by introducing the \emph{Laguerre partial derivatives}
$$
\delta_i=\sqrt{x_i}\partial_{x_i}
$$
and their formal adjoints in $L^2(\rdp,\mua)$
$$
\delta_i^*=-\sqrt{x_i}\left(\partial_{x_i}+\frac{\alpha_i+\frac{1}{2}-x_i}{x_i}\right)=-\big(\delta_i+\psi_i(x_i)\big),
$$
where
$$
\psi_i(x_i)=\delta_i \log\left(x_i^{\alpha_i+\frac{1}{2}}\,\e^{-x_i}\right)=\frac{\alpha_i+\frac{1}{2}}{\sqrt{x_i}}-\sqrt{x_i}.
$$
The operator $\mathcal{L}_{\alpha}$ is the infinitesimal generator of the Laguerre  semigroup $T^\alpha_t=e^{-t \mathcal{L}_{\alpha}}$, $t\geq 0$ on $L^2(\rdp, \mua)$. 
Since the set $\{\ell^{\alpha}_k\}_{k \in \mathbb{N}^d}$ is an orthonormal basis for $L^2(\rdp, \mua)$, each function $f$ in this space can be expressed as
\[f = \sum_{k \in \mathbb{N}^d}{\hat{f}(k) \ell_k^{\alpha}},\]
where
$$
\hat{f}(k)=\langle f,\ell^\alpha_k\rangle_\alpha
$$
and $\langle\phantom{f},\phantom{f}\rangle_\alpha$ denotes the inner product in $L^2(\rdp,\mua)$.
Hence
\begin{equation}
\label{dec}
T^{\alpha} f = \sum_{k \in \mathbb{N}^d} e^{-t |k|}\ \hat{f}(k)\  \ell^{\alpha}_k.
\end{equation}
The Laguerre semigroup acts as  a semigroup of integral operators
$$
T^{\alpha} f(x)=\int_{\rdp} G^\alpha_t(x,y)\, f(y) \d\mua(y).
$$
where
$$
G^\alpha_t(x,y)=\sum_{k\in\BN^d} \e^{-t|k|}\,\ell^\alpha_{k}(x)\, \ell^\alpha_{k}(y)
$$
is the \emph{Laguerre heat kernel}, which has the following explicit expression in terms of modified Bessel functions of the first kind $I_{\alpha_i}$
\begin{align*}
G^\alpha_t(x,y) = & 
 \prod_{i=1}^d \Gamma(\alpha_i+1) \ (1-u)^{-1}{\exp{ \left( -\frac{u(x_i + y_i) }{1-u} \right)} } \\
 & \times \left( {\sqrt{x_i y_i u}} \right)^{- \alpha_i} I_{\alpha_i} 
\left( 2 \frac{\sqrt{x_i y_i u}}{1-u}\right)
 \end{align*}
 with $u=e^{-t}$ \cite{Nowak:JFA}.
From this expression it is easily seen that $\{T^{\alpha}_t\}$ is a 
symmetric diffusion semigroup, i.e. for all $t\ge 0$
\begin{itemize}
\item[\rmi] $T^{\alpha}_t$ is positivity preserving for all $t\ge0$;
\item[\rmii] $T^{\alpha}_t1=1$;
\item[\rmiii] $\norm{T^{\alpha}_tf}{p}\le \norm{f}{p}$, $1\le p\le 
\infty$.
\end{itemize}
Note that the heat kernel on $\rdp$ is the product of the $d$ one-dimensional heat kernels relative to the operators $\cL^{\alpha_i}=\delta_i^*\delta_i$ acting on $L^2(\BR_+,\mu_{\alpha_i})$\begin{equation}\label{f: prod struc}
G^\alpha_t(x,y)=\prod_{i=1}^d G^{\alpha_i}_t(x_i,y_i).
\end{equation}

\subsection{The operators $\cL_{\alpha,I}$}\label{ss: LaI} 
In this subsection we define some generalizations  of the Laguerre 
operator $\cL_{\alpha}$ that will play an important role in the 
analysis of the Hodge-Laguerre operator on forms.
\begin{definition}\label{d: Lai}
Given a subset $I\subseteq\{1,2,\ldots,d\}$ we define the differential 
operator
$$
\cL_{\alpha,I}=\sum_{i\notin I}\delta_i^*\delta_i+\sum_{i\in I}\delta_i
\delta_i^*.
$$
\begin{remark*}\label{rem605}
When $I=\{i\}$ is a singleton, the operators $\cL_{\alpha,i}$, 
$i=1,\ldots,d$, coincide with the operators $M^\alpha_i$ introduced 
by Nowak in \cite{Nowak:JFA} and studied by Nowak and Stempak 
in \cite{NowakStempak:ILLJM} in connection with conjugate Poisson 
integrals.
\end{remark*}
We observe that $\cL_\alpha=\cL_{\alpha,\emptyset}$. Moreover, 
since
$$
[\delta_j^*,\delta_j] f(x)=\delta_j\psi_j(x_j)\, f(x),
$$
we have that 
\begin{equation}\label{rel LL}
\cL_\alpha=\cL_{\alpha,I} - M_{\alpha,I}
\end{equation}
where 
\begin{equation}\label{MI}
M_{\alpha,I} f(x)=-\left(\sum_{j\in I} \delta_j\psi_j(x_j) \right)\, f(x).
\end{equation}
Notice that {if $\alpha\in[-1/2,\infty)^d$ then}
\begin{equation}\label{e: MI>r/2}
M_{\alpha,I}\,f(x)\ge \frac{\card{I}}{2}\  f(x) \qquad\forall x\in\rdp,
\end{equation}
where $\card{I}$ denotes the cardinality of the set $I$, since 
$$
-\delta_j\psi_j(x)=\frac{1}{2}\left(\frac{\alpha_j+\frac{1}{2}}{x_j}
+1\right)\ge 1/2 \qquad\forall j=1,2,\ldots,d.
$$
\par 
For  each $I\subseteq\{1,2,\ldots,d\}$ we denote by $\cK(I)$ the set 
of all multi-
indexes $k=(k_1,\ldots, k_d)$ in $\BN^d$ such that $k_i\ge 1$ for $i
\in I$. Note that if $k\in \cK(I)$ then $|k|\ge \card{I}$.
For each $I\subseteq\{1,2,\ldots,d\}$  and $k\in \cK(I)$ define
$$
\ell^{\alpha,I}_k(x_{1},\ldots,x_{d})=
\prod_{i\notin I} \ell^{\alpha_{i}}_{k_{i}}(x_{i})\ \prod_{i\in I}
\sqrt{\frac{\alpha_i+1}{k}}\,\delta_{i} \ell^{\alpha_{i}}_{k_{i}}(x_{i}).
$$
\end{definition}
\begin{prop}\label{ONBr1}
The family of functions
$$
B_I=\left\{\ell^{\alpha,I}_k\  : k\in\cK(I)\right\}
$$
is an orthonormal basis of $L^2(\rdp,\mu_\alpha)$ of eigenfunctions 
of the operator $\cL_{\alpha,I}$. Namely
$$
\cL_{\alpha,I} \, \ell^{\alpha,I}_k=|k|\ \ell^{\alpha,I}_k.
$$
\end{prop}
\begin{proof}
In view of the tensor product structure of the 
functions $\ell^{\alpha,I}_k$, it is sufficient to prove that for $
\alpha>-1$  the families  $\{\ell^\alpha_k: k\in\BN\}$ and 
$\{\sqrt{\frac{\alpha+1}{k}}\,\delta \ell^\alpha_k: k\in\BN_+\}$ are 
orthonormal bases of 
$L^2(\BR_+,\mu_\alpha)$. We have already observed in the previous section 
that 
the former family is an orthonormal basis. Hence we only need to 
show that   $\left\{\sqrt{\frac{\alpha+1}{k}}\,\delta \ell^\alpha_k: k\in
\BN_+\right\}$ is an orthonormal basis. Since, by a well known 
property of 
Laguerre polynomials,  $\partial_t L^\alpha_k(t)=- L^{\alpha+1}_{k-1}
(t)$ and $\norm{L^\alpha_k}{2}=\big(\Gamma(\alpha+k+1)/
\Gamma(\alpha+1)\Gamma(k+1)\big)^{1/2}$, it is easily seen that
$$
\sqrt{\frac{\alpha+1}{k}}\delta \ell^\alpha_k(t)=\sqrt{\frac{\alpha+1}{k}}
\sqrt{t}\,\partial_t \,\ell^
\alpha_k(t)=-\sqrt{t}\,\ell^{\alpha+1}_{k-1}(t)
$$ 
Since $\{\ell^{\alpha+1}_j:j\in\BN\}$ 
is an orthonormal basis of $L^2(\BR_+,\mu_{\alpha+1})$, the 
desired 
conclusion follows immediately.\par
The fact that $\ell^{\alpha,I}_k$ is an eigenfunction of $\cL_{\alpha,I}
$ with eigenvalue $|k|$ follows from the identities
$
\delta_j^*\delta_j \ell^{\alpha_j}_{k_j}= k_j \ell^{\alpha_j}_{k_j}$ and $
\delta_j\delta_j^*\delta_j \ell^{\alpha_j}_{k_j}= k_j \delta_j
\ell^{\alpha_j}_{k_j}.$
\end{proof}
With a sligth abuse of notation we shall denote by $\cL_{\alpha,I}$ 
also the self-adjoint extension of $\cL_{\alpha,I}$ with spectral 
resolution
$$
\cL_{\alpha,I}=\sum_{n\ge\card{I}} n\  \cP_{I,n}\,,
$$
where $\cP_{I,n}$ is the orthogonal projection onto the subspace 
spanned by the functions $\ell^{\alpha,I}_k$, $k\in \cK(I)$, $|k|=n$.
\par
We denote by $\{T^{\alpha,I}_t: t\ge 0\}$ the semigroup  generated 
by $- \cL_{\alpha,I}$ and by $\{P^{\alpha,I}_t: t\ge 0\}$ the 
corresponding Poisson semigroup, generated by $-
(\cL_{\alpha,I})^{1/2}$. 
When $I=\emptyset$ we shall simply write $T^\alpha_t$ and 
$P^{\alpha}_t$ instead of $T^{\alpha,\emptyset}_t$ and $P^{\alpha,
\emptyset}_t$. These semigroups have the spectral representations
$$
T^{\alpha,I}_t \,f=\sum_{k\in \cK(I)} \e^{-t|k|} \hat{f}(I,k)\ \ell^{\alpha,I}
_k,
$$
and
$$
P^{\alpha,I}_t\, f=\sum_{k\in \cK(I)} \e^{-t|k|^{1/2}} \hat{f}(I,k)\ 
\ell^{\alpha,I}_k,
$$
where 
$$
\hat{f}(I,k)=\langle f,\ell^{\alpha,I}_k\rangle_\alpha.
$$
The Poisson semigroup can also be defined via the subordination 
principle
$$
P^{\alpha,I}_t f(x)=\frac{1}{\sqrt{\pi}}\int_0^\infty \frac{\e^{-u}}
{\sqrt{u}}\, T^{\alpha,I}_{t^2/4u} \,f(x) \d u.
$$
The semigroup $\{T^{\alpha,I}_t: t\ge0\}$ has also the integral 
representation
$$
T^{\alpha,I}_t f(x)=\int_{\rdp} G^{\alpha,I}_t(x,y)\, f(y) \d\mua(y),
$$
where
\begin{equation}\label{e: kernel}
G^{\alpha,I}_t(x,y)=\prod_{i\notin I} G^{\alpha_i}_t(x_i,y_i)\ \prod_{i\in 
I} \tilde{G}^{\alpha_i}_t(x_i,y_i)
\end{equation}
and 
$$
\tilde{G}^{\alpha_i}_t(x_i,y_i)=\e^{-t}\sqrt{x_iy_i}\  G^{\alpha_i+1}
_t(x_i,y_i)
$$
is the kernel of the semigroup generated by the operator $\delta_j
\delta_j^*$ on $L^2(\BR_+,\mu_{\alpha_j})$ (see \cite{Nowak:JFA}). 
The following lemma has been proved by A.~Nowak.
\begin{lem}\label{l: dom}
There exists a non increasing function $\Lambda$ on $(-1,\infty)$ 
such that $\Lambda(\nu)=1$ for $\nu\ge -1/2$, $\Lambda(\nu)=O
\big((\nu+1)^{-1/2})\big)$ for $\nu \to -1$ and
$$
\tilde{G}^{\nu}_t(x,y)\le \Lambda(\nu)\ \e^{-t/2} \ G^\nu_t(x,y)  \qquad
\forall x,y\in\BR_+, \ t>0.
$$
\end{lem}
\begin{proof}
See \cite[Lemma 2]{Nowak:JFA}.
\end{proof}
\begin{prop}\label{p: dom}
For every $\alpha\in(-1,\infty)^d$ there exists a constant $C(\alpha)$ 
such that $C(\alpha)=1$ if $\alpha_i\ge-1/2$ for all $i\in I$ and
\begin{itemize}
\item[\rmi] $|T^{\alpha,I}_t f(x)|\le C(\alpha)\ \e^{-\card{I}\, t/2}\  T^
\alpha_t |f|(x)$,  
\item[\rmii] $|P^{\alpha,I}_t f(x)|\le C(\alpha)\ 
P^\alpha_t |f|(x)$,
\end{itemize}
for all $f\in L^2(\rdp,\mua)$, $x\in\BR^d$, $t>0$.
\end{prop}
\begin{proof}
Set $C(\alpha)=\prod_{i\in I} \Lambda(\alpha_i)$.Then by \eqref{e: 
kernel} and Lemma \ref{l: dom}, 
$$
G^{\alpha,I}_t(x,y)\le \ C(\alpha) \ \e^{-\card{I}\, t/2}\ G^\alpha_t(x,y).
$$
Thus \rmi\  follows by the positivity of the kernel $G^{\alpha,I}_t(x,y)$  
and \rmii \ follows from the subordination principle.
\end{proof}
\begin{cor}\label{norm on Lp}
For every $\alpha\in(-1,\infty)^d$ there exists a constant $C(\alpha)$ 
such that $C(\alpha)=1$ if $\alpha_i\ge-1/2$ for all $i\in I$ and
$$
\norm{T^{\alpha,I}_t }{p-p}\le C(\alpha)^{(2/p)-1}\exp\left[-t\ \card{I}
\left(1-\left|\frac{1}{2}-\frac{1}{p}\right|\right)\right] \qquad\forall p
\in[1,\infty].
$$
In particular, if $\alpha_i\ge-1/2$ for all $i\in I$, then $\cL_{\alpha,I}-
(\card{I}/2)I$ generates  a symmetric semigroup of  contractions on 
$L^p(\rdp,\mua)$ for every $p\in[1,\infty]$.
\end{cor}
\begin{proof}
The estimate of $\norm{T^{\alpha,I}_t}{1-1}$ follows from 
Proposition \ref{p: dom}, since $T^\alpha_t$ is a contraction on 
$L^1(\rdp,\mua)$. The estimate  of $\norm{T^{\alpha,I}_t}{2-2}$ 
follows from the spectral resolution of $\cL_{\alpha,I}$. The general 
case follows by interpolation and duality.
\end{proof}
\section{The Hodge-Laguerre operator}\label{c: HLonf}
In this section we define the Hodge-Laguerre operator on differential 
forms on $\rdp$ and prove his basic properties. In the first two 
subsections we recall briefly the definition of differential forms, and 
the basic algebraic operations on them that we shall need in the 
sequel: the exterior and the interior products, the Hodge-star 
operator and their properties. These results are classical and we 
refer the reader to the monograph \cite{Warner} of F. W. Warner for 
complete proofs. The main purpose of this preliminary section is to 
establish notation and terminology. In the next subsection \ref{s: 
HLonf} we define the Laguerre exterior differential $\delta$ and its 
formal adjoint $\delta^*$ with respect to the Laguerre measure $
\mu_\alpha$.
\par
\subsection{Differential forms on $\rdp$}
For each $r\in\set{0,1,\ldots,d}$ we denote by $\Uplambda^r=
\Uplambda^r(\BR^d)$ the space of real alternating tensors of rank $r
$ on $\BR^d$.  For every $r$ we denote by $\cI_r$ the set of all 
multiindeces $(i_1,i_2,\ldots,i_r)$ such that $i_1<i_2<\cdots<i_r$. 
The space $\Uplambda^r$ is endowed with the inner product 
$\langle \omega,\eta\rangle_{\Uplambda^r}$ and the corresponding 
norm $|\omega|_{\Uplambda^r}$ for which the set of covectors 
$$
\d x_I=\d x_{i_1}\wedge\d x_{i_2}\wedge\cdots\wedge \d x_{i_r}, 
\qquad I\in \cI_r
$$
is an orthonormal basis. Often we shall simply denote by $\langle 
\omega,\eta\rangle$ and $|\omega|$ the inner product and the norm, 
omitting the subscript $\Uplambda^r$ when there is no risk of 
confusion.
Thus  if $\omega=\sum_{I\in\cI_r}\omega_I \d x_I$ and $\eta=
\sum_{I\in\cI_r}\eta_I \d x_I$ are two elements of $\Uplambda^r$ 
their inner product is $\langle\omega,\eta\rangle=\sum_{I\in\cI_r} 
\omega_I\eta_I $.
We shall denote by $*$ the Hodge-star isomorphism of the exterior 
algebra, mapping  $\Uplambda^r$ to $\Uplambda^{d-r}$ for each $r
$. Then, if we denote by $\mathbb{I}$ the volume form $\d 
x_1\wedge\cdots\wedge \d x_d$,
$$
\omega\wedge *\eta=\langle \omega,\eta\rangle \  \mathbb{I} \qquad
\forall \omega, \eta\in \Uplambda^r.
$$
\par
If $\omega\in \Uplambda^r$ we denote by $\iota_\omega:
\Uplambda^{s+r}\to \Uplambda^s$ the operator of interior 
multiplication by $\omega$, i.e. the adjoint of exterior multiplication 
by $\omega$ with respect to the inner product on $\Uplambda^r$.
\begin{lem}\label{li+il}
If $\phi\in \Uplambda^1(\BR^d)$ and $\omega\in 
\Uplambda^r(\BR^d)$, $0\le r\le d$, then
\begin{align}
\phi\wedge\iota_\phi(\omega)+\iota_\phi(\phi\wedge \omega)&=|\phi|
^2\ \omega \label{f: li+il}
\\ 
|\phi\wedge \omega|^2+|\iota_\phi\omega|^2&=|\phi|^2\ |\omega|
^2.\label{f: norm li+il}
\end{align}
\end{lem}
\begin{proof} 
If $I=(i_{1},\ldots,i_{r})\in \mathcal{I}_{r}$ and $1\le j\le d$, we denote 
by $\sigma(j,I)$ the number of components of $I$ which are strictly 
less than $j$.\\
If $j\notin I$ we denote by $I \cup j$ the element of $\mathcal{I}_{r
+1}$ obtained by adding $j$ to the components of $I$;  if $j\in I$ we 
denote by $I\setminus j$ the  element of $\mathcal{I}_{r-1}$ 
obtained by deleting $j$ from $I$. Thus
\begin{align*}
\d x_j\wedge \d x_I&=(-1)^{\sigma(j,I)} \d x_{I\cup j} \\ 
\iota_{\d x_j}(\d x_I)&=(-1)^{\sigma(j,I)} \d x_{I\setminus j}. \\ 
\end{align*}
Let $\phi=\sum_{j=1}^d \phi_j \d x_j$ and $\omega=\sum_{I\in } 
\omega_I\d x_I$.
Then, on the one hand,
\begin{align*}
\phi\wedge\iota_\phi(\omega)&=\phi\wedge\sum_{I\in\cI_r} \sum_{j\in 
I} \phi_j \omega_I \iota_{\d x_j}(\d x_I) \\ 
&=\phi\wedge \sum_{I\in\cI_r} \sum_{j\in I}(-1)^{\sigma(j,I)}\, \phi_j 
\omega_I \,\d x_{I\setminus j} \\ 
&=  \sum_{I \in \mathcal{I}_r} \sum_{j \in I} \sum_{i \notin \left\{ I 
\setminus j\right\}}(-1)^{\sigma(j, I)} \phi_i \phi_j \omega_I \ dx_i 
\wedge  dx_{I \setminus j} \\
 & =  \sum_{I \in \mathcal{I}_r} \sum_{j \in I} \sum_{i \notin \left\{ I 
\setminus j\right\}}(-1)^{\sigma(j, I) + \sigma(i, I \setminus j)} \phi_i 
\phi_j \omega_I \ dx_{(I \setminus j) \cup i}. \\
\end{align*}
On the other hand
\begin{align*}
\iota_\phi\left(\phi\wedge\, \omega\right)&=\iota_\phi\left(\sum_{I\in
\cI_r} \sum_{i\notin I} \phi_i \omega_I \d x_i\wedge\d x_I \right)\\ 
&=\iota_\phi\left(\sum_{I\in\cI_r} \sum_{i\notin I}(-1)^{\sigma(i,I)}\, 
\phi_i\omega_I \,\d x_{I\cup i}\right) \\ 
&=  \sum_{I \in \mathcal{I}_r} \sum_{i \notin I} \sum_{j \in \left\{ I \cup 
i\right\}}(-1)^{\sigma(i, I)} \phi_i \phi_j \omega_I\  \iota_{\!\ dx_j}(dx_{I 
\cup i}) \\
 & =  \sum_{I \in \mathcal{I}_r} \sum_{i \notin I} \sum_{j \in \left\{ I 
\cup i\right\}}(-1)^{\sigma(i, I) + \sigma(j, I \cup i)} \phi_j \phi_i
\omega_I \ dx_{(I \cup i) \setminus j}. \\
\end{align*}
Next we observe that in the sum $\phi\wedge\iota_\phi(\omega)+
\iota_\phi(\phi\wedge \omega)$ the terms containing indices $i\not=j
$ cancel out. Indeed,
\begin{itemize}
\item[(a)] $
\{(i,j): i\not=j, j\in I, i\not\in I\setminus j\}=\{(i,j): i\not=j, i\notin I, j\in I
\cup i\}
$
\item[(b)]  if $i\not=j$ the exponents $\sigma(j,I)+\sigma(i,I\setminus 
j)$ and $\sigma(i,I)+\sigma(j,I\cup i)$ have opposite parity, as it can 
be easily seen by observing that if $i<j$ then 
$$
\sigma(i, I\setminus j)=\sigma(i,I) \quad {\rm and}\quad \sigma(j,I\cup 
i)=\sigma(j,I)+1,
$$
while, if $j<i$ then
$$
\sigma(j, I\cup I)=\sigma(j,I) \quad {\rm and}\quad \sigma(i,I)=
\sigma(i,I\setminus j)+1.
$$ 
\end{itemize}
Therefore, only the terms with $i=j$ remain and, since for $i=j$ the 
exponents of $-1$ are even, we have that 
\begin{align*}
\phi\wedge\iota_\phi(\omega)+\iota_\phi(\phi\wedge \omega)&= 
\sum_{I\in\cI_r} \sum_{i\in I} \phi_i\phi_i \ \omega_I\d x_I+\sum_{I\in
\cI_r} \sum_{i\notin I} \phi_i\phi_i \ \omega_I\d x_I \\
&=|\phi|^2\ \omega.
\end{align*}
This proves (\ref{f: li+il}). To prove (\ref{f: norm li+il}) observe that
\begin{align*}
|\phi\wedge\omega|^2+|\iota_\phi(\phi\wedge \omega)|^2&= \langle 
\phi\wedge \omega,\phi\wedge \omega\rangle+\langle\iota_\phi
\omega,\iota_\phi\omega\rangle \\ 
&=\langle \omega,\iota_\phi(\phi\wedge \omega)\rangle+\langle 
\omega,\phi\wedge \iota_\phi\omega\rangle \\ 
&=|\phi|^2\ |\omega|^2.
\end{align*}
\end{proof}
We shall denote by $C^\infty(\rdp;\Uplambda^r)$ the space of 
differential forms of order $r$ on $\rdp$ with smooth coefficients and 
by $C^\infty_c(\rdp;\Uplambda^r)$ those with compact support. For 
every $p\in[1,\infty]$  we denote by $L^p(\rdp,\mua;\Uplambda^r)$ 
the space of forms of order $r$ with coefficients in $L^p(\rdp,\mua)$,
endowed with the norm
$$
\norm{\omega}{p}=\left(\int_{\rdp} |\omega(x)|^p\d \mua(x)\right)^{1/
p}
$$
with the usual modification when $p=\infty$. If $\omega$, $\eta$ are 
in $L^2(\rdp,\mua; \Uplambda^r)$, we denote by 
$$
\langle \omega,\eta\rangle_\alpha=\int_{\rdp}\langle\omega(x),\eta(x)
\rangle_{\Uplambda^r}\d\mua(x)
$$
their inner product in $L^2(\rdp,\mua; \Uplambda^r)$. To simplify 
notation, sometimes we shall write simply $\Lprs{p}{r}$ or $\Lprss{p}
{r}$ instead of $\Lpr{p}{r}$.
\subsection{The Hodge-Laguerre operator on forms}\label{s: HLonf}
In this subsection we define the Hodge-Laguerre operator acting on 
smooth forms  in $\rdp$ as a natural generalisation of the Laguerre 
operator on functions. We begin by defining the Laguerre exterior 
derivative operator $\delta$ and its formal adjoint, the Laguerre 
codifferential $\delta^*$.
\par
If $\omega=\sum_{I\in \cI_r} \omega_I \d x_I$ is an $r$-form in $C^
\infty(\rdp;\Uplambda^r)$, its Laguerre exterior differential is the $(r
+1)$-form
$$
\delta\omega=\sum_{j=1}^d\sum_{I\in \cI_r} \delta_j\omega_I \d x_j
\wedge\d x_I,
$$
where $\delta_j$ denotes the differential operator $\sqrt{x_j}\partial_j
$. Using the trivial fact that the partial derivatives $\delta_i$ and $
\delta_j$ commute for $i\not=j$, it is easy to see that $\delta^2=0$. 
Furthermore $\delta$ is an antiderivation, i.e.
$$
\delta(\omega\wedge\eta)=\delta\omega\wedge \eta+(-1)^r \omega
\wedge \delta\eta
$$
for all $r$-forms $\omega$ and $s$-forms $\eta$. \par
The Laguerre codifferential $\delta^*$ is the formal adjoint of $\delta
$ with respect to the inner product in $L^2(\rdp,\mua;\Uplambda^r)$. 
In other words, if $\omega$ is a form in $C^\infty(\Uplambda^{r+1}
(\rdp))$ then  $\delta^*\omega$ is the $r$-form defined by the 
identity 
$$
\langle \delta^* \omega,\eta\rangle_\alpha=\langle \omega, \delta\eta
\rangle_\alpha \qquad\forall \eta\in C_c^\infty(\rdp;\Uplambda^r).
$$ 
We define $\delta^*$ also on $0$-forms by setting $\delta^* 
\omega=0$ for each smooth $0$-form $\omega$.
\par
It follows immediately from the definition of $\delta^*$ that $
(\delta^*)^2=0$.
We give two more explicit expressions of $\delta^*$. To this purpose 
we introduce the $1$-form
$$
\psi(x)=\sum_{j=1}^d\psi_j(x_j)\d x_j,
$$
where
$$
\psi_j(x_j)=\frac{\alpha_j+1/2}{\sqrt{x_j}}-\sqrt{x_j}=\frac{1}{\rho_
\alpha(x)}\partial_j\big(\sqrt{x_j}\rho_\alpha(x)\big),
$$
where $\rho_\alpha$ denotes the Laguerre density (see 
\eqref{mua}).
The following two propositions give two representations of the action 
of $\delta^*$ on $r$-forms.
\begin{prop}\label{first}
On $C_c^\infty(\rdp;\Uplambda^r)$
$$
\delta^*=(-1)^{d(r-1)+1}\ *_{d-r+1}\,\delta\, *_r -\iota_\psi,
$$
where $*_r$ and $*_{d-r+1}$ denote the restrictions of the Hodge $*
$-operator to $r$-forms and to $(d-r+1)$-forms, respectively, and $
\iota_\psi$ is the interior multiplication by the form $\psi$.\par
\end{prop}
\par
\begin{prop}\label{2nd}
If $\omega\in C_c^\infty(\rdp;\Uplambda^r)$ then
$$
\delta^* \omega=\sum_{I\in\cI_r}\sum_{j=1}^d \delta_j^* \omega_I \ 
\iota_{\!\d x_j}(\!\d x_I),
$$
where $\delta_j^*=-\big(\sqrt{x_j}\partial_j+\psi_j(x_j)\big)$.
\end{prop}
\noindent
 To prove Proposition \ref{first}, it is convenient to state a lemma.
\begin{lem}\label{first lemma}
Denote by  $\BI$  the volume element $\d x_1\wedge\ldots\wedge\d 
x_d$ on $\rdp$. Then
$$
\delta^* \BI=-*_1\psi.
$$
\end{lem}
\begin{proof}
Let $\omega = \sum_{i=1}^d \omega_i \ \widehat{\d x_i}$ be a form 
in $C^\infty_c(\rdp;\Uplambda^{d-1})$, where 
$$
\widehat{\d x_i} = (-1)^{i-1} \ast_1 \d x_i=\d x_1 \wedge \ldots 
\wedge \sout{ \d x_i} \wedge \ldots \wedge \d x_d
$$ 
with the element $\d x_i$ omitted. Since
\[ dx_j \wedge \widehat{dx_i} =
\begin{cases}
0 \qquad \mbox{ if } j \neq i\\
(-1)^{i-1} \mathbb{I} \quad \mbox{ if } j = i,
\end{cases}
\]
we have that
$$
\delta \omega=\sum_{i=1}^d \sum_{j=1}^d \delta_j \omega_i \ dx_j 
\wedge \widehat{dx_i}= \sum_{i=1}^d (-1)^{i-1}\delta_i \omega_i \ 
\mathbb{I}.
$$
Hence
\begin{align*}
\langle \delta^* \mathbb{I} , \omega \rangle_\alpha &= \langle 
\mathbb{I}, \delta \omega\rangle_\alpha \\ 
&=\sum_{i=1}^d (-1)^{i-1} \int_{\rdp}{\delta_i \omega_i \ \d \mua}  \\ 
&=-\sum_{i=1}^d (-1)^{i-1} \int_{\rdp} \omega_i \ \partial_i
\big(\sqrt{x_i}\rho_\alpha(x)\big)\d x \\
&=-\sum_{i=1}^d (-1)^{i-1} \int_{\rdp} \omega_i \ \psi_i \d\mu_\alpha \
\
&=-\langle *_1\psi,\omega \rangle_\alpha.
\end{align*} 
Since the choice of $\omega$ is arbitrary, it follows that
\[ \delta^* \mathbb{I} = - \ast_1\psi.\]
\end{proof}
\noindent 
{\it Proof of Proposition \ref{first}}. Suppose that $\omega\in C_c^
\infty(\rdp;\Uplambda^{r-1})$ and $\eta\in C^\infty(\rdp;\Uplambda^r)
$. Then
\begin{equation}\label{-1}
\langle\omega,\delta^*\eta\rangle_\alpha=\langle\delta\omega,\eta
\rangle_\alpha=\int_{\rdp} \delta\omega\wedge *_r\eta\ \rho_\alpha.
\end{equation}
By the antiderivation property of $\delta$
$$
\delta(\omega \wedge \ast_r \eta) = (\delta \omega) \wedge \ast_r 
\eta + (-1)^{r-1} \omega \wedge \delta(\ast_r \eta).
$$
Thus
\begin{equation}\label{0}
\int_{\rdp} \delta\omega\wedge *_r\eta\ \rho_\alpha=\int_{\rdp}
\delta(\omega\wedge*_r\eta)\rho_\alpha - (-1)^{r-1} \int_{\rdp} 
\omega\wedge\delta(*_r \eta)\ \rho_\alpha.
\end{equation}
We evaluate separately the two integrals. On the one hand
\begin{align*}
\int_{\rdp}\delta(\omega\wedge*_r\eta)\rho_\alpha&=\langle 
\delta(\omega\wedge*_r\eta),\BI\rangle_\alpha \\ 
&= \langle \omega\wedge*_r\eta,\delta^*\BI\rangle_\alpha \\ 
&=- \langle \omega\wedge*_r\eta,*_1\psi\rangle_\alpha\\
&= -\int_{\rdp}\omega\wedge *_r\eta*_{d-1}*_1\psi\ \rho_\alpha \\
&=-(-1)^{d-1} \int_{\rdp} \omega\wedge*_r\eta\wedge\psi\ \rho_
\alpha\\
&=-(-1)^{2(d-1)}\int_{\rdp}\psi\wedge\omega\wedge*_r\eta\ \rho_
\alpha\\
&=-\langle \psi\wedge\omega,\eta\rangle_\alpha\\
&=-\langle \omega,\iota_\psi \eta\rangle_\alpha.
\end{align*}
Here we have used Lemma \ref{first lemma} in the third equality, the 
fact that $*_{d-1}*_1=(-1)^{d-1}{\rm id}$ in the fifth, and the 
anticommutativity of the wedge product in the sixth. Thus
\begin{equation}\label{first eq}
\int_{\rdp}\delta(\omega\wedge*_r\eta)\rho_\alpha=-\langle \omega,
\iota_\psi \eta\rangle_\alpha.
\end{equation}
On the other hand, since $\delta*_r\eta=(-1)^{(r-1)(d-r+1)}*_{r-1}
*_{d-r+1}\delta*_r\eta$,
\begin{equation}\label{sec id}
(-1)^{r-1}\int_{\rdp} \omega\wedge \delta(*_r \eta)\ \rho_\alpha=
 (-1)^{d(r-1)} \langle \omega, *_{d-r+1}*\delta *_r \eta\rangle_\alpha.
\end{equation}
By combining identities (\ref{-1})-(\ref{sec id}) we obtain
$$
\langle\omega,\delta^*\eta\rangle_\alpha=\langle \omega, ((-1)^{d(r
+1)+1} \ast_{r-1} \delta \ast_r - \iota_{\psi}) \eta \rangle_\alpha,
$$
which is the desired conclusion. 
\qed\par
\noindent
{\it Proof of Proposition \ref{2nd}.} 
Let $\eta= \sum_{J} \eta_J \d x_J$ be a form in $C^\infty(\rdp;
\Uplambda^{r-1})$.Then
\begin{align}\label{chain}
 \langle \delta^* \omega, \eta \rangle_\alpha = \langle \omega, \delta 
\eta \rangle_\alpha
& = \sum_{I} \sum_{J} \sum_{j =1}^d {\langle  \omega_I \d x_I,  
\delta_j\eta_J \d x_j \wedge dx_J\rangle_\alpha} \nonumber\\
&= \sum_{I} \sum_{J} \sum_{j =1}^d \int_{\rdp}\omega_I\,\delta_j
\eta_J\ \d x_I\wedge *_r(\d x_j\wedge \d x_J)\ \rho_\alpha \\
&= \sum_{I} \sum_{J} \sum_{j =1}^d \int_{\rdp}\omega_I\,\delta_j
\eta_J\ \iota_{\d x_j}(\d x_I)\wedge *_{r-1}\d x_J\ \rho_\alpha
\nonumber \\
&= \sum_{I} \sum_{J} \sum_{j =1}^d C(I,J,j) \int_{\rdp}\omega_I\,
\delta_j\eta_J\ \d \mu_\alpha.\nonumber
\end{align}
Here we have used the fact that  $\d x_I\wedge *_r(\d x_j\wedge \d 
x_J) \ \rho_\alpha=C(I,J,j)\  \d \mu_\alpha$, because $\d x_I\wedge 
*_r(\d x_j\wedge \d x_J)$ is a constant multiple of the volume form, 
with a coefficient $C(I,J,j)\in\{-1,0,1\}$. Now, integrating by parts, we 
obtain that 
$$
\int_{\rdp} \omega_I\ \delta_j \eta_J\d \mu_\alpha=\int_{\rdp} 
\delta_j^*\omega_I\  \eta_J\d \mu_\alpha.
$$
Thus, tracing back the chain of identities (\ref{chain}) after this 
integration by parts,  we obtain that
$$
\langle \delta^* \omega, \eta \rangle_\alpha = \langle \sum_I 
\delta_j^* \omega_I \,\iota_{\d x_j}\!(\d x_I),\eta\rangle_\alpha,
$$
which is desired conclusion.
\qed
\par \medskip
The \emph{Hodge-Laguerre operator} $\BL_\alpha$ is defined by
$$
\BL_\alpha=\delta\delta^*+\delta^*\delta,
$$
and is a linear operator on $C_c^\infty(\rdp;\Uplambda^r)$ for each 
$r$ with $0\le r\le d$. 
\par
By using the expression of $\delta^*$ given in Proposition \ref{2nd} it 
is easy to check that on $0$-forms the Hodge-Laguerre operator $
\BL_\alpha$ coincides with the Laguerre operator $\cL_\alpha$ on 
functions defined in Section \ref{s: 2.1}. Indeed one has
\begin{lem}\label{HL on 0-forms}
If $f\in C^\infty_c(\rdp;\Uplambda^0)=C^\infty_c(\rdp)$ then
$$
\BL_\alpha = -\sum_{i=1}^d (\delta_i+\psi_i)\delta_i f.
$$
\end{lem}
\begin{proof}
Since $\delta^* f=0$, by Proposition \ref{2nd}
$$
\BL_\alpha f= \delta^*\delta f=\sum_{i=1}^d\sum_{j=1}^d \delta_i^*
\delta_j f \ \iota_{\d x_i}\d x_j=\sum_{i=1}^d \delta_i^*\delta_i f= -
\sum_{i=1}^d (\delta_i+\psi_i)\delta_i f.
$$
\end{proof}

\subsection{The diagonalization of the Hodge-Laguerre operator}
\label{s: diag}
In this subsection we prove that the action of the Hodge-Laguerre 
operator on $r$-forms can be diagonalised with respect to the basis 
$\{\d x_I: I\in\cI_r\}$ of $\Uplambda^r$. Namely
\begin{prop}\label{p: diag}
If $\omega=\sum_{I\in\cI_r}\omega_I\d x_I\in C^\infty(\rdp;
\Uplambda^r)$ then
$$
\BL_\alpha\omega=\sum_{I\in\cI_r} \cL_{\alpha,I}\,\omega_I \d x_I,
$$
where 
$$
\cL_{\alpha,I}=\sum_{j\in I}\delta_j\delta_j^*+\sum_{j\notin I}\delta_j^*
\delta_j
$$
are the differential operators acting on scalar functions defined in Section~\ref{ss: LaI}.
\end{prop}
\begin{proof}
If $\omega \in C^\infty(\rdp;\Uplambda^r)$, then
$$
\BL_\alpha\omega=\delta\delta^* \omega+\delta^*\delta \omega.
$$
We compute separately the two summands. As before, we denote by $\sigma(j,I)$ the number of components of $I$ which are strictly less than $j$. On the one hand
\begin{align*}
\delta\delta^* \omega&=\delta \sum_{I\in\cI_r} \sum_{j\in I} \delta_j^* \omega_I \iota_{\d x_j}(\d x_I) \\ 
&=\delta \sum_{I\in\cI_r} \sum_{j\in I}(-1)^{\sigma(j,I)}\, \delta_j^* \omega_I \,\d x_{I\setminus j} \\ 
&=  \sum_{I \in \mathcal{I}_r} \sum_{j \in I} \sum_{i \notin \left\{ I \setminus j\right\}}(-1)^{\sigma(j, I)} \delta_i \delta^*_j \omega_I \ dx_i \wedge  dx_{I \setminus j} \\
 & =  \sum_{I \in \mathcal{I}_r} \sum_{j \in I} \sum_{i \notin \left\{ I \setminus j\right\}}(-1)^{\sigma(j, I) + \sigma(i, I \setminus j)} \delta_i \delta^*_j \omega_I \ dx_{(I \setminus j) \cup i}. \\
\end{align*}
On the other hand
\begin{align*}
\delta^*\delta\, \omega&=\delta^* \sum_{I\in\cI_r} \sum_{i\notin I} \delta_i \omega_I \d x_i\wedge\d x_I \\ 
&=\delta^* \sum_{I\in\cI_r} \sum_{i\notin I}(-1)^{\sigma(i,I)}\, \delta_i\omega_I \,\d x_{I\cup i} \\ 
&=  \sum_{I \in \mathcal{I}_r} \sum_{i \notin I} \sum_{j \in \left\{ I \cup i\right\}}(-1)^{\sigma(i, I)} \delta^*_i \delta_j \omega_I\  \iota_{\!\ dx_j}(dx_{I \cup i}) \\
 & =  \sum_{I \in \mathcal{I}_r} \sum_{i \notin I} \sum_{j \in \left\{ I \cup i\right\}}(-1)^{\sigma(i, I) + \sigma(j, I \cup i)} \delta_j^* \delta_i\omega_I \ dx_{(I \cup i) \setminus j}. \\
\end{align*}
Next we observe that in the sum $\delta^*\delta\,\omega+\delta\delta^*\omega$ the terms containing indices $i\not=j$ cancel out. Indeed,
 if $i\not=j$ the differential operators $\delta_i$ and $\delta_i^*$ commute with $\delta_j$ and $\delta_j^*$ because they act on different variables. 
Therefore, only the terms with $i=j$ remain and, since for $i=j$ the exponents of $-1$ are even, we have that 
$$
\delta^*\delta\,\omega+\delta\delta^*\omega= \sum_{I\in\cI_r} \sum_{i\in I} \delta_i\delta^*_i \omega_I\d x_I+\sum_{I\in\cI_r} \sum_{i\notin I} \delta^*_i\delta_i \omega_I\d x_I.
$$
\end{proof}

\subsection{A self-adjoint extension of the Hodge-Laguerre operator}\label{s: saeHL}
The operator $\BL_\alpha$ with domain $C^
\infty_c(\rdp;\Uplambda^r)$ is obviously symmetric with respect to 
the inner product $\langle\cdot,\cdot\rangle_\alpha$, but it is not self-adjoint on $L^2(\BR^d,\mu_\alpha;\Uplambda^r)$. In this subsection, to 
define a self-adjoint extension of $\BL_\alpha$, we modify the 
domains of the operators $\delta$, $\delta^*$ and $\BL_\alpha$. With 
the new domains the operators $\delta$ and $\delta^*$ will be adjoint 
to each other and $\BL_\alpha$ will be self-adjoint. The domains will 
be defined via the Fourier-Laguerre transform of forms, that we 
define presently. To this end, first we introduce an orthonormal basis 
for the space of square integrable $r$-forms.
\par
\begin{prop}\label{ONBr}
The family of $r$-forms
$$
B_r=\left\{\ell^{\alpha,I}_{k}\d x_I\  :\ I\in\cI_r, k\in\cK(I)\right\}
$$
is an orthonormal basis of $L^2(\rdp,\mua;\Uplambda^r)$.
\end{prop}
\begin{proof}
By Proposition \ref{ONBr1} the families $\{\ell^{\alpha_i}_{k_i}: k_i\in \BN\}$ 
and $\{\frac{1}{\sqrt{k_i}}\delta_i\ell^{\alpha_i}_{k_i}: k_i\in \BN_+\}$ 
are orthonormal bases of $L^2(\BR_+,\mu_{\alpha_i})$. Thus, by 
tensorization, $\left\{   \ell^{\alpha,I}_{k}\  :\ k\in\cK(I)\right\}$ is an 
orthonormal basis of $L^2(\rdp,  \mu_\alpha)$
 for each $I\in \cI_r$, and 
 $$
 C_r=\left\{\oplus_{I\in\cI_r}(0,\ldots, 0,  \ell^{\alpha,I}_{k},0,\ldots,
0)\  :\  k\in\cK(I)\right\}
$$ 
is an orthonormal basis of the direct 
sum $\cH=\oplus_{I\in\cI_r} L^2(\rdp,\mu_\alpha)$ of $\#\cI_r
$ copies of $L^2(\rdp,\mu_\alpha)$.  Since the map $\omega
\mapsto (\omega_I)_{I\in\cI_r}$ is an isometric isomorphism from 
$L^2(\rdp,\mua;\Uplambda^r)$ to $\cH$  that maps $B_r$ to 
$C_r$,  the family $B_r$ is an orthonormal basis of 
$L^2(\rdp,\mua;\Uplambda^r)$.
\end{proof}
\noindent
{\bf Definition.} The \emph{Fourier-Laguerre coefficients} of a form $
\omega\in L^2(\rdp,\mua;\Uplambda^r)$ are the coefficients of 
$\omega$ with respect to the basis $B_r$, i.e.
$$
\hat{\omega}(I,k)=\langle \omega_I,\ell^{\alpha,I}_{k}\rangle_\alpha, \qquad I\in\cI_r,\ k\in \cK(I).
$$
It is convenient to define $\hat{\omega}(I,k)$ for all $k\in \BN^d$, by setting $\hat{\omega}(I,k)=0$ when $k\notin \cK(I)$. 
 Observe that if $I\in\cI_{r}$ then
 $$
\cK(I)\subset 
\BN^d_r=\{k\in\BN^d: |k|\ge r\}.
$$
\par
To obtain nice formulas for the Fourier-Laguerre transform of the forms $\delta \omega$, $\delta^*\omega$ and $\BL_\alpha\omega$ it is useful to give some algebraic structure to the set of Fourier-Laguerre coefficients. Therefore, we define the \emph{Fourier-Laguerre transform} of the form $\omega\in L^2(\rdp,\mua;\Uplambda^r)$ as the multi-sequence of alternating tensors of rank $r$
$$
\BN^d_r\ni k\mapsto \hat{\omega}(k)=\sum_{I\in\cI_r}\hat{\omega}(I,k)\d x_I\in \Uplambda^r,
$$
and we denote by $\lambda^\alpha_k$ the $r$-form defined by
$$
\lambda^\alpha_k(x)=\sum_{I\in\cI_r} \ell^{\alpha,I}_{k}(x) \d x_I,  \qquad\forall k\in \BN^d.
$$
Define the bilinear map $[\cdot,\cdot]:\Uplambda^r\times\Uplambda^r\to\Uplambda^r$ by
$$
[\omega,\eta]=\sum_{I\in\cI_r} \omega_I\ \eta_I\d x_I.
$$
Then
$$
[\hat{\omega}(k),\lambda^\alpha_k(x)]=\sum_{I\in\cI_r} \hat{\omega}(I,k)\,\ell^{\alpha,I}_{k}(x)\d x_I.
$$
The following proposition gives the inversion formula and Parseval's identity for the Fourier-Laguerre transform. 
\begin{prop}\label{invform+Pars}
For every $\omega\in L^2(\rdp,\mua;\Uplambda^r)$
$$
\omega(x)=\sum_{k\in\BN^d_r} [\hat{\omega}(k),\lambda^\alpha_k(x)],
$$
where the series converges in $L^2(\rdp,\mua;\Uplambda^r)$. Moreover
$$
\normto{\omega}{L^2(\mu_\alpha)}{2}=\sum_{k\in\BN^d_r}
|{\hat{\omega}(k)}|^2.
$$
\end{prop}
\begin{proof} If $\omega\in L^2(\rdp,\mua;\Uplambda^r)$, by definition of orthonormal basis
$$
\omega(x)=\sum_{I\in\cI_r}\sum_{k\in\cK(I)} \hat{\omega}(I,k)\ \ell^{\alpha,I}_{k}(x) \d x_I,
$$
where the series converges in $L^2(\rdp,\mua;\Uplambda^r)$. Since we have defined $\hat{\omega}(I,k)=~0$ for $k\not\in \cK(I)$, we may extend the sum over $\cK(I)$ to a sum over $\BN^d_r$. Thus, exchanging the sums over $I$ and over $k$, we get
$$
\omega(x)=\sum_{k\in\BN^d_r} \sum_{I\in\cI_r}\hat{\omega}(I,k)\ \ell^\alpha_{I,k}(x) \d x_I=\sum_{k\in\BN^d_r} [\hat{\omega}(k),\lambda^\alpha_k(x)].
$$
Similarly,
\begin{align*}
\normto{\omega}{L^2(\mua)}{2}&=\sum_{I\in\cI_r}\sum_{k\in\cK(I)} |\hat{\omega}(I,k)|^2
\\ 
&=\sum_{k\in\BN^d_r}\sum_{I\in\cI_r} |\hat{\omega}(I,k)|^2 \\ 
&= \sum_{k\in\BN^d_r} |\hat{\omega}(k)|^2.
\end{align*}
\end{proof}
Denote by $\msP(\rdp;\Uplambda^r)$ the space  of finite linear combinations of elements of the basis $B_r$, i.e. the space of $r$-forms with polynomial coefficients. Clearly $\msP(\rdp;\Uplambda^r)\subset C^\infty(\rdp,\Uplambda^r)$. 
Next, we compute the Fourier-Laguerre transforms of the forms $\delta \omega$, $\delta^*\omega$ and $\BL_\alpha\omega$, when $\omega$ is in $\msP(\rdp;\Uplambda^r)$. 
\begin{prop}\label{FLTofddL}
For every $k\in \BN^d$ define the covector
$$
\hat{\delta}(k)=\sum_{j=1}^d
\sqrt{k_j} \d x_j.
$$
If $\omega\in \msP(\rdp;\Uplambda^r)$ then for all $k\in\BN^d_r$
\begin{equation}\label{FLTofddH}
\widehat{\delta\omega}(k)=\hat{\delta}(k)\wedge \hat{\omega}(k),\qquad \widehat{\delta^*\omega}(k)=\iota_{\hat{\delta}(k)}\hat{\omega}(k),\qquad \widehat{\BL_\alpha \omega}(k)=|k|\ \hat{\omega}(k).
\end{equation}
The operators $\delta$, $\delta^*$ and $\BL_\alpha$ with domain  $\msP(\rdp;\Uplambda^r)$ are closable in \break
$L^2(\rdp,\mua;\Uplambda^r)$.
\end{prop}
\begin{proof} We observe that 
\begin{align*}
\delta_j \ell^{\alpha,I}_k&=\sqrt{k_j}\  \ell^{\alpha,I\cup j}_k\qquad {\rm if}\quad  j\notin I
 \\ 
\delta^*_j \ell^{\alpha,I}_k&=\sqrt{k_j}\  \ell^{\alpha,I\setminus j}_k
\qquad {\rm if}\quad  j\in I\\
\cL_{\alpha,I} \ell^{\alpha,I}_k&=|k|\ \ell^{\alpha,I}_k.
\end{align*}
Indeed, the first identity follows immediately from the definition of $\ell^{\alpha,I}_k$, the second from the identity $\delta_j^*\delta_j \ell^{\alpha_j}_{k_j}(x_j)=k_j\ell^{\alpha_j}_{k_j}(x_j)$. The last identity follows from the first two and the fact that $\cL^{\alpha,I}_k= \sum_{I\in\cI_r}\sum_{j\in I} \delta_j\delta_j^*+\sum_{I\in\cI_r}\sum_{j\notin I} \delta_j^*\delta_j$.

If $ \omega=\sum_{I \in \mathcal{I}_r}\omega_I \d x_I$  in $\msP(\rdp;\Uplambda^r)$, then
\[ \delta\omega=\sum_{I \in \mathcal{I}_r}\sum_{j\notin I} \delta_j\omega_I \ \d x_j\wedge \d x_I=\sum_{I \in \mathcal{I}_r} \sum_{j\notin I} (-1)^{\sigma(j,I)}\delta_j\omega_I \ \d x_{I \cup j}.\]
Hence
\begin{align*}
 \widehat{\delta\omega}(I \cup j,k)=&  \langle (\delta \omega)_{I \cup j}, \ell^{\alpha,I\cup j}_k\rangle_\alpha = (-1)^{\sigma(j,I)}\ \langle \delta_j\omega_I, \ell^{\alpha,I \cup j}_k\rangle_\alpha\\ 
=&(-1)^{\sigma(j,I)}\ \langle \omega_I, \delta_j^* \ell^{\alpha,I \cup j}_k\rangle_\alpha = (-1)^{\sigma(j,I)}\ \langle \omega_I, \sqrt{k_j}\ \ell^{\alpha,I}_k\rangle_\alpha\\
= & (-1)^{\sigma(j,I)}\ \sqrt{k_j}\ \hat{\omega}(I,k).
\end{align*}
Thus
\begin{align*}
\widehat{\delta\omega}(k)=&\sum_{J\in\mathcal{I}_{r+1}}\widehat{\delta \omega}(J,k) \ \d x_J  \\
=& \sum_{I\in \mathcal{I}_r}\sum_{j\notin I} \widehat{\delta \omega}(I\cup j,k) \d x_{I \cup j} =  \\
=& \sum_{I \in \mathcal{I}_r}\sum_{j\notin I} (-1)^{\sigma(j,I)}\  \sqrt{k_j}\ \hat{\omega} (I,k) \d x_{I\cup j} \\ 
=& \sum_{I \in \mathcal{I}_r}\sum_{j\notin I} \sqrt{k_j}\ \hat{\omega}(I,k) \d x_j\wedge \d x_I  \\ 
=&\left(\sum_{j=1}^d \sqrt{k_j} \d x_j\right)\wedge \left(\sum_{I \in \mathcal{I}_r} \hat{\omega}(I,k) \d x_I\right)\\
=&\hat{\delta}(k)\wedge\hat{\omega}(k).
\end{align*}
To prove the identity $\widehat{\delta^*\omega}(k)=\iota_{\hat\delta(k)} \hat \omega(k)$, we observe that for all $\eta \in \msP(\rdp;\Uplambda^{r-1})$
\begin{align*}
\langle \delta^* \omega, \eta\rangle_\alpha= \langle \omega, \delta \eta\rangle_\alpha&= \sum_{k \in \mathbb{N}^d} \langle \hat{\omega}(k), \widehat{\delta \eta} (k)\rangle_{\Uplambda^r} \\ 
&= \sum_{k \in \mathbb{N}^d} \langle \hat{\omega}(k), \hat{\delta}(k)\wedge \hat{\eta}(k)\rangle_{\Uplambda^r}  \\
& = \sum_k \langle\iota_{\hat{\delta}(k)} \hat{\omega}(k), \hat{\eta}(k)\rangle_{\Uplambda^{r-1}},
\end{align*}
by  Parseval's identity and the fact that the operator of interior multiplication by $\hat{\delta}(k)$ is the adjoint with respect to the inner product on covectors of the exterior multiplication by $\hat{\delta}(k)$. Since $\eta$ is arbitrary, the conclusion follows.
\par
To prove the last identity, we observe that by  Proposition \ref{p: diag} $(\BL_\alpha\omega)_I=\cL_{\alpha,I}\omega_I$. Thus
\begin{align}\label{hsao}
\widehat{\BL_\alpha \omega}(I,k)&=\langle(\BL_\alpha \omega)_I,\ell^{\alpha,I}_{k}\rangle_\alpha \nonumber\\ 
&= \langle\cL_{\alpha,I} \omega_I,\ell^{\alpha,I}_k\rangle_\alpha \nonumber\\ 
&= \langle\omega_I,\cL_{\alpha,I}  \,\ell^{\alpha,I}_k\rangle_\alpha\\
&=|k|\ \langle\omega_I, \,\ell^{\alpha,I}_k\rangle_\alpha\nonumber\\
&=|k|\ \widehat{\omega}(I,k).\nonumber
\end{align}
Hence
$$
\widehat{\BL_\alpha\omega}(k)=\sum_I \widehat{\BL_\alpha \omega}(I,k) \d x_I= |k| \ \sum_I \widehat{\omega}(I,k) \d x_I=|k|\ \widehat{\omega}(k).
$$
\par
It is now an easy matter to see that $\delta$, $\delta^*$ and $\BL_\alpha$ with domain  $\msP(\rdp;\Uplambda^r)$ are closable in $L^2(\rdp,\mua;\Uplambda^r)$. Indeed, if $(\omega_n)$ is a sequence in $\msP(\rdp;\Uplambda^r)$ such that $\omega_n\to0$ and $\delta \omega_n\to \eta$ in $L^2(\rdp,\mua;\Uplambda^r)$, then $\hat\omega_n(k)\to 0$ and $\hat\eta(k)=\lim_n\widehat{\delta \omega_n}(k)=\hat\delta(k)\wedge\hat\omega_n(k)=0$ for every $k$. Hence $\eta=0$.\par
 The proofs that $\delta^*$ and $\BL_\alpha$ are closable are similar.
\end{proof}
\begin{notation}\label{not1501}
With a slight abuse of notation, we denote also by $\delta$, $\delta^*$ and $\BL_\alpha$ the closures in $L^2(\rdp,\mua;\Uplambda^r)$ of the operators $\delta$, $\delta^*$ and $\BL_\alpha$ on $\msP(\rdp;\Uplambda^r)$). 
\end{notation}
The following proposition characterises their domains in $L^2(\rdp,\mua;\Uplambda^r)$ via the Fourier-Laguerre transform. 
\begin{prop}\label{dom}
The domains of $\delta$, $\delta^*$ and $\BL_\alpha$ on $L^2(\rdp,\mua;\Uplambda^r)$ are 
\begin{align*}
\msD_r(\delta)&=\left\{\omega\in L^2(\rdp,\mua;\Uplambda^r):  \sum_{k \in \BN^d_r} |{\hat\delta(k)\wedge \hat\omega(k)}|_{\Uplambda^{r+1}}^2<\infty\right\} \\ 
\msD_r(\delta^*)&=\left\{\omega\in L^2(\rdp,\mua;\Uplambda^r): \sum_{k \in \BN^d_r} |{\iota_{\hat\delta(k)} \hat\omega(k)}|_{\Uplambda^{r-1}}^2<\infty \right\} \\ 
\msD_r(\BL_\alpha)&=\left\{\omega\in L^2(\rdp,\mua;\Uplambda^r): \sum_{k\in \BN^d_r} |k|^2\ |{ \hat\omega(k)}|_{\Uplambda^r}^2<\infty \right\}.
\end{align*}
The identities (\ref{FLTofddH}) continue to hold for $\omega$ in the domains of $\delta$, $\delta^*$ and $\BL_\alpha$.
\end{prop}
\begin{proof}
The proof is straightforward.
\end{proof}
\begin{prop}\label{coinc}
The space $C_c^\infty(\rdp;\Uplambda^r)$ is contained in the 
spaces $\msD_r(\delta)$, $\msD_r(\delta^*)$ and $\msD_r(\BL_\alpha)
$ and on it the operators $\delta$, $\delta^*$ and $\BL_\alpha$ 
coincide with the closures in $L^2(\rdp,\mua;\Uplambda^r)$ of the operators $\delta$, $\delta^*$ and $\BL_\alpha$ on $\msP(\rdp;\Uplambda^r)$). 
\end{prop}
\begin{proof}
If $\omega\in C_c^\infty(\rdp;\Uplambda^r)$ then its Fourier-
Laguerre coefficients decay faster than any power of $|k|$, because 
by (\ref{hsao}), for every positive integer $m$
\begin{align*}
|k|^m |\widehat{\omega}(I,k)|&=|\widehat{\BL_\alpha^m \omega}
(I,k)| \\ 
&=\int_{\rdp} (\BL_\alpha^m \omega)_I\  \ell^{\alpha,I}_k
\d \mu_\alpha \\ 
&\le \norm{\BL_\alpha^m \omega}{L^2}.
\end{align*}
This shows that $C_c^\infty(\rdp;\Uplambda^r)$ is contained in the 
domains $\msD_r(\delta)$, $\msD_r(\delta^*)$ and $\msD_r(\BL_
\alpha)$. The fact that the operators on $L^2(\rdp,\mua;\Uplambda^r)$ coincide with with the closures in $L^2(\rdp,\mua;\Uplambda^r)$ of the operators $\delta$, $\delta^*$ and $\BL_\alpha$ on $\msP(\rdp;\Uplambda^r)$) follows easily by computing the Fourier-
Laguerre coefficients of $\delta\omega$, $\delta^* \omega$ and $
\BL_\alpha\omega$, intended in the classical sense, as in the proof 
of  Proposition~\ref{FLTofddL}.
\end{proof}
\begin{prop}\label{sa}
The operators $\delta$ and $\delta^*$ on their domains $
\msD_r(\delta)$ and $\msD_r(\delta^*)$ are adjoint of each other. The 
operator $\BL_\alpha$ on $\msD_r(\BL_\alpha)$ is self-adjoint and its 
spectral resolution is
\begin{equation}\label{f: specres}
\BL_\alpha=\sum_{n\ge r} n \cP^\alpha_n
\end{equation}
where $\cP^{\alpha}_n$ is the orthogonal projection onto the space spanned
by the forms $\lambda^\alpha_k$, $|k|=n$. 
\end{prop}
\begin{proof}
To prove that $\delta^*$ is the adjoint of $\delta$ observe that, if $
\omega\in \mathscr{D}_r(\delta)$ and $\eta\in \mathscr{D}_{r+1}(\delta^*)
$, then by polarising Parseval's identity 
\begin{align*}
\langle \delta \omega, \eta\rangle_\alpha&=\sum_{k \in \BN^d_{r+1}}\langle \widehat{\delta \omega}(k),\hat\eta(k)
\rangle_{\Uplambda^{r+1}} \\ 
&= \sum_{k \in \BN^d_{r+1}} \langle \hat{\delta}(k)\wedge\hat 
\omega(k),\hat\eta(k)\rangle_{\Uplambda^{r+1}}\\
&=\sum_{k \in \BN^d_{r+1}}\langle  \omega(k),\iota_{\hat\delta(k)}
\eta(k)\rangle_{\Uplambda^{r+1}} \\
&=\langle \omega,\delta^*\eta \rangle_\alpha.\\		
\end{align*}
Thus the adjoint of $\delta$ is an extension of $\delta^*$.\\
Conversely, if $\eta\in L^2(\rdp,\mua;\Uplambda^{r+1})$ is in 
the domain of the adjoint of $\delta$, then for every $\omega\in 
\mathscr{P}(\rdp;\Uplambda^r)$ there exists a constant $C(\omega)
$ such that
\[ |{\langle\delta\omega,\eta\rangle_\alpha}| \le C(\eta) \ \|{\omega}\|
_{L^2(\mu_\alpha)}.\]
This implies that 
\[ \left|{\sum_{k \in \mathbb{N}^d_r}\langle  \omega(k),\hat\iota_{\hat
\delta(k)}\eta(k)\rangle_{\Uplambda^{r}(\rdp)}}\right|\le C(\eta)\ \|
{\omega}\|_{L^2(\mu_\alpha)}.\]
Since this holds for all $\omega\in\mathscr{P}(\Uplambda^r(\rdp))$, it 
follows that 
\[ \sum_{k\in \BN^d_r} \|{\iota_{\hat\delta(k)}\eta(k)}\|_{\Lambda^{r}(\rdp)}^2<\infty,
\]
that is $\eta\in \mathscr{D}(\delta^*)$. This proves that $\delta^*$ is 
the adjoint of $\delta$. The proof that the $\delta$ is the adjoint of $
\delta^*$ is similar.\par
To show that $\BL_\alpha$ is self-adjoint, it is enough to remark that $\BL_\alpha$ is unitarily equivalent, via the Fourier-Laguerre transform, to the  operator of multiplication by the function
$k\mapsto|k|$ acting on its natural domain in the space $\ell^2({\BN^d_r};\Uplambda^r(\BR^d))$  of square summable $\Uplambda^r(\BR^d)$-valued multi-sequences.\par
Finally the spectral resolution of $\BL_\alpha$ follows from the following facts 
\begin{itemize}
\item[(a)] $\{\ell^{\alpha,I}_k \d x_I: I\in\cI_r, k\in\cK(I)\}$ is an orthonormal basis of $L^2(\rdp,\mua;\Uplambda^r)$;
\item[(b)] $\BL_\alpha\left(\ell^{\alpha,I}_k \d x_I\right)=\cL^\alpha_{I,k}\ell^{\alpha,I}_k \d x_I=|k|\ \ell^{\alpha,I}_k \d x_I$;
\item[(c)] $\lambda^\alpha_k=\sum_{\in\cI_r} \ell^{\alpha,I}_k \d x_I$;
\item[(d)] if $k\in \cK(I)$ then $|k|\ge r$.
\end{itemize}
\end{proof}
We shall denote by 
$
\ker_r(\delta),\ \ker_r(\delta^*)$, $\ker_r(\BL_\alpha)$ and  $\im_r(\delta), \ \im_r(\delta^*) 
$ and $\im_r(\BL_\alpha)$
the kernels and the images of $\delta$, $\delta^*$ and $\BL_\alpha$, considered as  operators on $L^2(\rdp,\mua;\Uplambda^r)$.  Thus $\ker_r(\BL_\alpha)$ is the space of $r$-harmonic forms in $L^2(\rdp,\mua;\Uplambda^r)$. It follows from the spectral resolution of $\BL_\alpha$ that the only harmonic $0$-forms are the constants, while there are no non-trivial harmonic $r$-forms for $r\ge 1$. \par
In Section \ref{s: HLonf} we defined the Hodge-Laguerre operator on smooth forms as $\BL_\alpha=\delta\delta^*+\delta^*\delta$. The same identity for the corresponding unbounded operators on $L^2(\rdp,\mua;\Uplambda^r)$ is not obvious, because one must verify that the domains of the left and right hand side coincide. Indeed, one has
\begin{prop}\label{domid} Let
$$
\msD_r(\delta\delta^*+\delta^*\delta)=\{\omega\in \msD_r(\delta)\cap\msD_r(\delta^*): \delta\omega\in \msD_{r+1}(\delta^*) {\rm and}\ \delta^*\omega\in \msD_{r-1}(\delta)\}. 
$$
Then $\msD_r(\delta\delta^*+\delta^*\delta)=\msD_r(\BL_\alpha)$ and $\BL_\alpha=\delta\delta^*+\delta^*\delta$. Moreover $\BL_\alpha$ commutes with $\delta$ and $\delta^*$ and
$$
\im_r(\delta)\subset \ker_r(\delta),\qquad \im_r(\delta^*)\subset \ker_r(\delta^*),
$$
i.e. $\delta^2=0$ and $(\delta^*)^2=0$.
\end{prop}
\begin{proof}
If $\omega\in \msD_r(\delta\delta^*+\delta^*\delta)$ then $\delta
\delta^*\omega+\delta^*\delta\omega \in L^2(\Uplambda^r(\rdp),
\mua)$. Hence, by Parseval's identity, the Fourier-Laguerre 
transform of $\delta\delta^*\omega+\delta^*\delta\omega$ is in $
\ell^2(\BN^d_r,\Uplambda^r)$. Since by Lemma \ref{li+il}
\begin{align}
(\delta\delta^*\omega+\delta^*\delta\omega)\hat{\phantom{.}}(k)=&
\iota_{\hat{\delta}(k)}(\hat\delta(k)\wedge\hat\omega(k))+\hat\delta(k)
\wedge\iota_{\hat{\delta}(k)}\hat\omega(k) \nonumber \\ 
&=|\hat\delta(k)|^2\ \hat\omega(k) \label{comm}\\ 
&=|k|\ \hat\omega(k), \nonumber
\end{align}
also the multi sequence $k\mapsto |k|\ \hat\omega(k)$ is in $
\ell^2(\BN^d_r,\Uplambda^r)$, i.e. $\omega\in \msD_r(\BL_\alpha)
$. This proves the inclusion $\msD_r(\delta\delta^*+\delta^*\delta)
\subset \msD_r(\BL_\alpha)$. Conversely, if $\omega\in 
\msD_r(\BL_\alpha)$, by (\ref{comm}) and Lemma \ref{li+il}
$$
|\iota_{\hat\delta(k)}(\hat\delta(k)\wedge\hat\omega(k)|^2+|\hat
\delta(k)\wedge\iota_{\hat\delta(k)} \hat\omega(k)|^2 = |k|^2\ |\hat
\omega(k)|^2.
$$
Thus, the same argument based on  Parseval's identity proves that 
$\msD_r(\BL_\alpha)\subset\msD_R(\delta^*\delta+\delta\delta^*)$.
\par
The other statements  can be easily proved by observing that the 
Fourier-Laguerre transforms of the left and the right hand sides coincide.
\end{proof}
\begin{prop}\label{prop1670}
$$
\ker_r(\BL_\alpha)=\{\omega\in \msD_r(\delta)\cap\msD_r(\delta^*): 
\delta \omega=0, \delta^*\omega=0\}.
$$
\begin{proof}
If $\omega\in \ker_r(\BL_\alpha)$ then $\omega\in \msD_r(\delta)
\cap  \msD_r(\delta^*)$, by Proposition \ref{domid}. Moreover
$$
\normto{\delta\omega}{L^2(\mua)}{2}+\normto{\delta^*\omega}
{L^2(\mua)}{2}=\langle \BL_\alpha \omega,\omega\rangle_
\alpha=0.
$$
Thus $\delta \omega=0, \delta^*\omega=0$. Conversely, if $\omega
\in \msD_r(\delta)\cap\msD_r(\delta^*)$ and  $\delta \omega=0, \ \delta^*
\omega=0$ then, by Proposition \ref{domid}, $\omega\in 
\msD_r(\BL_\alpha)$ and $\BL_\alpha\omega=(\delta\delta^*+
\delta^*\delta)\omega=0$.
\end{proof}
\end{prop}
In Section \ref{s: diag} we proved that the action of the Hodge-Laguerre operator on \emph{smooth} $r$-forms is diagonalised by the basis $\{\d x_I: I\in\cI_r\}$ of $\Uplambda^r(\BR^d)$.  A similar result holds for the action of  $\BL_\alpha$ on $L^2(\mua)$ forms. Namely
\begin{prop}\label{p: L2diag}
A form $\omega=\sum_{I\in\cI_r}\omega_I\d x_I$ is in $\msD_r(\BL_\alpha)$ if and only if $\omega_I\in \msD(\cL_{\alpha,I})$ for all $I\in\cI_r$. Moreover
$$
\BL_\alpha\omega=\sum_{I\in\cI_r}\cL_{\alpha,I}\omega_I\d x_I \qquad\forall \omega \in \msD_r(\BL_\alpha).
$$
\end{prop}
\begin{proof}
The fact that $\omega$ is in $\msD_r(\BL_\alpha)$ if and only if $\omega_I\in \msD(\cL_{\alpha,I})$ for all $I\in\cI_r$ follows immediately from the chararcterization of the domains via the Fourier-Laguerre coefficients of $\omega$ and of its components $\omega_I$. The identity of the operators holds on $\msP(\rdp;\Uplambda^r)$ by Proposition \ref{p: diag} and extends to the $L^2$ domains of the operators, since $\msP(\rdp;\Uplambda^r)$ is dense in the domain in the graph norm.
\end{proof}
We denote by $\{\BT^\alpha_t: t\ge0\}$ the heat semigroup on $r$-forms, i.e. the semigroup on $L^2(\rdp,\mua;\Uplambda^r)$ generated by $-\BL_\alpha$, and by  $\{\BP^\alpha_t: t\ge0\}$ the corresponding Poisson semigroup generated by $-\BL_\alpha^{1/2}$. 
More generally, for every $\rho\le r$ we consider the semigroups generated by $\rho I-\BL_\alpha$ and by $-(\BL_\alpha-\rho I)^{1/2}$, i.e.  the semigroups
$$
\BT^{\alpha,\rho}_t=\e^{-t(\BL_\alpha-\rho I)}, \qquad \BP^{\alpha,\rho}_t=\e^{-t(\BL_\alpha-\rho I)^{1/2}}.
$$
Their spectral resolutions are
\begin{align}
\BT^{\alpha,\rho}_t &=\sum_{n\ge r} \e^{-t(n-\rho)} \cP^\alpha_n \\ 
\label{srPoisson}\BP^{\alpha,\rho}_t &=\sum_{n\ge r} \e^{-t\sqrt{n-\rho}} \cP^\alpha_n,
\end{align}
where, as before, $\cP^\alpha_n$ denotes the orthogonal projection onto the space spanned by the forms $\lambda^\alpha_k(x)=\sum_{I}\ell^{\alpha,I}_k(x)\d x_I$,  $|k|=n$.  \par
\begin{prop}\label{p: Bakry}
For every $\alpha\in(-1,\infty)^d$   there exists a constant $C(\alpha)$
such that for all forms $\omega\in \Lpr{2}{r}$
\begin{itemize}
\item[\rmi] $|\BT^{\alpha,\rho}_t \omega(x)|\le \,C(\alpha)\ \e^{t(\rho-r/2)}\, T^\alpha_t|\omega|(x)$
\item[\rmii]  $|\BP^{\alpha,\rho}_t \omega(x)|\le \,C(\alpha)\ \, P^\alpha_t|\omega|(x)$ if $\rho\le r/2$.
\end{itemize} 
for all $x\in\rdp$ and $t\ge 0$. {If $\alpha\in[-1/2,\infty)$ then $C(\alpha)=1$.}
\end{prop}
\begin{proof}
By Proposition \ref{p: L2diag} $\BT^\alpha_t \omega =\sum_{I\in\cI_r} T^{\alpha,I}_t\,\omega_I$. Thus, by Proposition \ref{p: dom} and the positivity of the semigroup,
\begin{align*}
|\BT^\alpha_t\omega(x)|&=\sup_{|\eta|=1}\sum_{I\in\cI_r}\eta_I\ T^{\alpha,I}_t \,\omega_I(x) \\ 
&\le\sup_{|\eta|=1}\sum_{I\in\cI_r}\ |T^{\alpha,I}_t (\eta_I\omega_I)(x)|  \\ 
&\le C(\alpha)\, \e^{-tr/2}\ \sup_{|\eta|=1}\sum_{I\in\cI_r} T^\alpha_t |\eta_I\omega_I|(x)\\
&= C(\alpha)\, \e^{-tr/2}\ T^\alpha_t\left(\sup_{|\eta|=1}\sum_{I\in\cI_r}  |\eta_I\omega_I|\right)(x)\\
&= C(\alpha)\  \e^{-tr/2}\  T^\alpha_t |\omega|(x).
\end{align*} 
This proves \rmi \ for $\rho=0$. The general case follows, since $\BT^{\alpha,\rho}_t=\e^{\rho t}\,\BT^\alpha_t$. The estimate \rmii\ follows from (i) and the subordination formula
$$
\BP^{\alpha,\rho}_t =\frac{1}{\sqrt{\pi}} \int_0^\infty \frac{\e^{-u}}{\sqrt{u}}\  \BT^{\alpha,\rho}_{\frac{t^2}{4u}} \d u.
$$
\end{proof}
\begin{cor}\label{c: sgbdd}
For every $\alpha\in(-1,\infty)^d$  and $1\le p\le \infty$  there exists a constant $C(\alpha,p)$ such that for all  forms $\omega\in L^p\cap \Lpr{2}{r}$ and all $\rho\le r$
\begin{itemize}
\item[\rmi] $\norm{\BT^{\alpha,\rho}_t\omega}{p}\le \,C(\alpha,p)\ \e^{\gamma(\rho,r,p)t}\ \norm{\omega}{p},$
\item[\rmii] $\norm{\BP^{\alpha,\rho}_t\omega}{p}\le \,C(\alpha,p)\ \ \norm{\omega}{p},$ if $\rho\le r/2$, 
\end{itemize} 
where
$
\gamma(\rho,r,p)=\rho-\left(1-\left| 1/2-1/p\right|\right)r.
$
\par\noindent
In particular, if $\rho\le r/2$
the semigroups $\BT^{\alpha,\rho}_t$ and $\BP^{\alpha,\rho}_t$ extend from $L^p\cap \Lprss{2}{r}$ to semigroups which are uniformly bounded on $\Lpr{p}{r}$ and for every $\delta>0$
$$
\norm{\BL_\alpha^{-\delta}}{p-p}\le \frac{C(\alpha,p)}{\gamma(0,r,p)} \qquad\forall p\in(1,\infty).
$$
\end{cor}
\begin{proof}
By the spectral resolution \eqref{f: sar}, the bottom of the spectrum of $\BT^{\alpha}_t$ on $\Lprss{2}{r}$ is $\e^{-rt}$. Thus we have that 
$
\norm{\BT^{\alpha}_t\omega}{2}\le \, \e^{-rt}\,\norm{\omega}{2},
$. On the other hand, by Proposition \ref{p: Bakry}, we also have that 
$
\norm{\BT^{\alpha}_t\omega}{1}\le \, C(\alpha)\, \e^{-rt/2}\norm{\omega}{1}.
$
Therefore, by interpolation, for all $1\le p \le 2$,
\begin{equation}\label{semigrest}
\norm{\BT^{\alpha}_t\omega}{p}\le \, C(\alpha)^{|(2/p)-1|}\, \e^{-(1-|1/2-1/p|)rt} \, \norm{\omega}{p}.
\end{equation}
Since the semigroup is symmetric, by duality the same result holds also for $2<p\le \infty$. The result for $\BT^{\alpha,\rho}_t$ follows immediately, since $\BT^{\alpha,\rho}_t=\e^{\rho t}\BT^{\alpha}_t$.  The same argument yields the desired estimate also for the Poisson semigroup $\BP^{\alpha,\rho}_t$. The estimate of the norm of $\BL_\alpha^{-\delta}$ follows from the identity
$$
\BL_\alpha^{-\delta}=\frac{1}{\Gamma(\delta)}\int_0^\infty t^{\delta-1} \BT^\alpha_t \d t
$$
and estimate (\ref{semigrest}).
\end{proof}
\section{Hodge-De Rham-Kodaira decomposition and Riesz-Laguerre transforms}\label{s: HdRL2}
In this section we state the analogue of the classical Hodge 
decomposition for the Hodge-Laguerre operator and we define the 
Laguerre-Riesz transforms. We recall that $\ker_0(\BL_\alpha)=
\BR$ and $\ker_r(\BL_\alpha)=\{0\}$ if $r>0$. Moreover we set $
\im_{-1}(\delta)=\{0\}$ and $\im_{d+1}(\delta^*)=\{0\}$.
\begin{thm}\label{Hdec}
For all $r\ge0$ the strong $L^2$-Hodge decomposition holds
$$
L^2(\rdp,\mua;\Uplambda^r)=\ker_r(\BL_\alpha)\oplus \im_{r-1}
(\delta)\oplus \im_{r+1}(\delta^*).
$$
\end{thm}
\begin{remark*}\label{rem1787}
Here strong refers to the fact that $\im_{r-1}(\delta)$ and $\im_{r+1}
(\delta^*)$ are closed.  The proof is essentially the same as in the 
classical case of complete manifolds with spectral gap (see for 
instance \cite[Theorem 5.10]{Bueler:TAMS}). 
\end{remark*}
\begin{definition}\label{def1793}
The Riesz transforms are the operators
$$
\cR=\delta\BL_{\alpha}^{-1/2}, \quad \cR^*=\delta^*\BL_{\alpha}^{-1/2}
$$
with domain $\msP(\rdp;\Uplambda^r)$.
\end{definition}
The following proposition is the counterpart in the Laguerre setting of a classical result of Strichartz for complete Riemannian manifolds \cite{S:JFA}.
\begin{prop}\label{p: RT}
For $r>0$ the Riesz transforms $\cR$ and $\cR^{*}$ are bounded on
$L^2(\rdp,\mua;\Uplambda^r)$. Moreover $\cR^*$ is the adjoint of $\cR$.
\end{prop}
\begin{proof}
Let $\omega$ be a form in $\msP(\rdp;\Uplambda^r)$. Then
\begin{align*}
\normto{\delta \omega}{2}{2}+\normto{\delta^*\omega}{2}{2}=\langle\BL_\alpha\omega,\omega\rangle_\alpha
=\normto{\BL_\alpha^{1/2}\omega}{2}{2}.
\end{align*}
This proves that the Riesz transforms are $L^2$-bounded on $\msP(\rdp;\Uplambda^r)$. The conclusion follows, since $\msP(\rdp;\Uplambda^r)$ is dense in $L^2(\rdp,\mua;\Uplambda^r)$.
Since $\delta^*$ and $\BL_\alpha$ commute, $\cR^*=\BL^{-1/2}_\alpha\delta^*$ is the adjoint of $\delta\BL^{-1/2}_\alpha$.
\end{proof}
\begin{prop}\label{p: proj}
For $r>0$ the orthogonal projections onto the spaces $\im_{r-1}(\delta)$ and $
\im_{r+1}(\delta^*)$ are 
$$\cR\cR^{*}=\delta \BL_\alpha^{-1}\delta^*, \quad\textrm{and}\quad \cR^*\cR=\delta^*\BL_
\alpha^{-1}\delta,
$$ 
respectively.
\end{prop}
\begin{proof}
Let $P=\cR\cR^*$ and $Q=\cR^*\cR$. Then $P$ and $Q$  are bounded on $\Lprs{2}{r}$ and self-adjoint, by Proposition \ref{p: RT}. Since $\delta$ and $\delta^*$ commute with $\BL_\alpha$, $P+Q=(\delta\delta^*+\delta^*\delta)\BL_\alpha^{-1}=I$. Moreover $PQ=QP=0$, because $\delta^2=0$ and $(\delta^*)^2=0$.
Hence $P^2-P=P(I-P)=PQ=0=(I-Q)Q=Q-Q^2$.  This proves that $P$ and $Q$ are idempotent. Therefore they are orthogonal projections. The conclusion follows, since $\im_r(P)\subseteq\im_{r-1}(\delta)$ and $\im_r(Q)\subseteq\im_{r+1}(\delta^*)$.
\end{proof}
\begin{remark}\label{r: RT}
If $r=0$ the conclusions of Propositions \ref{p: RT} and \ref{p: proj}
remain valid, if one replaces $\BL_\alpha^{-1/2}$ and $\BL_\alpha^{-1}$ by their restrictions to the orthogonal of constant functions.
\end{remark}
\begin{definition}\label{Rrho} More generally, for every $\rho<r$ we define the \emph{shifted} Riesz transforms 
$$
\cR_\rho=\delta(\BL_\alpha-\rho I)^{-1/2} \quad \textrm{and} \quad \cR^*_\rho=(\BL_\alpha-\rho I)^{-1/2}\delta^*.
$$
\begin{prop}\label{Rrhobd}
For every $\rho<r$ the shifted Riesz transforms $\cR_\rho$ and $\cR^*_\rho$ are bounded on $L^2(\rdp,\mua;\Uplambda^r)$. Moreover $\cR^*_\rho$ is the adjoint of $\cR_\rho$.
\end{prop}
\begin{proof}
As in the proof of Proposition \ref{p: RT} it suffices to observe that 
\begin{align*}
\normto{\delta\omega}{2}{2}+\normto{\delta^*\omega}{2}{2}&=\normto{\BL_\alpha^{1/2} \omega}{2}{2}
\\
&=\normto{\BL_\alpha^{1/2}(\BL_\alpha-\rho I)^{-1/2}(\BL_\alpha-\rho I)^{1/2}\omega}{2}{2}
\\
&
\le \left(\frac{r}{r-\rho}\right)\normto{(\BL_\alpha-\rho I)^{1/2}\omega}{2}{2}
\end{align*}
since $\norm{\BL_\alpha^{1/2}(\BL_\alpha-\rho I)^{-1/2}}{2}\le \sqrt{r/(r-\rho)}$ by \eqref{f: specres}.
\end{proof}

\end{definition}

\section
{The Bilinear Embedding Theorem  and its applications}\label{s: BETapp}
We consider the manifold $M=\rdp\times\BR_+$, with coordinates $ =(x_1,\ldots,x_d,t)$. We recall that  $\delta_i$, for $i=1,\ldots,d$,   denotes the Laguerre derivative $\sqrt{x_i}\partial_i$, and we denote by $\delta_{d+1}=\partial_t$  the classical derivative with respect to $t$.  Given a form $\omega=\sum_{I\in\cI_r} \omega_I(x,t)\,\d x_I$ in $C^\infty\big(\Uplambda^r(M)\big)$ 
we define
$$
|\overline{\nabla}\omega(x,t)|=\left(\sum_{i=1}^{d+1} \sum_{I\in\cI_r} |\delta_i \omega_I(x,t)|^2\right)^{1/2}.
$$
\begin{thm}[Bilinear embedding Theorem]
\label{bilemb}
Suppose that {$\alpha\in[-1/2,\infty)^d$} and $\rho\le r/2$. For each $p \in (1, \infty)$, $\omega\in \mathscr{P}(\Uplambda^{r}(M))$ and $\eta \in \mathscr{P}(\Uplambda^{r+1}(M))$ 
$$
\int_0^{\infty} \int_{\rdp} |\overline{\nabla} \BP^{\alpha,\rho}_t \omega(x)| |\overline{\nabla} \BP^{\alpha,\rho}_t\eta(x)|  \d \mua(x) \,t\d t \le 6(p^*-1) \norm{\omega}{L^p(\mua)}\, \norm{\eta}{L^q(\mua)},
$$
where $q$ is the conjugate exponent of $p$, and $p^* = \max \left\{p,q \right\}$.
\end{thm}
We postpone the proof of this result   to deduce some of its consequences.
\subsection{Riesz-Laguerre 
transforms on $\Lpr{p}{r}$}
A first consequence of the Bilinear Embedding Theorem is the boundedness on $\Lpr{p}{r}$ of the shifted Riesz transforms $\cR_\rho$ when $\rho\le r/2$.
\begin{thm}\label{t: Riesz bdd Lp} Suppose that {$\alpha\in[-1/2,\infty)^d$}, $r\ge1$ and  $\rho\le r/2$.
Then for each $p \in (1, \infty)$ the shifted Riesz transforms $\cR_\rho$ and $\cR{_\rho}^*$ extend to bounded operators from $\Lpr{p}{r}$ to  $\Lpr{p}{r+1}$ and to $\Lpr{p}{r-1}$, respectively. Moreoever  for all $\omega \in \Lpr{p}{r}$
$$
\norm{\cR_\rho\, \omega}{L^p(\mua)}\le C(p)\, \norm{\omega}{L^p(\mua)},
\qquad\norm{\cR_\rho^*\, \omega}{L^p(\mua)}\le C(p)\, \norm{\omega}{L^p(\mua)}
$$
where
$C(p)= 24 (p^*-1)$.
If $r=0$ the inequality holds for all $\omega$ in $\Lpr{p}{0}$ with integral zero.
\end{thm}

\begin{proof}
[Proof of Theorem \ref{t: Riesz bdd Lp}]
The result is a straightforward consequence of Theorem \ref{bilemb},  the following representation formula and a duality argument.
\begin{lem}
\label{repr}
If $r\ge 1$ then for every $\omega\in \mathscr{P}(\Uplambda^{r}(\rdp))$ 
and $\eta \in \msP(\Uplambda^{r+1}(\rdp))$,
\begin{equation}
\label{eqrepr}
\langle \cR_\rho\, \omega, \eta \rangle_\alpha = - 4 \int_{0}^{\infty}\left\langle \delta \BP^
{\alpha,\rho}_t \omega, \frac{\d}{\d t} \BP^{\alpha,\rho}_t\eta\right\rangle_\alpha
 t\d t.
\end{equation}
If $r=0$ the identity holds for all  $\omega$ in
$\cP(\Uplambda^0(\rdp))$ orthogonal to the constants.
\end{lem}
\begin{proof}[Proof of Lemma \ref{repr}]
Let
$$
\Psi(t) = \langle \BP^{\alpha,\rho}_t \cR \omega,  \BP^{\alpha,\rho}_t \eta \rangle_
\alpha.
$$
We claim that
\begin{itemize}
\item[\rmi] $\lim_{t\to\infty} \Psi(t)=\lim_{t\to\infty}t\Psi'(t)=0$.
\item[\rmii] $\Psi'(t)=-2\langle \delta\BP^{\alpha,\rho}_t\omega,\BP^{\alpha,\rho}_t \eta,\rangle_\alpha$;
\item[\rmiii] $\Psi''(t)=-4\langle \delta\BP^{\alpha,\rho}_t\omega,\frac{\d}{\d t} \BP^{\alpha,\rho}_t \eta\rangle_\alpha$;
\end{itemize} 
Assuming the claim for the moment, the desired identity follows since
$$
\langle \cR  \omega,\eta\rangle_\alpha=\Psi(0) = \int_0^\infty \Psi''(t)\ t\d t=-4 \int_0^\infty\langle \delta\BP^{\alpha,\rho}_t\omega,\frac{\d}{\d t} \BP^{\alpha,\rho}_t \eta\rangle_\alpha \ t\d t.
$$
It remains only to prove the claim. By the spectral resolution of the Poisson semigroup \eqref{srPoisson}
$$
\Psi(t)=\sum_{|k|\ge r} \e^{-2t|k-\rho|^{1/2}} \langle \cP_k\,\cR \omega,\eta\rangle_\alpha.
$$
This proves \rmi (notice that if $r=0$ the sum starts from $1$, since we assume that $\omega$ is orthogonal to the constants). The other two identities follow easily, since $\delta$ and $\cR_\rho=\delta(\BL_\alpha-\rho I)^{-1/2}$ commute with $\BP^{\alpha,\rho}_t$ and $$
\frac{\d}{\d t}\BP^{\alpha,\rho}_t\omega=-(\BL_\alpha-\rho I)^{1/2}\BP^{\alpha,\rho}_t
\omega,\quad \cR_\rho(\BL_\alpha-\rho I)^{1/2}\omega=\delta\omega \qquad\forall 
\omega\in
\cP(\Uplambda^r(\rdp)).
$$
Indeed, since $\Psi(t)=\langle \cR_\rho \BP^{\alpha,\rho}_{2t} \omega, \eta\rangle$,
\begin{align*}
\Psi'(t)&=-2\langle\delta \BP^{\alpha,\rho}_{2t} \omega, \eta \rangle_\alpha=-2\langle \delta\BP^{\alpha,\rho}_{t} \omega, \BP^{\alpha,\rho}_{t}\eta,\rangle_\alpha \\ 
\Psi''(t)&=4\langle\delta (\BL_\alpha-\rho I)^{1/2} \BP^{\alpha,\rho}_{2t} \omega, \eta \rangle_\alpha=4\langle \delta\BP^{\alpha,\rho}_{t} \omega, (\BL_\alpha-\rho I)^{1/2} \BP^{\alpha,\rho}_{t}\eta,\rangle_\alpha
\\
&
=-4\left\langle \delta\BP^{\alpha,\rho}_{t} \omega, \frac{\d}{\d t} \BP^{\alpha,\rho}_{t}\eta,\right\rangle_\alpha. 
\end{align*}
\end{proof}
\end{proof}
\subsection{The  Hodge decomposition for $\Lpr{p}{r}$}\label{s: Hodge Lp}
In this subsection we prove the strong Hodge decomposition for $\Lprs{p}{r}$ for all $p\in (1,\infty)$. First we define the Sobolev space $W^{1,p}(\rdp,\mua;\Uplambda^r)$. To simplify the notation we shall write $\Lprs{p}{r}$ instead of $\Lpr{p}{r}$ and $C^\infty_c(\Uplambda^r)$ instead of $C^\infty_c(\BR^d_+;\Uplambda^r)$.
\par
For every $p\in(1,\infty)$ denote by $H_p$ the operator from $\Lprs{p}{r}$ to \break $\Lprs{p}{r+1}\times \Lprs{p}{r-1}$ defined by
$$
\omega\mapsto H_p\omega=(\delta \omega,\delta^* \omega)
$$
with domain the space $\msP(\rdp;\Uplambda^r)$ of polynomial forms. 
\begin{lem}\label{Hclos}
The operator $H_p$ is closable.
\end{lem}
\begin{proof}
Suppose that $(\omega_n)$ is a sequence in $\msP(\rdp;\Uplambda^r)$ such that $\omega_n\to 0$ in $\Lprs{p}{r}$ and $H_p\omega_n\to (\phi,\psi)$ in $\Lprs{p}{r+1}\times \Lprs{p}{r-1}$. Then
$\delta\omega_n\to \phi$ in $\Lprs{p}{r+1}$ and $\delta^* \omega_n \to \psi$ in $\Lprs{p}{r-1}$. Therefore, for $\eta\in C^\infty_c(\Uplambda^{r+1})$ and every $\zeta\in C^\infty_c(\Uplambda^{r-1})$
$$
\langle \phi,\eta\rangle_\alpha=\lim_{n\to\infty} \langle \delta\omega_n,\eta\rangle_\alpha=\lim_{n\to\infty} \langle \omega_n, \delta^*\eta\rangle_\alpha=0,
$$
and
$$
\langle \psi,\zeta\rangle_\alpha=\lim_{n\to\infty} \langle \delta^*\omega_n,\zeta\rangle_\alpha=\lim_{n\to\infty} \langle \omega_n, \delta \zeta\rangle_\alpha=0,
$$
Hence $(\phi,\psi)=(0,0)$ and $H_p$ is closable.
\end{proof}
\begin{definition}\label{W1p}
We define the Sobolev space $W^{1,p}(\Uplambda^r)=W^{1,p}(\rdp,\mua;\Uplambda^r)$ as the domain of the closure $\overline{H}_p$ of the operator $H_p$ endowed with the graph norm.
\end{definition}
If $\omega\in W^{1,p}(\Uplambda^r)$, and $\overline{H}_p
\omega=(\phi,\psi)$, with a slight abuse of notation, we shall write $
\delta\omega=\phi$ and $\delta^* \omega=\psi$. Thus, $W^{1,p}(\Uplambda^r)$ is the space of forms $\omega \in L^p(\rdp,
\mua;\Uplambda^r)$ such that 
$\delta \omega\in \Lprs{p}{r+1}$, $\delta^*\omega
\in W^{1,p}(\Uplambda^{r-1})$ and
$$
\norm{\omega}{W^{1,p}}=\norm{\omega}{L^p(\mua)}+\norm{\delta
\omega}{L^p(\mua)}+\norm{\delta^*\omega}{L^p(\mua)}.
$$
\begin{prop}\label{Wref}
For $1<p<\infty$ the space $W^{1,p}(\Uplambda^{r})$ is reflexive.
\end{prop}
\begin{proof}
The space $W^{1,p}(\Uplambda^{r})$ is isometrically isomorphic to the graph $G(\overline{H}_p)$ of the operator $\overline{H}_p$. Since $G(\overline{H}_p)$ is closed in  $\Lprs{p}{r}\times\Lprs{p}{r+1}\times\Lprs{p}{r-1}$ and the latter space is reflexive, $G(\overline{H}_p)$, and hence $W^{1,p}(\Uplambda^{r})$, is reflexive.
\end{proof}
\par
For every $p\in (1,\infty)$ denote by $\BL_{\alpha,p}$ the 
infinitesimal generator of the semigroup $\{\BT^\alpha_t: t\ge 0\}$ on 
$\Lprs{p}{r}$. Then the operators $\BL^{-1/2}
_{\alpha,p}$ and $\BL^{-1}_{\alpha,p}$ are bounded on $\Lprs{p}{r}$, 
by Corollary \ref{c: sgbdd}. 
Since for all 
$1<p,q<\infty$ the operators $\BL_{\alpha,p}$ are consistent, i.e. $
\BL_{\alpha,p}=\BL_{\alpha,q}$ on $L^p\cap\Lprs{q}{r}$,  to simplify
notation henceforth 
we shall simply write $\BL_\alpha$ instead of $\BL_{\alpha,p}$.
\begin{lem}\label{l: L-1inW}
{Suppose that $\alpha\in[-1/2,\infty)^d$}.
If $\omega\in \Lprs{p}{r}$ then $\BL_\alpha^{-1}\omega$, 
$\delta \BL_\alpha^{-1}\omega$ and 
$\delta^* \BL_\alpha^{-1}\omega$ 
are in $W^{1,p}(\Uplambda^r)$. 
Moreover
\begin{equation}\label{LaonLp}
(\delta\delta^*+\delta^*\delta)\BL_\alpha^{-1}\omega=\omega.
\end{equation}
\end{lem}
\begin{proof}
Since by Theorem \ref{t: Riesz bdd Lp} the 
Riesz transforms $\delta \BL_\alpha^{-1/2}$ and $\delta^* \BL_
\alpha^{-1/2}$ are boun\-ded on $\Lprss{p}{r}$, the operators 
$
\delta\BL_\alpha^{-1} =\delta\BL_\alpha^{-1/2}\BL_
\alpha^{-1/2}$ and $\delta^* \BL_\alpha^{-1}=\delta^*\BL_
\alpha^{-1/2}\BL_
\alpha^{-1/2}
$
are bounded from $\Lprss{p}{r}$  to $\Lprss{p}{r+1}$ and  from  $\Lprss{p}
{r}$  to $\Lprss{p}{r-1}$, and the operators
$
\delta\delta^*\BL_\alpha^{-1} =\delta\BL_\alpha^{-1/2}\delta^*\BL_
\alpha^{-1/2}$, $\delta^* \delta\BL_\alpha^{-1}=\delta^*\BL_
\alpha^{-1/2}\delta\BL_
\alpha^{-1/2}
$
are bounded on $\Lprss{p}{r}$. Thus, if $(\omega_n)$ is a sequence in 
$\msP(\rdp;\Uplambda^r)$ that converges to $\omega$ in $\Lprs{p}{r}
$, then the sequences $(\BL_\alpha^{-1}\omega_n)$, 
$(\delta \BL_\alpha^{-1}\omega_n)$, 
$(\delta^* \BL_\alpha^{-1}\omega_n)$, $(\delta \delta^*\BL_
\alpha^{-1}\omega_n)$ and
$(\delta^* \delta\BL_\alpha^{-1}\omega_n)$ converge in $\Lprs{p}{r}$. 
The conclusion follows, since $W^{1,p}(\Uplambda^r)$ is 
by definition the domain of the closure of the operator $\omega
\mapsto (\delta\omega,\delta^*\omega)$ in $\Lprs{p}{r}$.  
Finally \eqref{LaonLp}  follows from a density argument, since the identity holds on $\msP(\rdp;\Uplambda^r)$ and the operator $(\delta\delta^*+\delta^*\delta)\BL_\alpha^{-1}$ is bounded on $\Lprs{p}{r}$.
\end{proof}
The following result is the strong Hodge decomposition in $\Lprs{p}{r}$.
\begin{thm}\label{wHp}
{Suppose that $\alpha\in[-1/2,\infty)^d$}.
For every $p\in(1,\infty)$ 
$$
\Lprs{p}{r}=\delta W^{1,p}(\Uplambda^{r-1})\oplus \delta^* 
W^{1,p}(\Uplambda^{r+1}) \qquad\forall r=1,\ldots, d.
$$
Moreover the spaces 
$\delta W^{1,p}(\Uplambda^{r-1})$ and $ \delta^* 
W^{1,p}(\Uplambda^{r+1})$are closed in $\Lprs{p}{r}$.
\end{thm}
\begin{proof}
For every $\omega \in \msP(\rdp;\Uplambda^r)$
$$
\omega=\BL_\alpha \BL_\alpha^{-1} \omega =(\delta\delta^*+\delta^*
\delta) \BL_\alpha^{-1} \omega=\delta\delta^*\BL_\alpha^{-1} \omega
+\delta^*\delta \BL_\alpha^{-1} \omega.
$$
Since $\msP(\rdp;\Uplambda^r)$ is dense in $\Lprs{p}{r}$ and the 
operators $\delta\delta^*\BL_\alpha^{-1}$ and $\delta^*\delta\BL_
\alpha^{-1}$ are bounded on $\Lprs{p}{r}$, the same identity holds for 
$\omega \in \Lprs{p}{r}$. Since $\BL_\alpha^{-1} \Lprs{p}{r}\subset 
W^{1,p}(\Uplambda^r)$ by Lemma~\ref{l: L-1inW}, it holds that
$$
\Lprs{p}{r}=\delta W^{1,p}(\Uplambda^{r-1}) +  \delta^* 
W^{1,p}(\Uplambda^{r+1}).
$$
\par
 To prove that the sum is direct, observe that if we write $P=\delta
\delta^*\BL_\alpha^{-1}$ and $Q=\delta^*\delta\BL_\alpha^{-1}$, then 
$P+Q=I$ and $PQ=QP=0$, since these identities hold on $
\msP(\rdp,\Uplambda^r)$ and $P$ and $Q$ are bounded on $\Lprs{p}
{r}$. Moreover, if $\omega\in\delta W^{1,p}(\rdp,\mua;
\Uplambda^{r-1})\cap W^{1,p}(\Uplambda^{r+1})$ then $P
\omega= \delta\delta^* \BL_\alpha^{-1} \delta^* \psi=0$ and $Q
\omega=\delta^*\delta \BL_\alpha^{-1} \delta \phi=0$. Thus $\omega=0$. Here we have used the fact that $\delta^* \BL_\alpha^{-1} \delta^*=0$ and $\delta \BL_\alpha^{-1} \delta=0$, because these operators are bounded on $\Lprs{p}{r}$ and vanish on $\msP(\rdp,\Uplambda^r)$, which is dense in $\Lprs{p}{r}$.
\par
It remains only to show  that $\delta W^{1,p}(\Uplambda^{r-1})$ and 
$\delta^* W^{1,p}(\Uplambda^{r+1})$ are closed subspaces of $
\Lprs{p}{r}$.
If $\omega \in \overline{\delta W^{1,p}(\Uplambda^r})$, then, since 
$\msP(\rdp;\Uplambda^r)$ is dense in $W^{1,p}(\Uplambda^{r-1})$ 
and $\delta$ is continuous from $W^{1,p}(\Uplambda^{r-1})$ to $
\Lprs{p}{r}$, there 
exists a sequence $(\eta_j)$ in $\msP(\rdp,;\Uplambda^{r-1})$ such that 
$\delta \eta_j \to \omega$ in $
\Lprs{p}{r}$.
Since $\eta_j \in  \msP(\rdp, \Uplambda^{r-1})$, we can write 
$$
\eta_j = \delta \delta^* \BL_\alpha^{-1}\eta_j + \delta^* \delta \BL_
\alpha^{-1} \eta_j = \beta_j + \gamma_j.
$$
Observe that $\delta \beta_j = 0$ since $\im(\delta) \subseteq 
\ker(\delta)$; thus $ \delta \gamma_j = \delta \eta_j \in \Lprs{p}{r}$. 
Moreover,  also $\gamma_j$ is in $\Lprs{p}{r-1}$ since $\gamma_j=
\delta^* \delta \BL_\alpha^{-1} \eta_j $ and the operator $\delta^* 
\delta \BL_\alpha^{-1} $ is bounded on $\Lprs{p}{r}$.
Thus $\gamma_j $ is in $W^{1,p}(\Uplambda^{r-1})$.
Therefore $(\gamma_j)$ is a bounded sequence in $W^{1,p}
(\Uplambda^{r-1})$. Since $W^{1,p}(\Uplambda^{r-1})$ is reflexive, 
there exists a subsequence $(\gamma_{j_k})$ that  converges to 
some $\gamma \in W^{1,p}(\Uplambda^{r-1})$ in the weak topology.\\
Since the norm $\| \cdot\|_{L^p(\mua)}$ is $L^p$-weakly-lower-semi-
continuous, and the operator $\delta$ is continuous from $W^{1,p}
(\Uplambda^{r-1})$ to $\Lprs{p}{r}$ in the weak topologies of both 
spaces,
\begin{align*}
\| \delta \gamma - \omega \|_{L^p(\mua)} & 
\le \liminf_{j \to \infty}  \| \delta \gamma_j - \omega\|_{L^p(\mua)}  \\
& = \liminf_{j \to \infty} \| \delta \eta_j - \omega\|_{L^p(\mua)} = 0.
\end{align*}
This shows that $\omega$ is in $\delta W^{1,p}(\Uplambda^{r-1})$. Thus  $\delta W^{1,p}(\Uplambda^{r-1})$ is 
closed in $\Lprs{p}{r}$.\\
The proof that also $\delta^* W^{1,p}(\Uplambda^{r+1})$ is closed is similar.
\end{proof}

\subsection{The Hodge system and the de Rham equation in $L^p$}
{Throughout this subsection we assume that 
$\alpha\in[-1/2,\infty)^d$}. 
We discuss here an existence theorem of the Hodge system associated to the Laguerre operator. The operator $\delta:
\msP(\Uplambda^r)\to \Lpr{p}{r+1}$ is closable in the $L^p$-norm. Denote by $\Omega^p(\Uplambda^r)$ the domain of its closure. 
As usual, we shall abuse notation denoting by $\delta$ also the 
closure. Notice 
that  $\delta\BL_\alpha^{-1/2}\varphi=\BL_\alpha^{-1/2}\delta \varphi$ for 
every $\varphi\in \Omega^p(\Uplambda^r)$, by a density argument, 
since the operators  $\delta\BL_\alpha^{-1/2}$ and $\BL_
\alpha^{-1/2}\delta$ coincide on $\msP(\Uplambda^r)$ and are 
bounded from $\Omega^p(\Uplambda^r)$ to $\Lpr{p}{r+1}$.
\par
Similarly, if we denote by $\Omega^p_*(\Uplambda^r)$ the domain of the closure in the $L^p$-norm of the operator $\delta^*:\msP(\Uplambda^r)\to \Lpr{p}{r-1}$, then $\delta^*\BL_\alpha^{-1/2}\varphi=\BL_\alpha^{-1/2}\delta^* \varphi$ for all $\varphi\in \Omega^p_*(\Uplambda^r)$.
\begin{thm}\label{t:Hsys}
For every $p \in (1, \infty)$ and $r=1,\ldots,d$, and for all $\varphi \in \Omega^p(\Uplambda^{r+1})$, $\psi \in \Omega^p_*(\Uplambda^{r-1})$ such that
\[ \delta \varphi = 0 \quad \mbox{ and } \quad \delta^* \psi = 0,\]
there exists a unique $\omega \in W^{1,p}(\Uplambda^r)$  solving the Hodge system
\[ \delta \omega = \varphi \quad \mbox{ and } \quad \delta^* \omega = \psi.\]
Moreover, there exists a constant $C_{p,r}> 0$ such that
\[ \| \omega \|_{L^p(\mua)} \le C_{\alpha,p,r} \left( \| \varphi\|_{L^p(\mua)} + \| \psi\|_{L^p(\mua)}\right).\]
\end{thm}
\begin{proof}
Since the operator $\BL_{\alpha}$ is invertible on $\Lprs{p}{r}$, we may define \[ \omega = \delta^* \BL_{\alpha}^{-1}\varphi + \delta \BL_\alpha^{-1}\psi.\]
Then, $\omega \in W^{1,p}(\Uplambda^r)$ by Lemma \ref{l: L-1inW} and, since $\delta\BL_\alpha^{-1}\varphi=\BL_\alpha^{-1} \delta\varphi=0$,
\begin{align*}
\delta \omega & = \delta\delta^* \BL_{\alpha}^{-1}\varphi \\
& = (\delta \delta^* + \delta^* \delta)\BL_{\alpha}^{-1}  \varphi  \\
& = \BL_{\alpha}^{-1} \BL_{\alpha} \varphi  \\
& = \varphi,
\end{align*}
In a similar way one can  show that $\delta^* \omega = \psi$.
Moreover, by the $L^p$-boundedness of the Riesz transforms and of $\BL_\alpha^{-1/2}$ (see Corollary \ref{c: sgbdd}),
\begin{align*}
\| \omega \|_{L^p(\mua)} &  \le \| \delta^*\BL_{\alpha}^{-\frac{1}{2}} \|_{p-p}\| \BL_{\alpha}^{-\frac{1}{2}}\varphi \|_{L^p(\mua)} +\| \delta\BL_{\alpha}^{-\frac{1}{2}} \|_{p-p}\| \BL_{\alpha}^{-\frac{1}{2}}\psi \|_{L^p(\mua)}   \\
& = C_{\alpha,p,r} (\| \varphi \|_{L^p(\mua)} + \| \psi\|_{L^p(\mua)} ) .
\end{align*}
This shows also that the solution is unique.
\end{proof}
As a last application of the Bilinear Embedding Theorem, we give an 
existence theorem of the de Rham equation. 
\begin{thm} For every $p \in (1, \infty)$ and $r=1,\ldots,d$, and for all $\varphi \in \Omega^p(\Uplambda^{r})$ such that $\delta \varphi = 0$, there exists $\omega \in W^{1,p}(\Uplambda^{r-1})$ solving the de Rham equation
\[ \delta \omega = \varphi,\]
and satisfying the estimate
\[ \| \omega \|_{L^p(\mua)} \le C_{\alpha,p,r} \| \varphi\|_{L^p(\mua)}.\]
\end{thm}
\begin{proof}
{It suffices to apply Theorem \ref{t:Hsys} with $\psi=0$.}\end{proof}
\section{Bellman function}
\label{chbell}
The Bellman function technique was introduced in harmonic analysis by Nazarov, Treil and Volberg in \cite{NTV:JAMS}. 
We recall here the definition and the basic properties of the particular Bellman function used by A. Carbonaro and O. Dragi{\v{c}}evi{{\'c}} in \cite{CD:JFA} to prove the boundedness of Riesz transforms on Riemannian manifolds. Even though the results coincide with those in \cite{CD:JFA} we have included full proofs for completeness.
\par
Assume that $p \ge 2$ and let $q = \frac{p}{p-1}$ be the conjugate 
exponent of $p$; moreover set $\gamma = \frac{q(q-1)}{8}$.
We define the function $\beta: \overline{\BR}_+ \times \overline{\BR}
_+ \rightarrow \overline{\BR}_+$ by
\begin{equation}
\label{beta}
\beta(u,v) = u^p + v^q + \gamma
\begin{cases}
u^2 v^{2-q} & \mbox{ if } u^p \le v^q\\
\frac{2}{p} u^p + \left(\frac{2}{q} -1\right) v^q & \mbox{ if } u^p > v^q.
\end{cases}
\end{equation}

The particular \emph{Bellman function} we are going to use is the map
\[ Q: \mathbb{R}^{m} \times \mathbb{R}^n\rightarrow \BR\]
defined by
\[ Q(\xi,\eta) = \frac{1}{2}\beta(|\xi|, |\eta|).\]
This function is an adaptation of the one introducted by Nazarov and Treil in 
\cite{NT:AA}. The proof of the following lemma is straightforward.
\begin{lem}\label{betareg}
The function $\beta$ is $C^1$ on its domain and it is $C^2$ except on the set $\set{(u,v): u^p=v^q \  \textrm{or} \ v=0}$. Moreover, for every $u,v\ge 0$
\begin{itemize}
\item[\rmi] $0\le\beta(u,v)\le (1+\gamma)(u^p+v^q)$;
\item[\rmii] there exists a constant $C$ 
such that 
$$
0\le \partial_u\beta(u,v)\le \ C\  p\ \max\set{u^{p-1},v}\quad and \quad   0\le \partial_v\beta(u,v)\le \ C\ v^{q-1}.
$$
\end{itemize} 
\end{lem}
If $\zeta=(\xi,\eta)$ and $z=(x,y)\in \BR^m\times\BR^n$ we denote by 
$H_Q(z)$ the Hessian matrix of $Q$ at $\zeta$ and by 
$$
H_Q(\zeta;z)=\langle H_Q(\zeta) z,z \rangle_{\BR^m\times\BR^n}
$$
the corresponding Hessian form.
\begin{prop}
\label{bellreg}
The function $Q$ is in $C^1(\mathbb{R}^m \times \mathbb{R}^n)$, and it is in $C^2$ everywhere in $\mathbb{R}^m \times \mathbb{R}^n$ except on the set
\[ \Upsilon = \left\{ (\xi,\eta) \in \mathbb{R} \times \mathbb{R}^d : \eta= 0 \mbox{ or } |\xi|^p = |\eta|^q \right\}. \]\par
If $\zeta=(\xi,\eta)\in \BR^m\times\BR^n\setminus \Upsilon$ then there exists $\tau=\tau(|\xi|,|\eta|)$ such that 
$$
H_Q(\zeta;z)\ge \frac{\gamma}{2} \left(\tau |x|^2+\tau^{-1}|y|^2\right) \qquad\forall z=(x,y)\in\BR^m\times\BR^n.
$$
\end{prop}
\begin{proof}
By  the chain rule the regularity properties of $Q$ follow from those of $\beta$ in Lemma \ref{betareg} and the fact that since $p\ge2$ and $q>1$ the function $\xi\mapsto |\xi|^p$ is in $C^2$ everywhere and $\eta\mapsto |\eta|^q$ is in $C^1$ everywhere and in $C^2$ in $\BR^m\times\BR^n\setminus \Upsilon$. \par
It remains to prove the estimate of the Hessian form.  We observe that $H_Q(\zeta;z)$ is the sum of  three forms, i.e.
$$
H_Q(\zeta;z)=\sum_{i,j=1}^m \partial_{\xi_i\xi_j} Q(z) x_ix_j+2\sum_{i=1}^m\sum_{j=1}^n \partial_{\xi_i\eta_j} Q(z) x_i y_j+\sum_{i,j=1}^n \partial_{\eta_i\eta_j} Q(z) y_iy_j
$$
that we must estimate in each of the two regions 
$$
R_1=\set{(\xi,\eta): |\xi|^p< |\eta|^q, \eta\not=0}\qquad \textrm{and}\qquad R_2=\set{(\xi,\eta): |\xi|^p> |\eta|^q, \eta\not=0}.
$$
First we compute the derivatives of $Q$ in $R_1$. Since in $R_1$
$$
Q(\xi,\eta)=\frac{1}{2}\big(|\xi|^p+|\eta|^q+\gamma |\xi|^2\,|\eta|^{2-q}\big),
$$
we have that
\begin{align*}
\partial^2_{\xi_i \xi_j}Q(\zeta) &= \frac{1}{2} \big\{ p(p-2) |\xi|^{p-4} \xi_i \xi_j + \big(p|\xi|^{p-2}+2\gamma|\eta|^{2-q}\big)\delta_{ij}\big\}\\
\partial^2_{\xi_i \eta_j}Q(\zeta) &=\gamma (2-q)|\eta|^{-q}\xi_i \eta_j\\
\partial^2_{\eta_i \eta_j} Q(\xi, \eta)& =\frac{q}{2} |\eta|^{q-2}
\left\{(q-2)|\eta|^{-2}\eta_i\eta_j +\delta_{ij}\right\}  \\
&\phantom{xxxxx}
+\frac{\gamma}{2}(2-q)|\xi|^2\,|\eta|^{-q}\left\{-q|\eta|^{-2}\eta_i\eta_j+\delta_{ij}\right\}.
\end{align*}
Thus, in $R_1$
\begin{align*}
\sum_{i,j=1}^{m}\partial^2_{\xi_i \xi_j}Q(\zeta)x_i x_j &=\frac{1}{2}p(p-2) |\xi|^{p-4} \langle \xi,x\rangle^2+	 \frac{1}{2}\left(p |\xi|^{p-2} +2 \gamma |\eta|^{2-q}\right){|x|^2} \\
& \ge \gamma   |\eta|^{2-q}{|x|^2} .
\end{align*}
Next, we have that 
\begin{align*}
2 \sum_{i=1}^{m} \sum_{j=1}^{n}\partial^2_{\xi_i \eta_j}{Q}(\zeta) x_i y_i &= 2 \gamma (2-q)|\eta|^{-q} \langle \xi,x\rangle \langle \eta,y\rangle\\
& \ge - 2\gamma (2-q)|\eta|^{-q}  \,|\xi| |x| \, |\eta| |y| \\
& \ge - 2\gamma |x|   |y| \\
& \ge - \gamma \left(\frac{|\eta|^{2-q}|x|^2}{2}+2|\eta|^{q-2}   |y|^2 \right);
\end{align*}
where, in the third inequality, we have used the fact that $|\xi|\,|\eta|^{1-q}\le1$ in $R_1$. Finally, recalling that $\gamma=q(q-1)/8$,
\begin{align*}
\sum_{i,j=1}^d \partial^2_{\eta_i \eta_j} Q(\zeta) y_i y_j & =
\frac{q}{2} |\eta|^{q-2}
\left\{(q-2)|\eta|^{-2}\langle \eta,y\rangle^2+|y|^2\right\}  \\
&\phantom{xxxxx}
+\frac{\gamma}{2}(2-q)|\xi|^2\,|\eta|^{-q}\left\{-q|\eta|^{-2}\langle \eta,y\rangle^2+|y|^2\right\}\\
&\ge \frac{\gamma}{2} \left\{ 8|\eta|^{q-2} +(2-q)(1-q)|\xi|^2\,|\eta|^{-q}\right\}|y|^2\\
&\ge\frac{\gamma}{2}\big[8+(2-q)(1-q)\big]|\eta|^{q-2}|y|^2
\end{align*}
where, in the second inequality, we have used the fact that $|\xi|^2\,|\eta|^{-q}\le|\eta|^{2-q}$ in $R_1$.
Combining these estimates of the three forms, we obtain that 
\begin{align*}
H_Q(\zeta;z)&\ge \frac{\gamma}{2}\left(|\eta|^{2-q}|x|^2+(q^2-3q+6)|\eta|^{q-2}|y|^2\right)  \\
&\ge \frac{\gamma}{2}\left(|\eta|^{2-q}|x|^2+|\eta|^{q-2}|y|^2\right) \\
&\ge \frac{\gamma}{2}\left(\tau |x|^2+\tau^{-1}|y|^2\right),
\end{align*}
with $\tau=|\eta|^{q-2}$.\par
Next, we estimate the Hessian form of $Q$ in the region $R_2$. Since in $R_2$
$$ 
{Q}(\zeta) = \frac{1}{2} \left[ |\xi|^p + |\eta|^q + \gamma \left( \frac{2}{p} |\xi|^p + \left( \frac{2}{q}-1\right)|\eta|^q\right)\right],
$$
the second derivatives of $Q$ are:
\begin{align*}
\partial^2_{\xi_i \xi_j} Q(\zeta) & = \frac{1}{2}(p+2\gamma) |\xi|^{p-2} \left[ (p-2) \frac{\xi_i \xi_j}{|\xi|^2} +  \delta_{ij}\right]\\
\partial^2_{\xi_i \eta_j} Q(\zeta) &=0\\
\partial^2_{\eta_i \eta_j} Q(\zeta) & = \frac{1}{2}(q + \gamma\left(2 -q\right)) |\eta|^{q-2}\left[ (q-2)\frac{\eta_i \eta_j}{|\eta|^2}+ \delta_{ij}\right].
\end{align*}
Hence
\begin{align*}
H_Q(\zeta;z)&=\sum_{i,j=1}^{m} \partial_{\xi_i\xi_j} Q(\zeta) x_ix_j+ \sum_{i,j=1}^{n} \partial_{\eta_i\eta_j} Q(\zeta) y_iy_j       \\ 
&\ge \frac{p+2\gamma}{2}\big[(p-2)|\xi|^{p-2}|\xi|^{-2}\langle \xi, x\rangle^2+|x|^2\big] \\ 
&\phantom{xxxx}+ \frac{\big(q+\gamma(2-q)\big)}{2}|\eta|^{q-2}\big[ (q-2) |\eta|^{-2}
\langle \eta,y\rangle^2+|y|^2\big] \\
&\ge \frac{1}{2}\Big[(p-1)|\xi|^{p-2}\,|x|^2+(q-1)\,|\eta|^{q-2}\, |y|^2  \Big]   \\
&\ge  \frac{1}{2}\Big[(p-1)|\xi|^{p-2}\,|x|^2+\frac{|\xi|^{2-p}}{p-1}\, |y|^2
\Big]    \\
&\ge  \frac{1}{2} (\tau|x|^2+\tau^{-1}|y|^2),
\end{align*}
with $\tau= (p-1)|\xi|^{p-2}$. Here in the second inequality we have used the facts that $p+2\gamma\ge 1$ and $q+\gamma(2-q)\ge 1$ and in the third  inequality we have used the identity $q-1=(p-1)^{-1}$ and the fact that $|\eta|^{q-2}\ge |\xi|^{p-2}$ in $R_2$. 
\end{proof}
The Bellman function $Q$ fails to be of class $C^2$ in all of 
$\BR^{m}\times\BR^{n}$ because the second derivatives are 
discontinuous on $|\xi|^p=|\eta|^q$. Since, for our purposes, we need 
to work with a  Bellman function of class $C^2$ everywhere, we must 
replace the function $Q$ by a regularised version that retains its 
essential properties.

To regularise $Q$ we apply the  standard technique of convolving  with a mollifier.
Let $\mc{B}_1$ be the open ball in $\Bb{R}^{m+n}$ with radius of length $1$, and set
\[ \phi(\zeta) = c \  e^{- \frac{1}{1-|\zeta|^2}} \chi_{\mc{B}_1}(\zeta),\]
where $c$ is the normalization constant chosen in such a way that
\[ \int_{\Bb{R}^{m+n}} \phi(\zeta) d\zeta = 1.\]
For each $\sigma > 0$ we introduce the mollifier on $\Bb{R}^{m+n}$
\[ \phi_{\sigma}(\zeta) = \frac{1}{\sigma^{m+n}} \phi \left( \frac{\zeta}{\sigma}\right), \]
and we define the regularized version of the Bellman function $Q$:
\[ Q_\sigma (\zeta)  = Q_\sigma(\xi, \eta)= \phi_\sigma \star Q (\zeta), \]
where $\star$ denotes the convolution in $\Bb{R}^{m+n}$. Since both $Q$ and $\phi_\sigma$ are separately radial in $\xi$ and $\eta$ , for each $\sigma > 0$ there exists a function 
\[ \beta_\sigma: \rp \times \rp \rightarrow \rp\]
such that for all $(\xi, \eta) \in \BR^{m}\times\BR^{n}$
\[ Q_\sigma(\xi, \eta) = \frac{1}{2} \beta_\sigma(|\xi|, |\eta|).\]
\begin{prop}
\label{propreg}
If $0 < \sigma < 1$, then $Q_\sigma \in C^{\infty}(\Bb{R}^{m} \times \BR^{n})$. Moreover for every $(u,v) \in \rp \times \rp$, the following assertions hold:
\begin{enumerate}[(i')]
\item $0 \le \beta_\sigma(u,v) \le (1+\gamma)\left[ (u+\sigma)^p + (v+\sigma)^q\right]$;
\item there exists a constant $C$ such that for every $u, v > 0$
\begin{align*} 
0 & \le \partial_u \beta_\sigma(u,v) \le C\ p\  \max \left\{ (u+\sigma)^{p-1}, v+\sigma\right\}, \\
0 &  \le \partial_v \beta_\sigma (u,v) \le C (v+\sigma)^{q-1}.\\
\end{align*}
\item for all $\zeta = (\xi, \eta) \in \left(\mathbb{R} \times \mathbb{R}^d\right) $ there exists $\tau_\sigma = \tau_\sigma \left( |\xi|, |\eta|\right) > 0$ such that
\[ H_{Q_\sigma}(\zeta; \omega) \ge \frac{1}{2}\gamma \left( \tau_\sigma |\rho|^2 + \tau_\sigma^{-1} |\psi|^2 \right)\]
whenever $\omega = (\rho, \psi) \in \Bb{R} \times \Bb{R}^d$;
\end{enumerate}
\end{prop}
\begin{proof}

The estimate of $\beta_\sigma$ derives from its definition and the properties of the Bellman function $Q$. Indeed, setting $u = |\xi|, v = |\eta|$ for some $(\xi, \eta) \in \Bb{R} \times \rd$,
\[ \beta_\sigma(u,v)   = \beta_\sigma(|\xi|, |\eta|) = 2 Q_\sigma(\xi, \eta) \ge 0,\]
and
\begin{align*}
\beta_\sigma(u,v) &  = 2 \phi_\sigma \star Q (\xi, \eta) = \\
& = 2 \int_{\Bb{R}^{m+n}} \phi_\sigma(\xi', \eta') Q(\xi - \xi', \eta-\eta') \d\xi' \d\eta' \le \\
& \le  (1+\gamma) \int_{\Bb{R}^{m+n}} \phi_\sigma(\xi', \eta')  \left( |\xi - \xi'|^p + |\eta - \eta'|^q \right)\d\xi' \d\eta'.
\end{align*}
Since    $|\xi'| < \sigma$ and $|\eta'| < \sigma$ in the support of $\phi_\sigma$, it follows that
\begin{align*}
\beta_\sigma(u,v) &  \le  
 (1+\gamma) \left[ (|\xi|+\sigma)^p + (|\eta|+\sigma)^q \right]  \int_{\Bb{R}^{m+n}}\phi_\sigma(\xi', \eta')d\xi' d\eta'  \\
& = (1+\gamma) \left[ (u+\sigma)^p + (v+\sigma)^q \right].
\end{align*}

Next we prove (ii'). With the change of variables $u = |\xi|$, $v = |\eta|$ we get
$$
\partial_u\beta_\sigma(u,v)=\partial_{\xi_1}Q_\sigma(\xi,\eta)\frac{\xi_1}{|\xi|},\qquad \partial_v\beta_\sigma(u,v)=\partial_{\eta_1}Q_\sigma(\xi,\eta)\frac{\eta_1}{|\eta|}.
$$
Therefore it suffices to show that $\partial_{\xi_1}Q_\sigma(\xi,\eta)\ge0$ for $\xi_1>0$, $\partial_{\eta_1}Q_\sigma(\eta,\eta)\ge0$ for $\eta_1>0$ and appropriate upper estimates for $|\partial_{\xi_1}Q_\sigma(\xi,\eta)|$ and $|\partial_{\eta_1}Q_\sigma(\xi,\eta)|$. 
We prove only the estimates of $\partial_{\xi_1}Q$, since the proof of those of $\partial_{\eta_1}Q_\sigma$ are similar. Write $\xi=(\xi_1,\hat\xi)$ where $\hat\xi=(\xi_2,\ldots,\xi_1)$ and set, for $t\in\BR$, $\hat\xi,\hat\xi',\eta,\eta'$ fixed
$$
f(t)=\partial_{\xi_1} Q(t,\hat\xi-\hat\xi',\eta-\eta'), \qquad g(t)=\phi_\sigma(t,\hat\xi',\eta').
$$
The function $f$ is odd and nonnegative, $g$ is even, nonnegative and decreasing on $[0,\infty)$. Thus, for $t>0$
$$
f\star g(t)=\int_0^t f(s)\,\big[g(t-s)-g(t+s)\big]\d s+\int_t^\infty f(s)\,\big[g(s-t)-g(s+t)\big]\d s \ge 0.
$$
Hence, for $\xi_1>0$
$$
\partial_{\xi_1}Q_\sigma(\xi,\eta)=\int_{\BR^{m-1}}\int_{\BR^{n}} f\star g(\xi_1)\d\hat\xi\d\eta\ge 0.
$$
This proves that $\partial_u\beta_\sigma(u,v)\ge0$.\par
To prove the upper estimate of $\partial_u\beta_\sigma(u,v)$ we observe that
\begin{align*}
\partial_u \beta_\sigma(u,v) & = 2 |\partial_{\xi_1} Q_\sigma (\xi, \eta)| \\
& = 2 \int_{\Bb{R}^{m+n}} \phi_\sigma(\xi', \eta') |\partial_{\xi_1} Q(\xi - \xi', \eta - \eta')| \d\xi' \d\eta' \\
& \le  C \int_{\Bb{R}^{m+n}} \phi_\sigma(\xi', \eta') \max\left\{ |\xi - \xi'|^{p-1}, |\eta - \eta'| \right\} d\xi' d\eta' \le \\
& \le  C  \max\left\{ (u+\sigma)^{p-1}, v+\sigma \right\}  \int_{\Bb{R}^{m+n}} \phi_\sigma(\xi', \eta')d\xi' d\eta' \le \\
& \le  C  \max\left\{ (u+\sigma)^{p-1}, v+\sigma \right\},
\end{align*}
where again we have used the fact that $\phi$ is supported in $\mc{B}_1$, and that $p \ge 2$.
\par
The proof of the inequalities for $\partial_v \beta_\sigma(u,v) $ is analogous.\par

Finally, we prove (iii').
Since the second order derivatives of $Q$ are locally integrable 
\[ H_{Q_\sigma}(\zeta; z) = \int_{\Bb{R}^{m+n}} H_Q(\zeta- \zeta'; z) \ \phi_\sigma(\zeta') \d\zeta',\]
is well defined for each $\zeta = (\xi, \eta)$ and $z = (x,y)$ in $\Bb{R}^{m} \times \BR^{n}$.
Therefore, by Proposition \ref{bellreg}, there exists $\tau = \tau(|\xi-\xi'|, |\eta-\eta'|) > 0$ such that
\begin{align*}
H_{Q_\sigma}(\zeta; z) &\ge  \frac{1}{2}\gamma \int_{\Bb{R}^{m+n}} (\tau |x|^2 + \tau^{-1} |y|^2) \phi_\sigma(\zeta') d\zeta' = \\
& = \frac{1}{2}\gamma \left((\tau \star \phi_\sigma)(\zeta)|x|^2 + (\tau^{-1} \star \phi_\sigma)(\zeta) |y|^2 \right).
\end{align*}
By H\"older's inequality 
\begin{align*}
(\tau \star \phi_\sigma)&(\zeta)\  (\tau^{-1} \star \phi_\sigma) (\zeta) =\\
& =  \int_{\Bb{R}^{m+n}} {\tau(\zeta') \phi_\sigma(\zeta-\zeta')\d\zeta'}  \int_{\Bb{R}^{m+n}} {\tau^{-1}(\zeta') \phi_\sigma(\zeta-\zeta') \d\zeta'} \\
& \ge  \left(\int_{\Bb{R}^{m+n}} \sqrt{ \tau(\zeta') \phi_\sigma(\zeta-\zeta')}\ \sqrt{\tau^{-1}(\zeta') \phi_\sigma(\zeta-\zeta') }\d\zeta'\right)^2 = 1.
\end{align*}
Thus 
\[ (\tau^{-1} \star \phi_\sigma) (\zeta) \ge (\tau \star \phi_\sigma)^{-1} (\zeta).\]
Hence, if we define 
\[ \tau_\sigma = \tau_\sigma(|\zeta|) = (\tau \star \phi_\sigma)(\zeta), \]
we obtain the desired estimate.
\end{proof}

In the last part of this section we define the Bellman function on differential forms on $\rdp$ and we prove a technical result that will be used in the proof of the Bilinear embedding Theorem.\par 
For each $s=1,\ldots,d$ we set $d_s=\dim(\Uplambda^s(\BR^d))=\binom{d}{s}$ and identify $\Uplambda^s(\BR^d)$ with $\BR^{d_s}$ via the map $\xi\mapsto (\xi_I)$ that associates to a $s$-form the vector of its components, in some fixed order (for instance the  lexicographic order on the set of indices $\cI_s$). Define the function $
{Q_\sigma}:\Uplambda^r(\BR^d)\times \Uplambda^{r+1}(\BR^d)\to[0,\infty)
$, by
$$
Q_\sigma(\xi,\theta)=\frac{1}{2}\beta_\sigma(|{\xi}|,|{\theta}|)
$$\par
If $\zeta=(\xi,\theta)\in C^\infty(\Uplambda^r(\rdp\times\BR_+))\times C^\infty(\Uplambda^{r+1}(\rdp\times \BR_+))$, for each  $i=1,\ldots,d+1$ denote by $\delta_i  {\zeta}$ the vector in $\BR^{d_r}\times\BR^{d_{r+1}}$, whose components are
\begin{align*}
\delta_i  {\zeta_I}&=\delta_i  {\xi_I}\qquad  I\in\cI_r\\
\delta_i  {\zeta_J}&= \delta_i {\theta_J}\qquad  J\in \cI_{r+1}.
\end{align*}
Here, as before, $\delta_i$ denotes the Laguerre derivative $\sqrt{x_i}\partial_i$ for $i=1,\ldots,d$, while $\delta_{d+1}=\partial_t$ is the classical derivative.
\par
Define the operator $M_\alpha$ acting on $r$-forms by
$$
M_\alpha\,\omega(x)=\sum_{I\in\cI_r} M_{\alpha,I}\,\omega_I(x) \d x_I,
$$
where $M_{\alpha,I}$ is the operator of multiplication defined in \eqref{MI}. \begin{remark}\label{r:M>}
{If $\alpha\in [-1/2,\infty)^d$, then by \eqref{e: MI>r/2}}
\begin{equation*}
\langle M_\alpha\,\omega(x),\omega(x)\rangle =\sum_{I\in\cI_r}\omega_I(x)\,M_{\alpha,I}\,\omega_I(x)\ge \frac{r}{2}\  |\omega(x)|^2.
\end{equation*}
\end{remark}
\begin{lem}\label{l: Bochner 2}
For every smooth $r$-form $\omega$ 
$$
\sum_I \omega_I\cL_\alpha \omega_I=\langle \BL_\alpha\omega,\omega\rangle -\langle M_\alpha\,\omega,\omega\rangle.
$$
\end{lem}
\begin{proof}
A straightforward application of the identity $\cL_\alpha=\cL_{\alpha,I}-M_{\alpha,I}$ (see \eqref{rel LL}), shows that
\begin{align*}
\sum_I \omega_I\cL_\alpha \omega_I=&\sum_I \omega_I \cL_{\alpha,I}\omega_I-\sum_I \omega_I M_{\alpha,I} \,\omega_I  \\
=&\langle \BL_\alpha \omega, \omega\rangle_{\Uplambda^r}-\langle M_\alpha\, \omega,\omega\rangle_{\Uplambda^r}.
\end{align*}
\end{proof}

Define the differential operators on $\rdp\times\BR_+$
$$
\cD_\alpha=\cL_\alpha-\partial^2_{tt}\qquad \textrm{and}\qquad \BD_\alpha=\BL_\alpha-\partial^2_{tt}.
$$
\begin{lem}
\label{l: lemtecforms}
Suppose that $\zeta= 
(\xi, \theta) \in C^\infty(\Uplambda^r(\rdp\times\BR_+))\times C^\infty(\Uplambda^{r+1}(\rdp\times\BR_+))$. Then for every $  \in \rdp \times \Bb{R}_+$ and $\sigma \in (0,1)$ we have
\begin{align*}
- \cD_\alpha Q_\sigma(\zeta)  = \sum_{i=1}^{d+1} H_{ {Q}_\sigma}( {\zeta}; \delta_i  {\zeta}) & + \frac{\partial_1 \beta_\sigma(|\xi|, |\theta|)}{2 |\xi|}   \langle (M_\alpha-\BD_\alpha)\xi, \xi\rangle+ \\
& + \frac{\partial_2 \beta_\sigma(|\xi|, |\theta|)}{2 |\theta|} \langle (M_\alpha-\BD_\alpha)\theta, \theta\rangle.
\end{align*} 
\end{lem}
\begin{proof}
First note that
\begin{align}\label{e: -cDa}
-\cD_\alpha Q_\sigma (\zeta) 
= \sum_{i=1}^d \big(\delta_{ii}^2  Q_\sigma ( {\zeta})  + \psi_i \delta_i  Q_\sigma ( {\zeta})\big) + \partial^2_{tt}  Q_\sigma ( {\zeta}).
\end{align}
Since for $i=1,\ldots,d+1$
\begin{align*}
\delta_i  Q_\sigma ( {\zeta}) 
& = \sum_{I\in\cI_r} \partial_{\xi_I}  Q_\sigma ( {\zeta})\  \delta_i {\xi}_I  + \sum_{J\in\cI_{r+1}}  \partial_{\theta_J} Q_\sigma ( {\zeta})\  \delta_i  {\theta}_l,\\
\delta^2_{ii}  Q_\sigma ( {\zeta}) & =\sum_{I_1\in\cI_r} \sum_ {I_2\in\cI_r} \partial^2_{\xi_{I_1}\xi_{I_2}}  Q_\sigma ( {\zeta})  \delta_i {\xi}_{I_1}  \delta_i {\xi}_{I_2} + \sum_{I\in\cI_r} \partial_{\xi_I}  Q_\sigma ( {\zeta})\  \delta^2_{ii} {\xi}_I \ +\\
& + 2 \sum_{I\in\cI_r} \sum_{J\in\cI_{r+1}} \partial^2_{\xi_I\theta_J}  Q_\sigma ( {\zeta})  \delta_i {\xi}_I  \delta_i {\theta}_{J}\  +\\
& + \sum_{J_1\in\cI_{r+1}} \sum_{J_2\in\cI_{r+1}}  \partial^2_{\theta_{J_1}\theta_{J_2} } Q_\sigma ( {\zeta})  \delta_i {\theta}_{J_1}  \delta_i {\theta}_{J_2} + \sum_{J\in\cI_{r+1}}  \partial_{\theta_J}  Q_\sigma ( {\zeta})\  \delta^2_{ii} {\theta}_{J} \\
&=H_{ Q_\sigma}( \zeta;\delta_i \zeta)+ \sum_{I\in\cI_r} \partial_{\xi_I}  Q_\sigma ( {\zeta})\  \delta^2_{ii} {\xi}_I 
+\sum_{J\in\cI_{r+1}}  \partial_{\theta_J}  Q_\sigma ( {\zeta})\  \delta^2_{ii} {\theta}_{J},
\end{align*}
the identity \eqref{e: -cDa} can be rewritten as
\begin{align*}
- \cD_\alpha Q_\sigma(\zeta)  & = \sum_{i=1}^{d+1} H_{ Q_\sigma} 
( \zeta;\delta_i \zeta)
-\sum_{I\in\cI_r} \partial_{\xi_I} Q_\sigma( {\zeta})\ \cD_\alpha  {\xi}_I -\sum_{J\in\cI_{r+1}}  \partial_{\theta_J} Q_\sigma( {\zeta})\ \cD_\alpha  {\theta}_J.
\end{align*}
Therefore, to conclude the proof, we only need to prove that
$$
-\sum_{I\in \cI_r} \partial_{\xi_I}Q_\sigma(\zeta) \ \cD_\alpha \xi_I=\frac{\partial_1\beta_\sigma(|\xi|,|\theta|)}{2|\xi|}\langle (M_\alpha-\BD_\alpha)\xi, \xi\rangle,
$$
and
\begin{equation}\label{e: second}
-\sum_{J\in \cI_{r+1}} \partial_{\theta_J}Q_\sigma(\zeta) \ \cD_\alpha \theta_J=
\frac{\partial_2\beta_\sigma(|\xi|,|\theta|)}{2|\theta|}\langle (M_\alpha-\BD_\alpha)\theta, \theta\rangle.
\end{equation}
Indeed, since $Q_\sigma(\zeta)=\beta_\sigma(|\xi|,|\theta|)/2$, we have that
$$
\partial_{\xi_I} Q_\sigma(\zeta)=\frac{\partial_1 \beta_\sigma(|\xi|,|\theta|)}{2|\xi|} \xi_I,\qquad \partial_{\theta_J} Q_\sigma(\zeta)=\frac{\partial_2 \beta_\sigma(|\xi|,|\theta|)}{2|\theta|} \theta_J.
$$
Thus
\begin{align*}
-\sum_{I\in \cI_r} \partial_{\xi_I}Q_\sigma(\zeta) \ \cD_\alpha \xi_I&=-\frac{\partial_1\beta_\sigma(|\xi|,|\theta|)}{2|\xi|} \sum_{I\in\cI_r} \xi_I\,\cD_\alpha\xi_I \\ 
&=- \frac{\partial_1\beta_\sigma(|\xi|,|\theta|)}{2|\xi|}\left( \sum_{I\in\cI_r} \xi_I\,\cL_\alpha\xi_I-\sum_{I\in\cI_r} \xi_I\,\partial^2_{tt}\xi_I\right)\\ 
&=- \frac{\partial_1\beta_\sigma(|\xi|,|\theta|)}{2|\xi|}\big(\langle \BD_\alpha \xi,\xi\rangle-\langle M_\alpha\,\xi,\xi\rangle\big),
\end{align*}
where in the last identity we have used Lemma \ref{l: Bochner 2}. The identity \eqref{e: second} is proved similarly.
\end{proof}
\section{Proof of the Bilinear Embedding Theorem}\label{c: proofBET}
In this section we assume that {$\alpha\in[-1/2,\infty)^d$}, $p\ge 2$, $q=p/(p-1)$ and $\gamma_p=q(q-1)/8$. Define the function
$$ 
r(x) = r(x_1, \ldots, x_d) = \sqrt{x_1 + \ldots + x_d} \qquad x\in\rdp.
$$
We think of the function $r(x)$ as representing the distance of $x$ from the origin of $\rdp$. 
\par
Choose a non-increasing cut-off function $\Theta \in C^{\infty}_c([0, 
\infty))$ such that $0 \le \Theta \le 1$, $\Theta(x) = 1$ if $x \in [0,1]$ 
and $\Theta(x) = 0$ if $x \in [2, \infty)$. For $\ell>0$ define
$$
F_\ell (x) = \Theta \left( \frac{r(x)^2 }{\ell^2}\right) = \Theta \left(\frac{x_1 + 
\ldots + x_d}{\ell^2} \right)
$$
Then $\supp F_\ell   \subseteq \overline{B_{\sqrt{2}\ell}(o)}=\{x\in\rdp: r(x)\le \sqrt{2}\ell\}$.
\begin{lem}
\label{lemr}
There exists a constant $C=C(d,|\alpha|, \Theta) \ge 0$ such that
$$
|\delta F_\ell(x)|\le C/\ell \quad \textrm{and}\quad \cL_\alpha\, F_\ell(x) \le C \qquad\forall x\in\rdp \quad \forall \ell \ge 1.
$$
\end{lem}
\begin{proof}
Since
$
\cL_\alpha\, r(x) = -\sum_{i=1}^d \big(\delta^2_{ii} r(x) - \psi_i(x_i) \delta_i r(x)\big),
$
and
\begin{align*}
\delta_i r(x ) & =  \frac{ \sqrt{x_i}}{2 r(x)},\quad &\delta^2_{ii} r(x) & =\frac{r^2(x)-x_i}{2r(x)}, \quad &
\psi_i(x_i)&= \left(\frac{\alpha_i+1/2}{\sqrt{x_i}}-\sqrt{x_i}\right), 
\end{align*}
a straightforward computation shows that $|\delta r(x)|\le 1$ and, if $\ell\le r\le\sqrt{2}\ell$,
\begin{equation}\label{e: rLr}
\frac{r\cL_\alpha \, r}{\ell^2}=\frac{r}{\ell^2}\left[\frac{d+|\alpha|}{2r}-\left(\frac{1}{4 r}+\frac{r}{2}\right)\right]\ge -\frac{5}{4}.
\end{equation}
Thus
$$
|\delta F_\ell(x)|\le 2\norm{\Theta'}{\infty}\frac{r(x)}{\ell^2}|\delta r(x)|\le C,
$$
since $\ell\le r\le\sqrt{2}\ell$ on the support of $\Theta'$.\par
Next observe that, setting $g=r^2/\ell^2$,
\begin{align*}
\cL_\alpha (\Theta\circ g)&= -(\Theta''\circ g)\,|\delta g|^2+(\Theta'\circ g)\cL_\alpha g\\
&= -\Theta''(r^2/\ell^2)\frac{4r^2|\delta r|^2}{\ell^4}+\Theta'(r^2/\ell^2)\left[\frac{2r\cL_\alpha r}{\ell^2}-\frac{2|\delta r|^2}{\ell^2}\right].
\end{align*}
The desired conclusion follows, since $\Theta'$ and $\Theta''$ are bounded, $\ell\le r\le\sqrt{2}\ell$ on the support of $\Theta'$, $|\delta r|\le 1$, $\Theta'\le 0$ and $-r\cL_\alpha r/\ell^2\le C$ by \eqref{e: rLr}. 
\end{proof}
For all $s,\ell> 0$ we define
\[ \mathscr{B}_{s,\ell} = \overline{B_{2\ell}(o)} \times \left[s^{-1}, s \right],\]
where 
\[ \overline{B_{2\ell}(o)} = \left\{ x \in \rd : r(x) \le 2\ell \right\}\]
 is the closed ball centered at the origin with radius $2\ell$ with respect to the distance $r(x)$;
moreover, for  $\varepsilon > 0$ fixed, $\omega\in \mathscr{P}(\Uplambda^{r}(\rdp))$ and $\eta \in \mathscr{P}(\Uplambda^{r+1}(\rdp))$ set
$$ 
\sigma_{s,\ell} = \varepsilon \inf_{ (x,t) \in  \mathscr{B}_{s,\ell}}  \min\left\{ P^\alpha_t |\omega|(x), P^\alpha_t|\eta|(x)\right\}. 
$$
Finally, define the function
$ b_{s,\ell}: \rdp \times \rp \to \rp$ by setting
$$ 
b_{s,\ell}(x,t)  = Q_{\sigma_{s,\ell}}(\BP^{\alpha,\rho}_t \omega(x),  \BP^{\alpha,\rho}_t \eta(x)).
$$

As in \cite{DV:JFA,CD:JFA} the main step of the proof of the bilinear embedding theorem consists in estimating an integral of $\BD_\alpha\, b_{s,\ell}$. We begin with the estimate from below.
\begin{lem}\label{M-D}
For every form $\phi\in L^2(\Uplambda^m(\rdp),\mua)$ and $\rho\le m/2$
$$
\langle (M_\alpha-\BD_\alpha)\BP^{\alpha,\rho}_t\phi(x),\BP^{\alpha,\rho}_t\phi(x)\rangle\ge 0.
$$
\end{lem}
\begin{proof}
Since $\BP^{\alpha,\rho}_t=\e^{-t(\BL^\alpha-\rho I)^{1/2}}$ and $
\BD_\alpha =\BL_\alpha-\partial^2_{tt}$, we have that 
$
\BD_\alpha \BP^{\alpha,\rho}_t \phi(x)=\rho \BP^{\alpha,\rho}_t 
\phi(x)$. Thus, {since $\alpha\in[-1/2,\infty)^d$,} by Remark \ref{r:M>},
\begin{align*}
\langle (M_\alpha-\BD_\alpha)\BP^{\alpha,\rho}_t\phi(x),\BP^{\alpha,\rho}_t\phi(x)\rangle&=\langle M_\alpha\, \BP^{\alpha,\rho}_t\phi(x),\BP^{\alpha,\rho}_t\phi(x)\rangle-\rho| \BP^{\alpha,\rho}_t\phi(x)|^2 \\
&\ge \left(\frac{m}{2}-\rho\right) |\BP^{\alpha,\rho}_t\phi(x)|^2.
\end{align*}
\end{proof}
\begin{prop}
\label{p: below}
For all $(x,t)  \in \rdp \times \rp$
$$
-\cD_\alpha\, b_{s,\ell}(x,t)  \ge  \gamma_p
\  |\overline{\nabla} \BP^{\alpha,\rho}_t \omega(x)| |\overline{\nabla} \BP^{\alpha,\rho}_t \eta(x) |.
$$
\end{prop}
\begin{proof}
Set, for the sake of brevity, 
$\xi =\BP^{\alpha,\rho}_t\omega(x),\  \theta =\BP^{\alpha,\rho}_t\eta(x)$ and $\zeta =(\xi ,\theta ).
$ 
Then, by applying Lemma~ 
\ref{l: lemtecforms}  we obtain that
\begin{align*}
-\cD_\alpha\, b_{s,\ell}(x,t)  & =\sum_{i=1}^{d+1} H_{Q_{\sigma_{s,\ell}}}(\zeta ;\delta_i\zeta ) 
\\
&
+\frac{\partial_1\beta_\epsilon(|\xi |,|\theta |)}{2|\xi |}
\langle (M_\alpha-\BD_\alpha)
\xi , \xi \rangle
\\
&
+ \frac{\partial_2\beta_\epsilon(|\xi |,|\theta |)}{2|\theta |}
\langle (M_\alpha-\BD_\alpha)
\theta , \theta \rangle.
\end{align*}
Since, by Proposition \ref{propreg} (ii') the partial derivatives of $\beta_\epsilon$ are nonnegative, by Lemma \ref{M-D} and Proposition  \ref{propreg} (iii')

\begin{align*}
-\cD_\alpha\, b_{s,\ell}(x,t) &\ge \sum_{i=1}^{d+1} H_{Q_{\sigma_{s,\ell}}}(\zeta ;\delta_i\zeta ) \\
&\ge \frac{\gamma_p}{2} \left(\tau_{\sigma_{s,\ell}}\sum_{i=1}^{d+1}|\delta_i\xi|^2+
\tau^{-1}_{\sigma_{s,\ell}}\sum_{i=1}^{d+1}|\delta_i\theta|^2\right) \\
&\ge \gamma_p \ |\overline{\nabla}\xi|^2\  |\overline{\nabla}\theta|^2. 
\end{align*}
\end{proof}
To estimate $-\cD_\alpha \, b_{s,\ell}$ from above we need the following two lemmas.
\begin{lem}
\label{l: lemabove}
Suppose that $\rho\le r/2$. Then 
for every 
$(x,t) \in \mathscr{B}_{s,\ell}$
$$
b_{s,\ell}(x,t) \le C\  \frac{1+\gamma_p}{2} (1+\varepsilon)^p \big[ P^\alpha_t |\omega|^p(x) + P^\alpha_t |\eta|^q(x)\big].
$$
Moreover, there exists a constant $C = C(p,\varepsilon)$
such that for each $(x,t) \in \mathscr{B}_{s,\ell}$
\begin{align*}
|\partial_t b_{s,\ell}(x,t)| &\le C \max\big\{(P^\alpha_t|\omega|(x))^{p-1},\ P^\alpha_t|\eta|(x)\big\}\ |\partial_t \BP^{\alpha,\rho}_t \omega(x)| \\
&\phantom{.}\hskip2truecm+C\big(P^\alpha_t|\eta|(x)\big)^{q-1} \, |\partial_t \BP^{\alpha,\rho}_t \eta(x)|.
\end{align*}
\end{lem}
\begin{proof}
By Proposition \ref{propreg} 
\begin{align*}
b_{s,\ell}(x,t) & = Q_{\sigma_{s,\ell}}(\BP^{\alpha,\rho}_t\omega(x), \BP^{\alpha,\rho}_t\eta(x) ) \\
& = \frac{1}{2} \beta_{\sigma_{s,\ell}} \left( |\BP^{\alpha,\rho}_t\omega(x)|, |\BP^{\alpha,\rho}_t\eta(x)|\right)  \\
& \le  \frac{1+\gamma_p}{2} \left[ \left(|\BP^{\alpha,\rho}_t\omega(x)| + \sigma_{s,\ell} \right)^{p} + \left(|\BP^{\alpha,\rho}_t\eta(x)| + \sigma_{s,\ell} \right)^q\right].
\end{align*}
Since $\rho\le r/2$ and $\alpha\in[-1/2,\infty)^d$, we have that 
$|\BP^{\alpha,\rho}_t\omega(x)|\le P^\alpha_t|\omega|(x)$ and 
$|\BP^{\alpha,\rho}_t\eta(x)|\le P^\alpha_t|\eta|(x)$ by Proposition \ref{p: Bakry}. Thus, since 
$$
\sigma_{s,\ell} \le \varepsilon \inf_{\mathscr{B}_{s,\ell}} \min \left\{P^\alpha_t|\omega|(x), {P^\alpha_t} |\eta|(x) \right\}
$$ 
by definition, and $p\ge q$,
\begin{align*}
b_{s,\ell}(x,t) 
&\le \frac{1+\gamma_p}{2} (1+\varepsilon)^p \left( P^\alpha_t |\omega|^p(x) + P^\alpha_t |\eta|^q(x)\right).
\end{align*}
\par
To prove the second part of the statement, we apply  Propositions \ref{p: Bakry} and  \ref{propreg} once more
\begin{align*}
|\partial_t b_{s,\ell}(x,t)| & 
 \le \frac{1}{2} \left[ \partial_1  \beta_{\sigma_{s,\ell}} \left( |\BP^{\alpha,\rho}_t\omega(x)|, |\BP^{\alpha,\rho}_t\eta(x)|\right) |\partial_t \BP^{\alpha,\rho}_t\omega(x)|  \right.\\
& \qquad+  \left. \partial_2 \beta_{\sigma_{s,\ell}} \left( |\BP^{\alpha,\rho}_t\omega(x)|, |\BP^{\alpha,\rho}_t\eta(x)|\right) |\partial_t \BP^{\alpha,\rho}_t\eta(x)|  \right] \\
& \le C \max \left\{ (P^\alpha_t|\omega(x)|)^{p-1}, P^\alpha_t|\eta(x)|\right\}  |\partial_t \BP^{\alpha,\rho}_t\omega(x)|  \\
& \qquad +  C\big(P_t |\omega|(x)\big)^{q-1} |\partial_t \BP^{\alpha,\rho}_t \eta(x)|.
\end{align*}
\end{proof}
In the following we set
$$
R(x,t)=\frac{1+\gamma_p}{2} (1+\varepsilon)^p \big[P^\alpha_t |\omega|^p(x)+ P^\alpha_t |\eta|^q(x)\big].
$$
\begin{lem}\label{l: deltaR}
There exists a constant $C=C(p,\varepsilon,\omega,\eta)$ such that 
$$
\int_{\rdp} |\delta R(x,t)|\d\mua(x)\le  C\,  
\qquad\forall t>0.
$$
\end{lem}
\begin{proof}
Since $\delta P^\alpha_t=\BP^{\alpha,0}_t\delta$, by Proposition \ref{p: Bakry} 
\begin{align*}
&
\big|\delta P^\alpha_t|\omega|^p(x)\big|=\big|\BP^{\alpha,0}_t \delta|\omega|^p\big|(x)\le \, 
\,P^\alpha_t \big|\delta|\omega|^p\big|(x)  
\end{align*}
and, similarly, $\big|\delta P^\alpha_t|\eta|^q(x)\big|\le \, 
\,P^\alpha_t \big|\delta|\eta|^q\big|(x)$. Next, observe that, since the coefficients of the forms $\omega$ and $\eta$ are products of polynomials by square roots of the coordinates, and $p\ge q>1$, the functions $|\omega|^p$ and $|\eta|^q$ are differentiable and
$$
\big|\delta|\omega|^p\big|\le p\,|\omega|^{p-1}|\delta \omega|, \qquad
\big|\delta|\eta|^q\big|\le q\,|\eta|^{q-1}|\delta \eta|.
$$
Hence the functions $\big|\delta|\omega|^p\big|$ and $\big|\delta|\eta|^q\big|$ are in $L^1(\rdp,\mua)$ and, since $P^\alpha_t$ is a contraction on $L^1(\rdp,\mua)$,
$$
\int_{\rdp} |\delta R(x,t)|\d\mua(x)\le  C\,  
 \big(\norm{\delta|\omega|^p}{L^1(\mua)}+\norm{\delta|\eta|^q}{L^1(\mua)}\big).
$$
\end{proof}
We are now ready to state and prove  the estimate from above.
\begin{prop}
\label{p: above}
Suppose that $\rho\le r/2$. Then 
\begin{align*}
\limsup_{s \to \infty} \limsup_{\ell \to \infty} & \int_{s^{-1}}^s \int_{B_\ell(o)} - \cD_\alpha b_{s,\ell}(x,t) \d \mua(x) \, t\d t  \\
&\phantom{.}\qquad  \le 
\frac{1+\gamma_p}{2}(1+\varepsilon)^p \left( \normto{\omega}{L^p(\mua)}{p}+ \normto{\eta}{L^q(\mua)}{q}\right).
\end{align*}
\end{prop}
\begin{proof}
We observe that $(\supp F_\ell ) \times \left[ s^{-1}, s\right] \subseteq \mathscr{B}
_{s,l}$.
Since $-\cD_\alpha b_{s,\ell}\ge0$ by Proposition \ref{p: below},
\begin{align*}
 \int_{s^{-1}}^s \int_{B_\ell(o)} - \cD_\alpha b_{s,\ell}(x,t) \d \mua(x) \, t\d t  
\le  \int_{s^{-1}}^s \int_{\rdp} - \cD_\alpha b_{s,\ell}(x,t) F_\ell(x) \d 
\mua(x) \, t\d t.
\end{align*}
Since $\cD_\alpha = \la - \partial^2_{tt}$, to complete the proof it is 
enough to show that
\begin{align}
\limsup_{s \to \infty} &\limsup_{\ell \to \infty}   
\int_{s^{-1}}^s \int_{\rdp} \partial^2_{tt} b_{s,\ell}(x,t) F_\ell(x) \d \mua(x) \, 
t\d t\nonumber\\
&\phantom{.} \qquad \qquad\le  \frac{1+\gamma_p}{2}(1+
\varepsilon)^p  \left( \normto{\omega}{L^p(\mua)}{p}+ \normto{\eta}
{L^q(\mua)}{q}\right)\label{(i)}
\end{align}
and
\begin{align}\label{(ii)}
\lim_{\ell \to \infty}\int_{s^{-1}}^s  \int_{\rdp} \la b_{s,\ell}
(x,t) F_\ell(x) \d \mua(x) \, t\d t =0 \qquad \text{ for all } s > 0. 
\end{align}
First we prove (\ref{(i)}). Integrating by parts in the variable $t$ we get 
\begin{align*}
\int_{s^{-1}}^s  \partial^2_{tt} b_{s,\ell}(x,t) \, t\d t & = \left.t  \partial_{t} 
b_{s,\ell}(x,t)\right|_{s^{-1}}^s - \int_{s^{-1}}^s  \partial_{t} b_{s,\ell}(x,t)  \d 
t 	\\
& = s  \partial_{t} b_{s,\ell}(x,s) - s^{-1}\partial_{t} b_{s,\ell}\big(x,s^{-1}\big)+ b_{s,\ell}
\big(x,s^{-1}\big)-b_{s,\ell}(x,s) .
\end{align*}
Since $b_{s,\ell}(x,s)\ge0$, by Lemma \ref{l: lemabove} we have that
\begin{align*}
b_{s,\ell}\big(x,s^{-1}\big) -   b_{s,\ell}(x,s) & \le b_{s,\ell}\big(x,s^{-1}\big)  \\
& \frac{1+\gamma_p}{2} (1+\varepsilon)^p \left( P^{\alpha}_{s^{-1}} |\omega(x)|
^p + P^{\alpha}_{s^{-1}} |\eta(x)|^q\right).
\end{align*}
Thus we obtain the estimate
\begin{align*}
\int_{s^{-1}}^s   \int_{\rdp} &  \partial^2_{tt} b_{s,\ell}(x,t) F_\ell(x) \d 
\mua(x) \, t\d t  \\
& \le\  \frac{1+\gamma_p}{2} (1+\varepsilon)^p \int_{\rdp} \left( P^{\alpha}_{s^{-1}} |
\omega(x)|^p + P^{\alpha}_{s^{-1}} |\eta(x)|^q\right) \d \mua(x)  \\
& \qquad + \int_{B_{2\ell}(o)} \left[ s  \partial_{t} b_{s,\ell}(x,s) -  
s^{-1}\partial_{t} b_{s,\ell}\left(x,s^{-1}\right) \right] \d \mua(x) \\
& \le\  \frac{1+\gamma_p}{2} (1+\varepsilon)^p \left( \|\omega\|_{L^p(\mua)}^p  + \|
\eta\|_{L^q(\mua)}^q\right) \\
& \qquad + \norm{s  \partial_{t} b_{s,\ell}(\phantom{x},s)}{L^1(B_{2\ell}(o),\mua)}+  
\norm{s^{-1}\partial_{t} b_{s,\ell}
(\phantom{x},s^{-1})}{L^1(B_{2\ell}(o),\mua)}
\end{align*}
by the contractivity of the Poisson semigroup $P^{\alpha}_t$ on $L^r(\rdp,\mua)$ for all $r\in[1,\infty)$. Therefore, to conclude the proof of \eqref{(i)}, it is enough to show that
\begin{equation}\label{e: lim0}
\lim_{s\to0}\sup_\ell\norm{s  \partial_{t} b_{s,\ell}(\phantom{x},s)}{L^1(B_{2\ell}(o),\mua)}=\lim_{s\to\infty}\sup_\ell\norm{s  \partial_{t} b_{s,\ell}(\phantom{x},s)}{L^1(B_{2\ell}(o),\mua)}=0.
\end{equation}
Fix $v$ such that $v(q-1)>1$; then $v>2$. By Lemma \ref{l: lemabove}
and H\"older's inequality, to prove \eqref{e: lim0} it suffices to show that
\begin{equation}\label{e: boh}
\sup_{t>0} \norm{\big(P^\alpha_t |\omega|\big)^{p-1}+P^\alpha_t |\eta|+\big(P^\alpha_t|\eta|\big)^{q-1}}{L^v(\mua)}\le C(p,\omega,\eta)
\end{equation}
and
\begin{equation}\label{arriboh}
\norm{t\partial_t \BP^{\alpha,\rho}_t\omega}{L^{v'}(\mua)}+\norm{t\partial_t \BP^{\alpha,\rho}_t\eta}{L^{v'}(\mua)}\to 0\qquad \textrm {as} \quad t\to 0, \infty. 
\end{equation}
Now, \eqref{e: boh} follows from the fact that the Poisson semigroup $P^{\alpha}_t$ is a contraction on $L^r(\rdp,\mua)$ for all $r\in[1,\infty)$, whereas the spectral theorem implies that  \eqref{arriboh} holds for the norms in $L^2(\mua)$. Since $v'<2$ and the measure $\mua$ is finite it holds also for the norms in $L^{v'}(\mua)$. This concludes the proof of \eqref{(i)}.

 Next we prove (\ref{(ii)}). Integrating by parts twice we obtain
 $$
\int_{\rdp}   \cL_\alpha b_{s,\ell}(x,t) \  F_\ell(x) \d \mua(x)   = 
 \int_{\rdp} b_{s,\ell}(x,t)\   \cL_\alpha F_\ell(x) \d \mua(x).  
 $$
 Note that in the integrations by parts the boundary terms vanish, since both $x\mapsto b_{s,\ell}(x,t)$ and $F_\ell$ are smooth up to the boundary of  $\rdp$ and $F_\ell$ has compact support. Thus, since $b_{s,\ell}\ge 0$ and $\supp(\cL_\alpha F_\ell)\subset B_{\sqrt{2}\ell}(o) \setminus B_\ell(o)$, by Lemma \ref{lemr} we have that
 \begin{align*}
\int_{\rdp}   \cL_\alpha b_{s,\ell}(x,t)& \  F_\ell(x) \d \mua(x)\le C\int_{B_{l\sqrt{2}}(o) \setminus B_\ell(o)} b_{s,\ell}(x,t) \d \mua(x) \\
 & \le C  \int_{B_{\sqrt{2}\ell}(o) \setminus B_\ell(o)} \big[ P^\alpha_t |\omega|^p(x) + P^\alpha_t |\eta|^q(x)\big] \d \mua(x).
 \end{align*}
 Set
 $$
 \\
 \Xi_\ell(t)=\int_{B_{\sqrt{2}\ell}(o) \setminus B_\ell(o)} \big[ P^\alpha_t |\omega|^p(x) + P^\alpha_t |\eta|^q(x)\big] \d \mua(x).
 $$
 Then $\lim_{\ell\to\infty} \Xi_\ell(t)=0$ and $0\le \Xi_\ell(t)\le \normto{\omega}{L^p(\mua)}{p}+\normto{\eta}{L^q(\mua)}{q}$. Hence, by the Lebesgue dominated convergence theorem,
 $$
 \limsup_{\ell\to \infty} \int_{s^{-1}}^s \int_{\rdp} \cL_\alpha b_{s,\ell}(x,t)\ F_\ell(x) \d\mua(x)\ t\d t \le 0.
 $$
 Thus, to conclude the proof of \eqref{(ii)} it remains only to prove that
\begin{equation}\label{(iii)}
\liminf_{\ell\to \infty} \int_{s^{-1}}^s \int_{\rdp} \cL_\alpha b_{s,\ell}(x,t)\ F_\ell(x) \d\mua(x) \ t\d t\ge 0.
\end{equation}
By adding and subtracting to $b_{s,\ell}$  the function
$$
R(x,t)=\frac{1+\gamma_p}{2} (1+\varepsilon)^p \big[P^\alpha_t |\omega|^p(x)+ P^\alpha_t |\eta|^q(x)\big]
$$
we may write the integral in left hand side of \eqref{(iii)} as
$$
 \int_{s^{-1}}^s \int_{\rdp} \cL_\alpha \big(b_{s,\ell} -R \big)\ F_\ell \d\mua \ 
t\d t+\int_{s^{-1}}^s \int_{\rdp} \cL_\alpha R \ F_\ell\d\mua\ t\d t.
$$
Therefore, to prove \eqref{(iii)} it suffices to show that
\begin{equation}\label{(iv)}
\lim\sup_{\ell\to\infty} \int_{s^{-1}}^s \int_{\rdp} \cL_\alpha \big(R-b_{s,\ell} 
\big)\ F_\ell \d\mua \ t\d t\le0,
\end{equation}
 and
\begin{equation}\label{(v)}
\lim_{\ell\to\infty} \int_{s^{-1}}^s \int_{\rdp} \cL_\alpha R \ F_\ell \d\mua \ t
\d t=0.
\end{equation} 
To prove \eqref{(iv)} we proceed much as before. Since $R-b_{s,\ell}
\ge0$ by Lemma~\ref{l: lemabove}, integrating by parts twice and using Lemma \ref{lemr} we 
obtain that 
\begin{align*}
\int_{\rdp} \cL_\alpha\big(R-b_{s,\ell}\big)\,F_\ell\d\mua&=\int_{\rdp} \big(R-
b_{s,\ell}\big)\,\cL_\alpha F_\ell\d\mua \\ 
&\le C\int_{B_{\sqrt{2}\ell}(o) \setminus B_\ell(o)}\big(R-b_{s,\ell}\big) \d\mua  \\ 
&\le C\int_{B_{\sqrt{2}\ell}(o) \setminus B_\ell(o)}R \d\mua \\
&=C\  \Xi_\ell(t).
\end{align*}
 Hence \eqref{(iv)} follows by the Lebesgue dominated convergence theorem.
 It remains only to prove \eqref{(v)}. By integrating by  parts and using Lemmas \ref{l: lemabove} and \ref{l: deltaR}, we obtain that
\begin{align*}
\left| \int_{s^{-1}}^s\int_{\rdp} \cL_\alpha R\ F_\ell \d\mua \ t\d t\right|
&
\le  \int_{s^{-1}}^s\int_{\rdp} |\delta R|\ |\delta F_\ell| \d\mua \ t\d t 
\\
&\le C/\ell.
\end{align*} 
This concludes the proof of \eqref{(v)} and of the proposition.
\end{proof}
\noindent
\emph{Proof of the Bilinear embedding Theorem.} 
First we prove the statement for $p \ge 2$. By Propositions \ref{p: below} and \ref{p: above} and passing to the limit as $\epsilon\to 0$ we obtain
\begin{align}
\label{stimanabla}
\nonumber  \gamma_p \int_0^{\infty}&  \int_{\rdp}  |\overline{\nabla} \BP^{\alpha,\rho}_t \omega(x)| |\overline{\nabla} \BP^{\alpha,\rho}_t \eta(x)(x)| d\mua(x) t\d t \\
\le& \frac{1+\gamma_p}{2} \left(\normto{\omega}{L^p(\mua)}{p} + \normto{\eta}{L^q(\mua)}{q}\right).
\end{align}
By replacing $\omega$ with $\lambda \omega$ and $\eta$ with $\lambda^{-1} \eta$ in this inequality and minimizing for $\lambda>0$ we get
\begin{align*}
\int_0^{\infty}&  \int_{\rdp}  |\overline{\nabla} \BP^{\alpha,\rho}_t \omega(x)| |\overline{\nabla} \BP^{\alpha,\rho}_t \eta(x)(x)| d\mua(x) t\d t \\
& \le\, \frac{1+\gamma_p}{2\gamma_p} \left[ \left(\frac{p}{q}\right)^{1/p}+\left(\frac{q}{p}\right)^{1/q}\right] \norm{\omega}{L^p(\mua)} \norm{\eta}{L^q(\mua)}.
\end{align*}
Since $\gamma_p=q(q-1)/8$, we have that 
 \begin{align*}
 \frac{1+\gamma_p}{2 \gamma_p} \left[  \left( \frac{q}{p}  \right)^{ \frac{1}{q}} +  \left( \frac{p}{q}  \right)^{ \frac{1}{p}}\right] 
&=\frac{8 + q(q-1)}{2}\,(q-1)^{ \frac{1}{q}-1}\,(p-1)\\
&<5.7 \, (p-1).
 \end{align*}
 Since $p^* =  \max \left\{p,q \right\}=p$ when $p>2$, this proves the Bilinear embedding Theorem for $p \ge 2$.
The proof in the case $1<p<2$ is similar: it suffices to replace the constants $\gamma_p=q(q-1)/8$  and $(1+\epsilon)^p$ in Lemma \ref{l: lemabove}  and in Proposition \ref{p: above} by $\gamma_q=p(p-1)/8$ and $(1+\epsilon)^q$, and in the proof of \eqref{e: lim0} to fix $v$ such that $v(p-1)>1$. Proceeding as before we obtain the result also for $1<p<2$.
\qed
\section{Spectral multipliers}\label{c: specmult}
In Section \ref{s: saeHL} we have seen that the Hodge-Laguerre operator on $r$-forms has a self-adjoint extension $\BL_\alpha$ on $\Lpr{2}{r}$ with spectral resolution
$$
\BL_\alpha=\sum_{n\ge r} n\,\cP^\alpha_n,
$$
where, for each integer $n\ge r$, $\cP_n$ is the orthogonal projection onto the space spanned by the forms $\lambda^\alpha_k(x)=\sum_{I\in\cI_r} \ell^\alpha_{k,I} \d x_I$, $|k|=n$.
Thus, by the spectral theorem, if $m=(m(n))_{n\ge r}$ is a bounded sequence in $\BC$ the operator
$$
m(\BL_\alpha)=\sum_{n\ge r} m(n)\,\cP^\alpha_n
$$
is bounded on $\Lpr{2}{r}$ and $\norm{m(\BL_\alpha)}{2-2}= \sup_{n\ge r}|m_n|$. \par
In this section we shall give a sufficient condition for the 
boundedness of $m(\BL_\alpha)$ on $\Lpr{p}{r}$ 
also for $p\not=2$. 

Before stating and proving the multiplier theorem for $\BL_\alpha$, we briefly recall the universal multiplier theorem for symmetric contraction semigroups of Carbonaro and Dragi\v{c}evi\v{c} \cite{CDmult}. We denote by $H^\infty(S_\theta)$ the space of all functions that are bounded and holomorphic in the sector 
$$
S_{\theta}=\set{z\in\BC: |\arg z|<\theta}.
$$
By Fatou's theorem a function $m$ in $H^\infty(S_\theta)$ has non tangential limits almost everywhere on the boundary of $S_\theta$. We denote by $m_\pm$ the boundary values of $m$ on the edges of the sector, considered as functions on $\BR_+$, i.e.
$$
m_\pm(\lambda)=m(\lambda\,\e^{\pm i\theta}) \qquad\forall \lambda \in \BR_+.
$$
For every $r>0$ let $D_rm_\pm(\lambda)=m_\pm(r\lambda)$.
For every $J>0$ denote by $H^J(\BR)$ the usual $L^2$-Sobolev space on $\BR$. Fix a smooth function $\psi$ with compact support in $[1/4,4]$ such that $\psi=1$ on $[1/2,2]$. 
\begin{definition}\label{HMcond}
We denote by $H^\infty(S_\theta;J)$ the space of all functions $m\in H^\infty(S_\theta)$ such that 
$$
\norm{m}{S_\theta;J}=\sup_{r>0}\norm{\psi \,D_rm_+}{H^J}+\sup_{r>0}\norm{\psi \,D_rm_-}{H^J}<\infty.
$$
\end{definition}
Then $H^\infty(S_\theta;J)$ does not depend on the choice of the function $\psi$ and it is a Banach space with respect to the norm $\norm{\phantom{m}}{S_\theta;J}$. 
For every $p\in [1,\infty]$ set 
$$\phi^*_p=  \arcsin \left| 
\frac{2}{p}-1\right|.$$
\begin{thm}[Carbonaro and Dragi\v{c}evi\v{c}]
\label{cd}
For every generator $\mc{A}$ of a symmetric contraction semigroup on a $\sigma$-finite measure space $(\Omega, \nu)$, every $p \in (1, \infty)$, $J > 3/2$ and $m \in H^{\infty}(S_{\phi^*_p}; J)$, the operator $m(\mc{A})$ extends to a bounded operator on $L^p(\Omega, \nu)$, and there exists $C(p,J) > 0$ such that 
\[ \| m(\mc{A})\|_{p-p} \le C(p,J) \left( \| m\|_{H^{\infty}(S_{\phi^*_p}; J)} + |m(0)|\right).\]
\end{thm}
\begin{remark}\label{injop} It follows from the proof of Theorem 1 in \cite{CDmult} that if the operator $\mc{A}$ is injective, then the term $|m(0)|$ can be omitted in the right hand side of the estimate of $\norm{m(\mc{A})}{p-p}$. 
\end{remark}
For each $a\ge 0$ define  the translated sector 
$$
\tau_a S_{\phi^*_p}=\set{z\in\BC: \arg(z-a)<\phi^*_p}.
$$
and the space $H^\infty(\tau_aS_{\phi^*_p};J)$ as the space of all functions $m$ such that $z\mapsto\tau_a m(z)=m(z+a)\in H^\infty(S_{\phi^*_p};J)$ endowed with the norm 
$$
\norm{m}{\tau_aS_{\phi^*_p};J}=\norm{\tau_a m}{S_{\phi^*_p};J}
$$
\begin{thm}\label{mtforHL}
Suppose that $\alpha\in(-1/2,\infty)^d$. If $1\le r\le d$ and $m\in H^\infty(\tau_{r/2}S_{\phi^*_p};J)$ for some $p\in (1,2)$ and some $J>3/2$, then the operator $m(\BL_\alpha)$ is bounded on $\Lpr{q}{r}$ for all $q\in [p,p']$ and
$$
\norm{m(\BL_\alpha)}{q-q}\le\, C(p,J,r)\  \norm{m}{\tau_{r/2}S_{\phi^*_p};J}.
$$
\end{thm}
\begin{proof}
Suppose that $\omega=\sum_{I\in\cI_r} \omega_I\d x_I\in \cP(\Uplambda^r(\rdp))$. Then, by Proposition~\ref{p: diag} and the spectral theorem
\begin{align*}
m(\BL_\alpha)\,\omega&=\sum_{I\in\cI_r} m(\cL_{\alpha,I})\,\omega_I\d x_I 
\\ 
&=\sum_{I\in\cI_r} \tau_{r/2}m(\cL_{\alpha,I}-(r/2)I)\,\omega_I\d x_I . 
\end{align*}
By Corollary \ref{norm on Lp} the operators $\cL_{\alpha,I}-(r/2)I$ generate semigroups of symmetric contractions on $L^q(\rdp,\mua)$ for every $q\in[1,\infty]$. Since $\tau_{r/2}m\in H^\infty(S_{\phi^*_p};J)$ the operators $m(\cL_{\alpha,I})=\tau_{r/2}m(\cL_{\alpha,I}-(r/2)I)$ are bounded on $L^p(\rdp,\mua)$ and 
$$
\norm{m(\cL_{\alpha,I})}{p-p}\le\,C(J,p) \ \norm{m}{\tau_{r/2}S_{\phi^*_p};J}
$$
by Theorem \ref{cd}.
To conclude that $m(\BL_\alpha)$ is bounded on $\Lpr{p}{r}$ we apply a randomisation argument based on Rademacher's functions.
We recall that the Rademacher functions are an orthonormal family of $\set{-1,1}$-valued functions $\set{r_k: k\in\BN}$ in $L^2([0,1],\d t)$ such that if $F(t)=\sum_{n=0}^{\infty}a_k\,r_k(t)$ is a function in $L^2([0,1]),\d t)$ then $F\in L^p([0,1]),\d t)$ for every $p<\infty$ and
\begin{equation}\label{Ki}
A_p \norm{F}{p}\le \norm{F}{2}=\left(\sum_{k=0}^\infty |a_k|^2\right)^{1/2}\le\, B_p \norm{F}{p}
\end{equation}
for two positive constants $A_p$ and $B_p$ (see \cite[Appendix C]{GRAF:CFA}). Let $I_1,\ldots,I_{d_r}$, where $d_r=\#\cI_r$, be an enumeration of the multiindeces in $\cI_r$. Then, by applying the second inequality in \eqref{Ki} to the function
$$
F(t)=\sum_{k=1}^{d_r} m(\cL_{\alpha,I_k})\,\omega_{I_k}(x)\ r_k(t),
$$
we obtain that
\begin{align*}
|m(\BL_\alpha)\,\omega(x)|&= \left(\sum_{k=1}^{d_r}
|m(\cL_{\alpha,I_k})\,\omega_{I_k}(x)|^2\right)^{1/2} \\ 
&\le B_p  \left(\int_0^1\left|\sum_{k=1}^{d_r}m(\cL_{\alpha,I_k})\,
\omega_{I_k}(x)\,r_k(t)\right|^p\d t\right)^{1/p}.
\end{align*}
Thus, since $|r_k(t)|=1$ for every $t\in [0,1]$, an application of Fubini's theorem and Minkowski's inequality yield
\begin{align*}
\norm{m(\BL_\alpha)\,\omega}{L^p(\mua)}&=\left(\int_{\rdp}
 |m(\BL_\alpha)\,\omega(x)|^p \d \mua(x)\right)^{1/p} 
 \\ 
&\le B_p \left(\int_{\rdp}\int_0^1 \left|\sum_{k=1}^{d_r}
m(\cL_{\alpha,I_k})\,
\omega_{I_k}(x)\,r_k(t)\right|^p\d t\d \mua(x)\right)^{1/p} 
\\ 
&\le B_p \left(\int_{\rdp} \left(\sum_{k=1}^{d_r}|m(\cL_{\alpha,I_k})\,
\omega_{I_k}(x)|\right)^p\d \mua(x)\right)^{1/p} 
\\
&\le \sum_{k=1}^{d_r} \norm{m(\cL_\alpha,I_k) \omega_{I_k}}
{L^p(\mua)}\\
&\le \sum_{k=1}^{d_r} \norm{m(\cL_\alpha,I_k)}{p-p}\ 
\norm{ \omega_{I_k}}{L^p(\mua)}\\
&\le C(p,J)\ d_r \norm{\omega}{L^p(\mua)}.
\end{align*}
Hence $m(\BL_\alpha)$ is bounded on $\Lpr{p}{r}$. Let $m^*(z)=\overline{m(\overline{z}})$. Then $m^*\in H^\infty(\tau_a S_{\phi^*_p};J)$ and $\norm{m^*}{\tau_{r/2}S_{\phi^*_p};J}
=\norm{m}{\tau_{r/2}S_{\phi^*_p};J}$. Then $m^*(\BL_\alpha)$ is bounded on $\Lpr{p}{r}$ and 
$$
\norm{m(\cL_{\alpha,I})}{p-p}\le\,C(J,p) \ \norm{m}{\tau_{r/2}S_{\phi^*_p};J}.
$$
Since  $m(\BL_\alpha)^*=m^*(\BL_\alpha)$ because $\BL_\alpha$ is self-adjoint, the operator $m(\BL_\alpha)=m^*(\BL_\alpha)^*$ is also bounded on $\Lpr{p'}{r}$, $p'=p/(p-1)$ by duality. Thus it is bounded on $\Lpr{q}{r})$ for all $q\in[p,p']$ by interpolation. This concludes the proof of the theorem.
\end{proof}
\bigskip
\centerline{\textsc{Acknowledgments}}
\medskip We thank Andrea Carbonaro for helpful discussions on the Bellman function technique and Emanuela Sasso for discussing with us her work on Riesz-Laguerre transforms. We are also grateful to Adam Nowak for his thorough
and careful reading of a preliminary version of the manuscript and for pointing out a couple of
mathematical errors and a host of typos.

\end{document}